\documentclass[11pt]{amsart}

\usepackage{amsmath,amsfonts,amssymb,amsthm,amscd,epsfig,color,euscript,mathrsfs, texdraw,multicol,rotating,pifont}
\usepackage[vcentermath,enableskew]{youngtab}
\usepackage[all]{xy}
\newtheorem{theorem}{Theorem}[section]
\newtheorem{proposition}[theorem]{Proposition}
\newtheorem{corollary}[theorem]{Corollary}

\newtheorem{lemma}[theorem]{Lemma}

\theoremstyle{definition}
\newtheorem{definition}[theorem]{Definition}
\newtheorem{example}[theorem]{Example}
\newtheorem{remark}[theorem]{Remark}

\numberwithin{equation}{section}

\newcommand{\soplus}{\mathop{\mbox{\normalsize$\bigoplus$}}\limits}
\newcommand{\stens}{\mathop{\mbox{\normalsize$\bigotimes$}}\limits}
\newcommand{\CC}{{\mathscr{C}}}
\newcommand{\ko}{\mathbf{k}}

\newcommand{\dct}[2]{\overset{#2}{\underset{#1}{\conv}} }

\def\seteq{\mathbin{:=}}
\def\g{\mathfrak g}
\def\h{\mathfrak h}

\def\Z{\mathbb Z}
\def\C{\mathbb C}
\def\F{\mathcal F}
\def\A{\mathbb A}
\def\va{\varpi}
\def\Q{\mathbb Q}
\def\Hom{{\rm Hom}}
\def\Mod{{\rm Mod}}
\def\Ind{{\rm Ind}}

\def\cl{{\rm cl}}
\def\aff{{\rm aff}}
\def\dim{{\rm dim}}
\def\Rep{{\rm Rep}}
\def\Dim{\underline{\dim}}

\newcommand{\nc}{\newcommand}
\nc{\hs}{\hspace*}
\nc{\be}{\begin{enumerate}} \nc{\ee}{\end{enumerate}}
\nc{\bnum}{\be[{\rm(i)}]} \nc{\bna}{\be[{\rm(a)}]}
\nc{\Rnorm}{R^{\rm{norm}}}
\nc{\Runiv}{R^{\rm{univ}}}
\nc{\conv}{\mathbin{\mbox{\large $\circ$}}}

\nc{\tens}{\mathop\otimes}
\nc{\wl}{\mathsf{P}}
\nc{\rl}{\mathsf{Q}}
\nc{\cm}{\mathsf{A}}

\nc{\redex}{\widetilde{w}} \nc{\redez}{{\widetilde{w}_0}}
\newcommand{\Lto}{\longrightarrow}

\newcommand{\um}{\underline{m}}

\nc{\rmat}[1]{{\mathbf
r}_{\mspace{-2mu}\raisebox{-.5ex}{${\scriptstyle{#1}}$}}}

\nc{\rectangle}
{\xy
(3,-3)*{}="T"; (5,-25)*{}="B";(15,-15)*{}="C";
"T"; "C" **\dir{-}; "C"; "B" **\dir{-};
"T"; "T"+(-5,-5) **\dir{-}; "T"+(-5,-5); "T"+(-10,-10) **\dir{-};
"B"; "B"+(-5,5) **\dir{-}; "B"+(-5,5); "B"+(-12,12) **\dir{-};
"C"*{\bullet};
"B"*{\bullet};
"T"*{\bullet};
"B"+(-12,12)*{\bullet}; "B"+(-17,12)*{\scriptstyle [a',b]};
"B"+(0,-3)*{\scriptstyle [a,b]};
"C"+(5,0)*{\scriptstyle [a,b']};
"T"+(0,3)*{\scriptstyle [a',b']};
\endxy}

\nc{\Lray}
{\xy
(3,-3)*{}="T"; (5,-25)*{}="B";(15,-15)*{}="C";
"T"; "C" **\dir{-}; "C"; "B" **\dir{-};
"T"; "T"+(-5,-5) **\dir{-}; "T"+(-5,-5); "T"+(-10,-10) **\dir{.};
"B"; "B"+(-5,5) **\dir{-}; "B"+(-5,5); "B"+(-12,12) **\dir{.};
"C"*{\bullet};
"B"*{\bullet};
"T"*{\bullet};
"B"+(-12,12)*{\circ};
"B"+(0,-3)*{\scriptstyle [a,b]};
"C"+(-5,0)*{\scriptstyle [a,b']};
"T"+(0,3)*{\scriptstyle [a',b']};
\endxy}

\nc{\Rray}
{\xy
(3,-3)*{}="T"; (5,-25)*{}="B";(15,-15)*{}="C";
"T"; "T"+(5,-5) **\dir{-};
"C"; "T"+(5,-5) **\dir{.};
"C"; "C"+(-5,-5) **\dir{.};
"B"; "C"+(-5,-5) **\dir{-};
"T"; "T"+(-5,-5) **\dir{-}; "T"+(-5,-5); "T"+(-10,-10) **\dir{-};
"B"; "B"+(-5,5) **\dir{-}; "B"+(-5,5); "B"+(-12,12) **\dir{-};
"C"*{\circ};
"B"*{\bullet};
"T"*{\bullet};
"B"+(-12,12)*{\bullet};
"B"+(-7,12)*{\scriptstyle [a,b']};
"B"+(0,-3)*{\scriptstyle [a',b']};
"T"+(0,3)*{\scriptstyle [a,b]};
\endxy}

\nc{\Uray}
{\xy
(10,8)*{\overbrace{\ \ \qquad}^{\scriptstyle \ell \ge 4 }};
(0,0)*{}="T0"; (25,-5)*{}="T1";(15,-15)*{}="B";
"T0"; "B" **\dir{-}; "T1"; "B" **\dir{-};
"T1"; "T1"+(-10,10) **\dir{-};
"T1"+(-25,10); "T1"+(-5,10) **\dir{.};
"T0"; "T0"+(5,5) **\dir{-};
"T0"*{\bullet};"T1"*{\bullet};
"B"*{\bullet};
"T0"+(6,0)*{\scriptstyle [a',b']};
"T1"+(-5,0)*{\scriptstyle [a,b]};
"B"+(0,-3)*{\scriptstyle [a,b']};
\endxy}

\nc{\Dray}
{\xy
(10,-8)*{\underbrace{\ \ \qquad}_{\scriptstyle \ell \ge 4 }};
(0,0)*{}="B0"; (25,5)*{}="B1";(15,15)*{}="T";
"B0"; "T" **\dir{-}; "B1"; "T" **\dir{-};
"B1"; "B1"+(-10,-10) **\dir{-};
"B1"+(-25,-10); "B1"+(-5,-10) **\dir{.};
"B0"; "B0"+(5,-5) **\dir{-};
"B0"*{\bullet};"B1"*{\bullet};"T"*{\bullet};
"B0"+(6,0)*{\scriptstyle [a,b]};
"B1"+(-6,0)*{\scriptstyle [a',b']};
"T"+(0,3)*{\scriptstyle [a,b']};
\endxy}

\nc{\UrayG}
{\xy
(10,2)*{\underbrace{\ \ \qquad}_{\scriptstyle \ell = 2 }};
(0,0)*{}="T0"; (25,-5)*{}="T1";(15,-15)*{}="B";
"T0"; "B" **\dir{-}; "T1"; "B" **\dir{-};
"T1"; "T1"+(-10,10) **\dir{-};
"T1"+(-25,10); "T1"+(-5,10) **\dir{.};
"T0"; "T0"+(5,5) **\dir{-};
"T0"+(10,10); "T0"+(5,5) **\dir{.};
"T1"+(-15,15); "T1"+(-10,10) **\dir{.};
"T0"*{\bullet};"T1"*{\bullet};
"B"*{\bullet};
"T1"+(-15,15)*{\circ};
"T0"+(-3,0)*{\scriptstyle \beta};
"T1"+(3,0)*{\scriptstyle \alpha};
"B"+(0,-3)*{\scriptstyle \gamma};
\endxy}

\nc{\DrayG}
{\xy
(10,-2)*{\overbrace{\ \ \qquad}^{\scriptstyle \ell = 2 }};
(0,0)*{}="B0"; (25,5)*{}="B1";(15,15)*{}="T";
"B0"; "T" **\dir{-}; "B1"; "T" **\dir{-};
"B1"; "B1"+(-10,-10) **\dir{-};
"B1"+(-10,-10); "B1"+(-15,-15) **\dir{.};
"B1"+(-25,-10); "B1"+(-5,-10) **\dir{.};
"B0"; "B0"+(5,-5) **\dir{-};
"B0"+(10,-10); "B0"+(5,-5) **\dir{.};
"B0"*{\bullet};"B1"*{\bullet};"T"*{\bullet};"B0"+(10,-10)*{\circ};
"B0"+(-3,0)*{\scriptstyle \beta};
"B1"+(3,0)*{\scriptstyle \alpha};
"T"+(0,3)*{\scriptstyle \gamma};
\endxy}

\begin{document}

\title[AR quiver of type A and generalized quantum affine Schur-Weyl duality] {Auslander-Reiten quiver of type A and \\ the generalized quantum affine
Schur-Weyl duality }

\author[Se-jin Oh]{Se-jin Oh}
\address{School of Mathematics, Korea Institute for Advanced Study Seoul 130-722, Korea}
\email{sejin092@gmail.com}


\begin{abstract}
The quiver Hecke algebra $R$ can be also understood as a
generalization of the affine Hecke algebra of type $A$ in the
context of the quantum affine Schur-Weyl duality by the results of
Kang, Kashiwara and Kim. On the other hand, it is well-known that
the Auslander-Reiten(AR) quivers $\Gamma_Q$ of finite simply-laced
types have a deep relation with the positive roots systems of the
corresponding types.
In this paper, we present explicit combinatorial descriptions for the AR-quivers $\Gamma_Q$ of finite type $A$. 
Using the combinatorial descriptions, we can investigate relations
between finite dimensional module categories over the quantum affine
algebra $U'_q(A_n^{(i)})$ $(i=1,2)$ and finite dimensional graded
module categories over the quiver Hecke algebra $R_{A_n}$ associated
to $A_n$ through the generalized quantum affine Schur-Weyl duality
functor.
\end{abstract}

\maketitle



\section*{Introduction}

Let $\Delta_n$ be a rank $n$ Dynkin diagram of finite type $A$, $D$
or $E$, and $Q$ be a Dynkin quiver by orienting edges of $\Delta_n$.
Gabriel's theorem \cite{Gab80} states that isomorphism classes of
indecomposables of the category of modules ${\rm mod }\C Q$ over the
path algebra $\C Q$ are labeled by the set of positive roots
$\Phi^+_n$ related to $\Delta_n$.

The Auslander-Reiten quiver (AR-quiver) $\Gamma_Q$ arising from $Q$
encodes the structure of the ${\rm mod }\C Q$ in the following
sense: (i) vertices of AR-quiver present the indecomposables, (ii)
arrows present the irreducible morphisms between them. Moreover,
$\Gamma_Q$ itself provides reduced expressions of the longest
element $w_0$ of the Weyl group associated with $\Delta_n$ by
reading the levels of vertices in $\Gamma_Q$ in a manner compatible
with paths \cite{B99}. The AR-quiver also plays an important role in
the research area of cluster algebras, cluster-tilted algebras and
cluster categories (for example, \cite{BMMRT}).

For each reduced expression $\tilde{w}_0$ of $w_0$, we can assign a
total order $<_\redez$ on $\Phi^+_n$ by \cite[Chapter VI, \S
1.6]{Bour}. Moreover, the total order is {\it convex} in the sense
that for any $\alpha,\beta \in \Phi_n^+$ with $\alpha+\beta \in
\Phi_n^+$, we have that $\alpha <_\redez \beta$ implies $\alpha
<_\redez \alpha+\beta <_\redez \beta$ (\cite{Papi94}). By
abstracting this concept, we say a partial order $\prec$ on
$\Phi_n^+$  convex if $\alpha+\beta \in \Phi_n^+$ implies either
$\alpha \prec \alpha+\beta \prec \beta$ or $\beta \prec \alpha+\beta
\prec \alpha$ (see \S \ref{subsec: order}).

On the other hand, Hernandez and Leclerc \cite{HL11} defined certain
category $\mathscr{C}^{(1)}_Q$ which consists of finite dimensional
integrable modules over quantum affine algebras $U_q'(\g)$, and
depends on the AR-quiver $\Gamma_Q$. They proved that
$\mathscr{C}^{(1)}_Q$ categorifies $\C[N]$ the coordinate ring of
the unipotent group $N$ associated with $\g_0$. Here $\g$ is affine
Kac-Moody algebra of untwisted simply laced type and $\g_0$ is the
finite dimensional semisimple Lie subalgebra inside $\g$.

The quiver Hecke algebras $R_{\mathsf g}$, introduced by
Khovanov-Lauda \cite{KL09,KL11} and Rouquier \cite{R08}, categorify
the negative part of $U_q({\mathsf g})$ for all symmetrizable
Kac-Moody algebras ${\mathsf g}$. Since $R_{\mathsf g}$ provides
quantum versions of Lascoux-Leclerc-Thibon-Ariki theory
(\cite{Ar96,LLT}), many authors study the representation theories on
quiver Hecke algebras $R_{\mathsf g}$ very actively since its
introduction (for example,
\cite{BK09A,BK09B,HM10,KK12,KKO13,LV09,VV09}). In particular, if
${\mathsf g}$ is a finite classical semisimple Lie algebra, the
theory on the construction of simple modules over $R_{{\mathsf g}}$
is well-developed (see
\cite{BKOP,BKM12,HMM09,KP11,Kato12,KR09,Mc12}). Among these results,
the approach of \cite{Kato12,KR09,Mc12} provides the way of
constructing simple modules over $R_{{\mathsf g}}$ by using
PBW-bases which are arisen from a convex order induced by a fixed
reduced expression $\tilde{w}_0$ of $w_0$.

Note that, for $\g_0$ inside $\g$, the specialization of the
integral form $U^{-}_{\mathbb A}(\g_0)$ of $U_q^-(\g_0)$ at $q=1$ is
isomorphic to $\C[N]$ and hence $U^{-}_{\mathbb A}(\g_0)|_{q=1}$ is
categorified via $R_{\g_0}$ (with forgetting grading) and
$\mathscr{C}^{(1)}_Q$ over $U_q'(\g)$, simultaneously. Since the
result of Brundan and Kleshchev in \cite{BK09A} provides the
isomorphism between affine Hecke algebras of type $A$ and the
$R_{A_n^{(1)}}$ up to a specialization and a localization, Hernandez
and Leclerc expected that quiver Hecke algebras might be interpreted
as a generalization of affine Hecke algebra of type $A$ in the
context of the quantum affine Schur-Weyl duality
(\cite{CP96,Che,GRV94}).

In \cite{KKK13a,KKK13b}, Kang, Kashiwara and Kim came up with an
answer for the expectation. They constructed the {\it generalized
quantum affine
Schur-Weyl duality functor} 
$$  \F : \Rep(R_J) \rightarrow \CC_\g  $$
between the finite dimensional integrable modules category $\CC_\g$
over $U'_q(\g)$ and the category $\Rep(R_J)$ of finite dimensional
graded modules over quiver Hecke algebra $R_J$, by observing the
zeros of denominator of {\it normalized $R$-matrix}
$\Rnorm_{k,l}(z)$. Their work can be understood as a graded
generalization of quantum affine Schur-Weyl duality since
$\Rep(R_J)$ has a grading. Here $R_J$ is defined by the quiver
$Q_J$, which appears in the process of constructing $\mathcal{F}$.
Interestingly, $R_J$ does depend on the {\it choice} of {\it good}
modules over $U'_q(\g)$.

By the results of \cite{KKK13a,KKK13b} and \cite{HL11}, we have an
exact functor $\mathcal{F}^{(1)}_Q$ from $\Rep(R_{\g_0})$ to
$\mathscr{C}^{(1)}_Q$ over $U'_q(\g)$
$$ \mathcal{F}^{(1)}_Q:  \Rep(R_{\g_0}) \to \mathscr{C}^{(1)}_Q $$
which sends simples to simples, for
\begin{itemize}
\item[({\rm a})] each Dynkin quiver $Q$ of finite type $A$ (resp. $D$),
\item[({\rm b})] $\g$ is of type $A_n^{(1)}$ (resp $D^{(1)}_{n}$) and $\g_0$ is of type $A_n$ (resp. $D_n$).
\end{itemize}

In this paper, we investigate $\mathcal{F}^{(1)}_Q$ more deeply by
using the explicit combinatorial description of an AR-quiver
$\Gamma_Q$ which arises from a chosen Dynkin quiver $Q$ of finite
type $A$.

Note that all vertices in $\Gamma_Q$ are labeled by the set of
positive roots $\Phi_n^+$ and each positive root $\beta$ in
$\Phi_n^+$ of finite type $A$ can be expressed by the segment
$[a,b]$, where $\beta=\sum_{k=a}^b \alpha_a$. We say $a$ the first
component of $\beta$ and $b$ the second one. Identifying $\beta$
with $[a,b]$, $\Gamma_Q$ satisfies the following property: Every
positive root appearing in the {\it maximal $N$-sectional} (resp.
{\it $S$-sectional}) path in $\Gamma_Q$ has the same first (resp.
second) component (Theorem \ref{Thm: sectional}). With this
property, the facts
\begin{itemize}
\item[(i)] $\Phi_n^+=\{ [a,b] \ | \ 1 \le a \le b \le n \}$,
\item[(ii)] we know the position of simple roots in $\Gamma_Q$ by \cite[Lemma 3.2.2]{KKK13b}
\end{itemize}
provides a way (Remark \ref{Alg: easy}) for computing the bijection
$$\phi^{-1}: \Phi_n^{+} \times \{ 0 \} \longrightarrow \Gamma_Q$$
without using {\it the Coxeter element $\tau$ adapted $Q$} and {\it
the additive property of the dimension vectors}.

The description of $\mathscr{C}^{(1)}_Q$ (Definition \ref{def: Q
module category}) is the smallest category containing $V_Q(\beta)$
for $\beta \in \Phi_n^+$ and stable by certain operations on
modules, where
$$V_Q(\beta) \seteq V(\va_i)_{(-q)^{p}} \text{ is the fundamental $U_q'(A_n^{(1)})$-modules for } \phi^{-1}(\beta,0)=(i,p).$$

Using the combinatorial descriptions of $\Gamma_Q$, we can prove
that the {\it Dorey's rule} in \cite{CP96} always holds for every
$\alpha \prec_Q \beta \in \Phi_n^+$ with $\gamma=\alpha+\beta \in
\Phi_n^+$; i.e., the following surjective homomorphisms exist:
$$V_Q(\beta) \otimes V_Q(\alpha)  \twoheadrightarrow V_Q(\gamma) \quad \text{ and } \quad S_Q(\beta) \conv S_Q(\alpha)  \twoheadrightarrow S_Q(\gamma),$$
where $\prec_Q$ is {\it the convex partial order arising from
$\Gamma_Q$} and the simple $R_{A_n}$-module $S_Q(\beta)$ is the
preimage of $V_Q(\beta)$ under the functor $\mathcal{F}^{(1)}_Q$
(Theorem \ref{Thm: i_k+i_l=i_j}).

Moreover, we prove that every pair $(\alpha, \beta) \in
(\Phi_n^+)^2$ with $\alpha+\beta \in \Phi_n^+$ is indeed a {\it
minimal pair} in the sense of McNamara \cite{Mc12}, with respect a
suitable total order which is compatible with the convex partial
order $\prec_Q$ (Theorem \ref{thm: minimal}). Thus, for every pair
$(\alpha,\beta)$ with $\alpha+\beta \in \Phi_n^+$, we have
(Corollary \ref{cor: length2})
$$V_Q(\beta)\otimes V_Q(\alpha)\text{ and } S_Q(\beta) \conv S_Q(\alpha) \text{ have composition length $2$}.$$

In \cite{AK,Kas02}, it is proved that every finite dimensional
simple integrable $U_q'(\g)$-module $M$ appears as a simple head of
$\bigotimes_{i=1}^{l}V(\varpi_{i_k})_{a_k}$ for some finite sequence
$$\left( (i_1,a_1),\ldots, (i_l,a_l)\right) \text{ in } (\{ 1,\ldots,n\} \times \ko^\times)^l$$
such that the normalized $R$-matrices $\Rnorm_{i_k,i_{k'}}(z)$ have
no pole at $z=a_{k'}/a_k$, for all  $1\le k<k'\le l$. Here, $\ko
=\overline{\C(q)} \subset \cup_{m >0} \C((q^{1/m}))$. With this
theorem and the denominators in \cite[Appendix A]{Oh14}, we can
construct various exact sequences and simple modules over
$R_{A_n}$-modules, which can be understood as generalizations of the
exact sequences and simple modules in \cite[Proposition
4.2.3]{KKK13b} and \cite{KP11} (see Section \ref{sec: revisit}). We
also obtain an alternative proof for the exactness of
$\mathcal{F}^{(1)}_Q$ (Theorem \ref{thm: comb proof}). Moreover, we
can rephrase one of the main results of \cite{Mc12} with the
following forms: For any simple module $M \in \Rep(R_{A_n})$ and
Dynkin quiver $Q$ of finite type $A_n$, there exist two $n$-tuple of
sequences
\begin{align*}
& \{  (a_{1;\mathbf{i}_k},a_{2;\mathbf{i}_k},\ldots,a_{\mathbf{i}_k-1;\mathbf{i}_k},a_{\mathbf{i}_k;\mathbf{i}_k}) \mid 1 \le k \le n \} \ \text{ and } \\
& \{
(a_{\mathbf{j}_l;\mathbf{j}_l},a_{\mathbf{j}_l;\mathbf{j}_l+1},\ldots,a_{\mathbf{j}_l;n-1},a_{\mathbf{j}_l;n})
\mid 1 \le l \le n \} \end{align*} such that $a_{s;\mathbf{i}_k},
a_{\mathbf{j}_l;u} \in \Z_{\ge 0}$ for $1 \le s \le \mathbf{i}_k$,
$\mathbf{j}_l \le u \le n$, $1 \le k,l \le n$ and
$$ \dct{k=1}{n} \left(\dct{s=1}{\mathbf{i}_k} S_Q\big([s,\mathbf{i}_k]\big)^{\conv a_{s;\mathbf{i}_k}}\right) \twoheadrightarrow M \quad \text{and} \quad
\dct{l=1}{n} \left(\dct{u=\mathbf{j}_l}{n}
S_Q\big([\mathbf{j}_l,u]\big)^{\conv a_{\mathbf{j}_l;u}}\right)
\twoheadrightarrow M,$$ where $[s,\mathbf{i}_k],[\mathbf{j}_l,u] \in
\Phi_n^+$,
$(\mathbf{i}_1,\mathbf{i}_2,\cdots,\mathbf{i}_{n-1},\mathbf{i}_{n})$
and
$(\mathbf{j}_1,\mathbf{j}_2,\cdots,\mathbf{j}_{n-1},\mathbf{j}_{n})$
are re-indexing of $\{ 1,2,\cdots,n \}$ defined by the combinatorial
property we observed in this paper (Theorem \ref{thm: re-indexing}).
Here, $\dct{s=1}{\mathbf{i}_k} S_Q\big([s,\mathbf{i}_k]\big)^{\conv
a_{s;\mathbf{i}_k}}$ and $\dct{u=\mathbf{j}_l}{n}
S_Q\big([\mathbf{j}_l,u]\big)^{\conv a_{\mathbf{j}_l;u}}$ are
simple, and hence every simple $R_{A_n}$-module appears as a head of
convolution product of $n$-simple modules.

Lastly, for $A^{(2)}_{m}$, we will prove that the description for
$\mathscr{C}_{Q}^{(2)}$ in Definition \ref{def: Q module category 2}
can be re-expressed (Theorem \ref{Thm: m 2}) by using the
combinatorial properties and the Dorey's type morphisms studied in
\cite{Oh14}. The observations on $\mathscr{C}_{Q}^{(2)}$ will be one
of the main ingredients for investigating relations between
$\Rep(R_{A_n})$ and $\mathscr{C}^{(2)}_Q$ in \cite{KKKO14b}.

\medskip

We call a positive root $\beta=\sum_{k}n_k\alpha_k$ {\it
multiplicity free} if $n_k \le 1$ for all $1 \le k \le n$. Note that
all positive roots in $\Phi_n^+$ of finite type $A$ are multiplicity
free. However, it is not true for $\Phi_n^+$ of finite type $D$ and
$E$. Thus we can ask the following question naturally:
\medskip

\noindent \textbf{Conjecture} For any Dynkin quiver $Q$ of finite
type $A$, $D$ or $E$, every pair $(\alpha,\beta)$ of a positive root
$\gamma=\alpha+\beta \in \Phi_n^+$ is a minimal with respect to a
suitable total order compatible with $\prec_Q$ if and only if
$\gamma$ is multiplicity free.

\medskip

In the forthcoming paper \cite{Oh14D}, we will prove that the above
conjecture is true for $\Phi_n^+$ of finite type $D$. In other word,
there exists a pair $(\alpha,\beta)$ of a multiplicity non-free
positive root $\gamma=\alpha+\beta$ such that the pair can not be
minimal for {\it any} total order compatible with $\prec_Q$.
Furthermore, we will prove that the Dorey's rule always holds even
though $(\alpha,\beta)$ is not minimal; i.e., $S_{Q}(\beta) \circ
S_{Q}(\alpha) \twoheadrightarrow S_{Q}(\gamma)$ for every Dynkin
quiver $Q$ of finite type $D$, every positive root $\gamma$ and its
pair $(\alpha,\beta)$.

\medskip

The outline of this paper is as follows. In Section \ref{sec:
Combinatorics}, we first recall the definition of AR-quivers
$\Gamma_Q$ and their basic properties,
briefly review the various orders on $\Phi_n^+$, and  provide explicit descriptions for $\Gamma_Q$ of finite type $A$. 
In Section \ref{sec: SW-duality}, we recall the backgrounds and
theories of the generalized quantum affine Schur-Weyl duality
briefly. In the later sections, we investigate the categories
$\mathscr{C}_Q^{(i)}$ $(i=1,2)$ and $\Rep(R_{A_n})$ using the
generalized quantum affine Schur-Weyl duality and the explicit
descriptions of $\Gamma_Q$ in Section \ref{sec: Combinatorics}.

\bigskip

\noindent {\bf Acknowledgements.} The author would like to express
his gratitude to Professor Seok-Jin Kang, Professor Masaki
Kashiwara, Professor Kyungyong Lee and Myungho Kim for many fruitful
discussions. The authors would like to thank the anonymous reviewers
for their valuable comments and suggestions, also. The author
gratefully acknowledge the hospitality of RIMS (Kyoto) during his
visit in 2013 and 2014.

\section{Combinatorial characterization of AR-quivers of finite type $A$} \label{sec: Combinatorics}
In this section, we recall Gabriel's Theorem and various orders on
the set of positive roots. We also observe the combinatorial feature
of AR-quivers $\Gamma_Q$ of finite type $A$ and provides an explicit
combinatorial description of $\Gamma_Q$. We refer to
\cite{ARS,ASS,Gab80} for basic theories on quiver representations
and Auslander-Reiten theory.

\subsection{Gabriel's Theorem.} \label{subsec: Gabriel}
Let $\Delta_n$ be a rank $n$ {\it Dynkin diagram} of finite type. We
denote by $I_0=\{1,2,\ldots,n\}$ the set of indices,
$\Pi_n=\{\alpha_i \ | \ i \in I_0 \}$ the set of {\it simple roots},
$\Phi_n$ ($\Phi_n^+$, $\Phi_n^-$) the set of (positive, negative)
roots, $\mathsf{E}$ the vector space which $\Phi_n$ lies, and $( \
,\ )$ scalar product defined on $\mathsf{E}$ which are associated to
the Dynkin diagram $\Delta_n$.

Let $W_0$ be the Weyl group associated to $\Delta_n$, which is
generated by the set of simple reflections $(s_i)_{i \in I_0}$. We
denote by $w_0$ the unique longest element in $W_0$. It is
well-known that $w_0$ induces an involution $^*$ on $I_0$ given by
$w_0 \alpha_i= -\alpha_{i^*}$ (see \cite[PLATE I$\sim$IX]{Bour}).

The Dynkin quiver $Q$ is obtained by assigning orientation on the
edges of $\Delta_n$; i.e., $Q$ is a data $(Q_0,Q_1)$ where $Q_0$ is
a set of vertices indexed by $I_0$ and $Q_1$ is a set of arrows with
underlying graph $\Delta_n$. We say that a vertex $i$ is a {\it
source} (resp. {\it sink}) if and only if there are only exiting
arrows out of it (resp. entering arrows into it).

Let $\Mod \C Q$ be the category of finite dimensional modules over
the path algebra $\C Q$. An object $\mathtt{M}$ in this category is
defined by the following data:
\begin{enumerate}
\item[{\rm(a)}] To each $i \in Q_0$ is associated a finite dimensional $\C$-vector space $\mathtt{M}_i$.
\item[{\rm(b)}] To each arrow $i \overset{\mathsf{a}}{\to} j$ in $Q_1$ is associated to a linear map
$\varphi_{i \overset{\mathsf{a}}{\to} j}: \mathtt{M}_i \to
\mathtt{M}_j$.
\end{enumerate}

We define {\it the dimension vector} of $\mathtt{M} \in \Mod \C Q$
as
$$\Dim \ \mathtt{M} \seteq \sum_{i \in Q_0} (\dim \ \mathtt{M}_i) \alpha_i.$$

The simple objects in $\Mod \C Q$ are one dimensional vector spaces
$\mathtt{S}(i)$ $(i \in Q_0)$ which can be characterized by $\Dim \
\mathtt{S}(i)=\alpha_i$. We denote by $\Ind Q$ the set of
isomorphism classes $[\mathtt{M}]$ of indecomposable modules in
$\Mod \C Q$.

\begin{theorem} \label{Thm: Gabriel}
{\rm(}Gabriel's Theorem\rm{)} For a Dynkin quiver $Q$ of finite type
$A$, $D$ or $E$, the map $[\mathtt{M}] \mapsto \Dim \ \mathtt{M}$
gives a bijection from $\Ind Q$ to $\Phi^+_n$.
\end{theorem}

Thus the set $\Ind Q$ consists of the indecomposable modules
$\mathtt{M}(\beta)$ ($\beta \in \Phi_n^+$) such that $\Dim\big(
\mathtt{M}(\beta)\big)=\beta$ when a Dynkin quiver $Q$ is of finite
type $A$, $D$ or $E$.

\subsection{Auslander-Reiten quiver.} \label{subsec: AR-quiver}
In this subsection, we fix a Dynkin quiver $Q$ whose underlying
graph is of finite type $A$, $D$ or $E$. We denote by $s_iQ$ the
quiver obtained from $Q$ by reversing the orientation of each arrow
that ends at $i$ or starts at $i$. A map $\xi: I_0 \to \Z$ is called
a {\it height function on $Q$} if $\xi_j=\xi_i-1$ for $i \to j \in
Q_1$. Connectedness of $Q$ implies that any two height functions on
$Q$ differ by a constant.

Set
$$ \Z Q \seteq \{ (i,p) \in \{ 1,2,\ldots,n\} \times \Z \ | \ p -\xi_i \in 2\Z \}.$$
We view $\Z Q$ as the quiver with arrows
$$ (i,p) \to (j,p+1), \ (j,q) \to (i,q+1) \ \text{ for which $i$ and $j$ are adjacent in $\Delta_n$}$$
and call it the {\it repetition quiver} of $Q$. Note that $\Z Q$
does not depend on the orientation of the quiver $Q$. It is
well-known that the quiver $\Z Q$ itself has an isomorphism with the
AR-quiver of $D^b(\C Q)$-${\rm mod}$, the bounded derived category
of $\C Q$-${\rm mod}$ (\cite{Hap}). In our convention, the injective
module $\mathtt{I}(i)$ is located on the vertex $(i,\xi_i)$ of $\Z
Q$.

\medskip

For a reduced expression $\widetilde{w}=s_{i_1}s_{i_2}\cdots
s_{i_l}$ of an element $w$ in $W_0$, it is called {\it adapted to
$Q$} if
$$\text{$i_k$ is a source of $s_{i_{k-1}}\cdots s_{i_{2}}s_{i_{1}}Q$ for all $1 \le k \le l$.}$$
It is well-known that there is a {\it unique Coxeter element} $\tau$
whose reduced representations are adapted to $Q$.

Set $\widehat{\Phi}_n \seteq \Phi^+_n \times \Z$. For $i \in I_0$,
we define
\begin{align} \label{eq: dim P I}
\gamma_i = \sum_{j \in B(i)} \alpha_j \quad  \text{and}
\quad\theta_i = \sum_{j \in C(i)} \alpha_j \qquad \text{ where }
\end{align}
\begin{itemize}
\item $B(i)$ is the set of vertices $j$ in $Q_0$
such that there exists a path from $j$ to $i$,
\item $C(i)$ is the set of vertices $j$  in $Q_0$
such that there exists a path from $i$ to $j$.
\end{itemize}

The bijection $\phi : \Z Q \rightarrow \widehat{\Phi}_n$ defined by
$\mathtt{M}(\beta)[m] \mapsto (\beta,m)$ is described
combinatorially as follows (\cite[\S 2.2]{HL11}):
\begin{eqnarray} &&
  \parbox{80ex}{
\begin{enumerate}
\item[({\rm a})] $\phi(i,\xi_i)=(\gamma_i,0)$,
\item[({\rm b})] for a given $\beta \in \Phi_n^+$ with $\phi(i,p)=(\beta,m)$,
\begin{itemize}
 \item if $\tau(\beta) \in \Phi^+_n$, we set $\phi(i,p-2)=(\tau(\beta), m)$,
      \item if $\tau(\beta) \in \Phi^-_n$, we set $\phi(i,p-2)=(-\tau(\beta), m-1)$,
      \item if $\tau^{-1}(\beta) \in \Phi^+_n$, we set $\phi(i,p+2)=(\tau^{-1}(\beta), m)$,
      \item if $\tau^{-1}(\beta) \in \Phi^-_n$, we set $\phi(i,p+2)=(-\tau^{-1}(\beta), m+1)$.
 \end{itemize}
\end{enumerate}} \label{eq:phi}
\end{eqnarray}

The {\it Auslander-Reiten quiver $\Gamma_Q$} is the full subquiver
of $\Z Q$ whose set of vertices is $\phi^{-1}(\Phi^+_n \times \{ 0
\} )$. Here the vertex $\phi^{-1}(\beta,0)$ corresponds to the
indecomposable module $\mathtt{M}(\beta)$ in $\Ind Q$ and the arrow
$\phi^{-1}(\beta,0) \to \phi^{-1}(\beta',0)$  corresponds to an {\it
irreducible} morphism from $\mathtt{M}(\beta)$ to
$\mathtt{M}(\beta')$.

In particular, the injective envelope $\mathtt{I}(i)$ of
$\mathtt{S}(i)$ corresponds to the vertex $\phi^{-1}(\gamma_i,0)$
and the projective cover $\mathtt{P}(i)$ of $\mathtt{S}(i)$
corresponds to the vertex $\phi^{-1}(\theta_{i^*},0)$. It is
well-known that
\begin{equation} \label{Def: m_i}
\theta_i = \tau^{m_{i^*}}(\gamma_{i^*}) \quad \text{ where }   m_i =
\max \{ k \ge 0 \ | \ \tau^{k}(\gamma_i) \in \Phi^+_n \}.
\end{equation}

For $\beta \in \Phi^+_n$ with $\tau(\beta) \in \Phi^+_n$, we set
$\tau \mathtt{M}(\beta)\seteq \mathtt{M}(\tau(\beta))$. In the
Auslander-Reiten quiver $\Gamma_Q$, this map $\tau$ is called {\it
the Auslander-Reiten translation}.

{\rm (i)} The dimension vector is an {\it additive function} on
$\Gamma_Q$ with respect to the map $\tau$; that is, for each
vertices $\mathtt{X} \in \Gamma_Q$ such that $\mathtt{X} =
\phi^{-1}(\beta,0)$ and $\tau(\beta) \in \Phi^+_n$,
\begin{align} \label{eq: additive function}
\Dim \ \mathtt{X} + \Dim \ \tau \mathtt{X} = \sum_{\mathtt{Z} \in
\mathtt{X}^{-}}\Dim \ \mathtt{Z}.
\end{align}
Here $\mathtt{X}^-$ is the set of vertices $\mathtt{Z}$ in
$\Gamma_Q$ such that there exists an arrow from $\mathtt{Z}$ to
$\mathtt{X}$.

{\rm (ii)} It is well-known that, for $\beta \in \Phi_n^+$,
\begin{align} \label{eq: negative image}
\tau(\beta) \in \Phi^-_n \quad \text{ if and only if } \quad  \beta
= \theta_i \quad \text{ for some } i \in I_0.
\end{align}

A subquiver $\Gamma'$ of $\Gamma_Q$ is said to be {\it sectional}
(\cite{Br}) if whenever $x \to y$ and $y \to z$ are arrows in
$\Gamma'$, we have $x \ne \tau(z)$.

The following description is one of the characterization of
$\Gamma_Q$ inside $\Z Q$:
\begin{align} \label{eq: known characterization}
 \phi^{-1}(\Phi^+_n \times \{ 0 \} ) = \{ (i,p) \in \Z Q \ | \ \xi_i-2m_i \le p \le \xi_i \}.
\end{align}

Recall that there is the {\it Nakayama permutation} $\nu$ on $\Z Q$
which is given by
$$ \nu(i,p) = (i^*,p+\mathtt{h}_n-2) $$
where $\mathtt{h}_n$ is the Coxeter number associated to $\Delta_n$.
Since, for all $i \in I_0$, the Nakayama permutation $\nu$ sends
vertices corresponding to $\mathtt{P}(i)$ to vertices corresponding
to $\mathtt{I}(i^*)$, we have
\begin{align} \label{eq: Nakayama equation}
\xi_{i^*}-2m_{i^*}= \xi_i - \mathtt{h}_n +2.
\end{align}

\subsection{Orders on $\Phi_n^+$.} \label{subsec: order}
 In this section, we recall various orders on $\Phi_n^+$.
\medskip

A {\it convex order} on $\Phi_n^+$ is a partial order $\preceq$
satisfying the following condition (\cite{BFZ96}):
$$\text{For all $\alpha,\beta$ and $\gamma=\alpha+\beta \in \Phi_n^+$, either $\alpha \prec \gamma \prec \beta$ or $\beta \prec \gamma   \prec \alpha$.}$$

Note that a {\it coarsest} convex partial order on $\Phi_n^+$
determines the {\it commutation class} of $w_0$; that is, the
equivalence class on reduced expressions of $w_0$ which is given by
$$s_{i_1}s_{i_2}\cdots s_{i_{\mathsf{N}}} \sim s_{j_1}s_{j_2}\cdots s_{j_{\mathsf{N}}} \qquad (\mathsf{N} \seteq |\Phi_n^+|)$$
if and only if $s_{j_1}s_{j_2}\cdots s_{j_{\mathsf{N}}}$ is obtained
from $s_{i_1}s_{i_2}\cdots s_{i_{\mathsf{N}}}$ by replacing
$s_{a}s_{b}$ by $s_{b}s_{a}$ for $a$ and $b$ not linked in
$\Delta_n$ (see \cite{B99}).

A Dynkin quiver $Q$ defines a convex partial order $\preceq_Q$ on
$\Phi_n^+$ in the following way \cite{R96}:
\begin{equation}\label{eq: paths Q}
\alpha \preceq_Q \beta \text{ if and only if } \text{there is a path
from $\beta$ to $\alpha$ in $\Gamma_Q$}.
\end{equation}
Thus, for a pair $(\alpha,\beta)$ with
$\gamma=\alpha+\beta\in\Phi_n^+$ and $\alpha \prec_Q \beta$,
\begin{equation}\label{eq: paths pair Q}
\text{ there exist paths from $\beta$ to $\gamma$ and from $\gamma$
to $\alpha$ in $\Gamma_Q$}.
\end{equation}

On the other hand, it is well-known that a reduced expression
$\redez=s_{i_1}s_{i_2} \cdots s_{i_N}$ of $w_0$ induces a {\it
convex total order} $\le_{\redez}$ on $\Phi^+_n$ by (\cite{Bour})
\begin{align} \label{eq: total}
\beta_z \seteq  s_{i_1}s_{i_2} \cdots s_{i_{z-1}}\alpha_{i_z} \text{
and $\beta_x <_{\redez} \beta_y$ if and only if $x < y$}.
\end{align}
Note that any convex total order is obtained in this way by a
reduced expression $\redez$ of $w_0$ (\cite{Papi94}).

The following theorem provides a way of obtaining all reduced
expressions of $w_0$ adapted to $Q$ and hence convex total orders
compatible with the convex partial order $\preceq_Q$:
\begin{align} \label{eq: compati}
\alpha \prec_Q \beta \text{ implies } \alpha <_{\widetilde{w}_0}
\beta \text{ for any } \widetilde{w}_0 \text{ adapted to } Q.
\end{align}

\begin{theorem} \cite[Theorem 2.17]{B99} \label{Thm: Berdard reading}
Let $Q$ be a quiver of finite type $A$, $D$ or $E$. Then any reduced
expression of $w_0$ adapted to $Q$ can be obtained in the following
way: We read $\Gamma_Q$ sequentially, in a manner compatible with
arrows; that is, if there exists an arrow $\beta \to \alpha$,
$\alpha$ appears before $\beta$. Replacing a vertex $\beta$ with $i$
for $\phi^{-1}(\beta)=(i,p)$, we have a sequence
$(i_1,i_2,\ldots,i_{\mathsf{N}})$ giving a reduced expression of
$w_0$ adapted to $Q$,
$$\redez = s_{i_1}s_{i_2}\cdots s_{i_{\mathsf{N}}}.$$
\end{theorem}

\begin{remark} \label{Rmk: total orders} In this remark, we suggest two canonical readings of $\Gamma_Q$ which follow the rule in Theorem \ref{Thm: Berdard reading}.
Thus we can obtain two convex total orders on $\Phi_n^+$ compatible
with $\preceq_Q$.
\begin{enumerate}
\item[({\rm A})]  We denote by $<^{L}_Q$ the convex total order induced from the following reading:
$$ \text{$(i,p)$ appears before $(i',p') \iff  \begin{cases} d(1,i)-p < d(1,i')-p' \quad\text{ or } \\ d(1,i)-p = d(1,i')-p'
\quad\text{ and }\quad i > i'. \end{cases}$}$$
\item[({\rm B})] We denote by $<^{U}_Q$ the convex total order induced from the following reading:
$$ \text{$(i,p)$ appears before $(i',p') \iff  \begin{cases} d(1,i)+p > d(1,i')+p' \quad\text{ or } \\ d(1,i)+p = d(1,i')+p'
\quad\text{ and }\quad i < i'.\end{cases}$}$$
\end{enumerate}
Here, $d(i,j)$ denotes the distance between $i$ and $j$ in
$\Delta_n$.
\end{remark}

\begin{definition} \label{Def: minimal pair}\cite[\S 2.1]{Mc12}.
For a total order $<$ on $\Phi^+_n$, a pair $(\alpha,\beta)$ with
$\alpha<\beta$ is called a {\it minimal pair of $\gamma \in
\Phi_n^+$ with respect to the total order $<$} if
\begin{itemize}
\item $\gamma=\alpha+\beta$,
\item there exist {\it no} pair $(\alpha',\beta')$ such that $\gamma=\alpha'+\beta' \text{ and }\alpha<\alpha'<\gamma<\beta'<\beta.$
\end{itemize}
\end{definition}

\subsection{Characterization of Auslander-Reiten quiver of type $A$} In this subsection, we fix the $\Delta_n$ as the Dynkin diagram of finite type $A$.
Citing the referee, the combinatorial properties in this subsection
is well-known to the experts. More precisely, using the {\it linear
quiver} $\overset{\gets}{Q}$ (see \eqref{def: linear quiver}), the
{\it reflection functor} on $D^b(\C Q)$-${\rm mod}$ and the tilting
theorem (\cite[Chapter VII]{ASS}), one can observe the descriptions
in this subsection. However, we have a difficulty for finding the
explicit statement of Theorem \ref{Thm: sectional} in standard
textbooks, we shall deal with and derive it from Lemma \ref{Lem: Key
Lem} and Remark \ref{rem: key} below.

Since $\Delta_n$ is the Dynkin diagram of finite type $A$, we have
the Coxeter number $\mathtt{h}_n=n+1$ and the involution induced by
$w_0 \in W_0$ is given by $i \mapsto i^*= n+1-i$. We say that a
vertex $i \in Q$ is a right (resp. left) intermediate if
$$ \xymatrix@R=3ex{ *{ \ }<3pt> \ar@{..}[r]  &*{\bullet}<3pt>
\ar@{<-}[r]_<{i-1}  &*{\bullet}<3pt> \ar@{<-}[r]_<{i}
&*{\bullet}<3pt> \ar@{..}[r]_<{i+1}  & *{ \ }<3pt> }\quad \text{
(resp. } \xymatrix@R=3ex{ *{ \ }<3pt> \ar@{..}[r]  &*{\bullet}<3pt>
\ar@{->}[r]_<{i-1}  &*{\bullet}<3pt> \ar@{->}[r]_<{i}
&*{\bullet}<3pt> \ar@{..}[r]_<{i+1}  & *{ \ }<3pt> } ) \quad
\text{in $Q$}.$$

Recall that, for every $1 \le a \le b \le n$, $\beta = \sum_{a \le k
\le b} \alpha_k$ is a positive root in $\Phi^+_n$ and every positive
root in $\Phi^+_n$ is of the form. Thus we sometimes identify $\beta
\in \Phi^+_n$ (and hence vertex in $\Gamma_Q$) with a segment
$[a,b]$. For $\beta=[a,b]$, we say $a$ the {\it first component} of
$\beta$ and $b$ the {\it second component} of $\beta$. If $\beta$ is
simple, we write $\beta$ as $[a]$.

\begin{example} \label{ex: example 1}
Consider the quiver $\xymatrix@R=3ex{ *{ \bullet }<3pt>
\ar@{->}[r]_<{1}  &*{\bullet}<3pt> \ar@{<-}[r]_<{2}
&*{\bullet}<3pt> \ar@{->}[r]_<{3} &*{\bullet}<3pt> \ar@{->}[r]_<{4}
& *{\bullet}<3pt> \ar@{-}[l]^<{\ \ 5} }$ of type $A_5$. We set
$\xi_1=0$. Then the $\Gamma_Q$ inside $\Z Q$ can be obtained by
using the Coxeter element $\tau$ \eqref{eq:phi} or the additive
property of the dimension vectors \eqref{eq: additive function} as
follows:
$$ \scalebox{0.9}
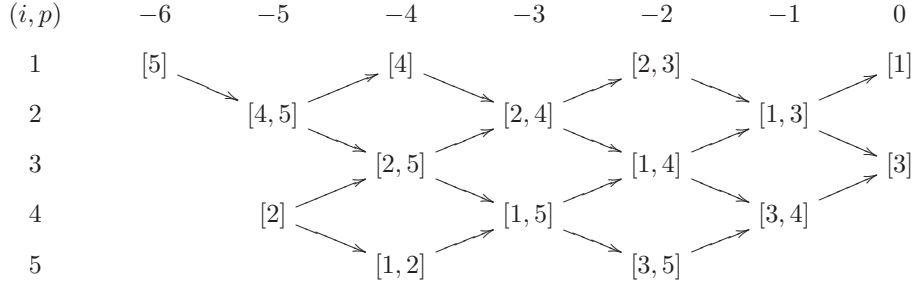
{\xymatrix@R=0.5ex{
(i,p) & -6 & -5 & -4 & -3 & -2 & -1 & 0 \\
1&[5] \ar@{->}[dr] && [4]\ar@{->}[dr] && [2,3]\ar@{->}[dr] && [1] \\
2&& [4,5]\ar@{->}[dr]\ar@{->}[ur] && [2,4]\ar@{->}[dr]\ar@{->}[ur] && [1,3]\ar@{->}[dr]\ar@{->}[ur] \\
3&&& [2,5]\ar@{->}[dr]\ar@{->}[ur] && [1,4]\ar@{->}[dr]\ar@{->}[ur] && [3] \\
4&& [2] \ar@{->}[dr]\ar@{->}[ur], && [1,5]\ar@{->}[dr]\ar@{->}[ur] && [3,4]\ar@{->}[ur] \\
5&&& [1,2]\ar@{->}[ur] && [3,5]\ar@{->}[ur] }}
$$
\end{example}

\begin{definition} 
\begin{enumerate}
\item[({\rm a})] A path $\rho$ in $\Gamma_Q$ is {\it $S$-sectional} (resp. {\it $N$-sectional}) if the path is a concatenation of arrows whose forms are $(i,p) \to (i+1,p+1)$
(resp. $(i,p) \to (i-1,p+1)$).
\item[({\rm b})] A positive root $\beta \in \Phi^+_n$ is {\it contained in the path $\rho$} in $\Gamma_Q$ if
$\beta$ is an end or a start of some arrow in the path $\rho$.
\item[({\rm c})] An $S$-sectional (resp. $N$-sectional) path $\rho$ is {\it maximal} if there is no longer $S$-sectional (resp. $N$-sectional) path
containing all positive roots in $\rho$.
\item[({\rm d})] For $\beta \in \Phi^+_n$ with $\phi^{-1}(\beta,0) \hspace{-0.3ex} = \hspace{-0.3ex}(i,p)$, we denote by
$\phi^{-1}_1(\beta)\hspace{-0.3ex}=\hspace{-0.3ex}i$ and
$\phi^{-1}_2(\beta)\hspace{-0.3ex}=\hspace{-0.3ex}p$.
\end{enumerate}
\end{definition}

The main goal of this subsection is to provide an explicit
description of an AR-quiver of finite type $A$ by using the
following lemma and basic properties for an AR-quiver in \S
\ref{subsec: AR-quiver}.

\begin{lemma}\cite[Lemma 2.11]{B99},\cite[\S 3.2]{KKK13b}  \label{Lem: Key Lem} 
\begin{enumerate}
\item[({\rm a})]
For $k \in I_0$,
$$\phi^{-1}(\alpha_k,0) = \begin{cases} (k,\xi_k) & \text{ if $k$ is a source}, \\
(n+1-k, \xi_k-n+1)& \text{ if $k$ is a sink}, \\
(1, \xi_k-k+1) & \text{ if $k$ is a left intermediate, } \\
(n, \xi_k-n+k) & \text{ if $k$ is a right intermediate. }
\end{cases}$$
\item[({\rm b})] If $\beta \to \alpha$ is an arrow in $\Gamma_Q$ for $\alpha,\beta \in \Phi^+_n$, then $(\beta,\alpha)=1$.
\item[({\rm c})] We have
\begin{align} \label{eq: range}
(i,\xi_j-d(i,j)),\ (i,\xi_j-2m_j+d(i,j))\in\Gamma_Q \quad \text{for
any $i,j\in I_0$.}
\end{align}
\end{enumerate}
\end{lemma}

\begin{proof}
The second and third assertions are identical with \cite[Lemma
2.11]{B99} and \cite[Lemma 3.2.2]{KKK13b}, respectively. For the
first assertion, we can use the
 \cite[Lemma 3.2.3]{KKK13b} in the following way: {\rm (i)} If $k$ is a left intermediate, then take the extremal vertex $i$ as $1$ in $I_0$. Then
the formula
\begin{equation} \label{eq: position simple}
\phi^{-1}(\alpha_k,0)=(i,\xi_k-d(i,k))
\end{equation}
in \cite[page 17]{KKK13b} tells that
$\phi^{-1}(\alpha_k,0)=(1,\xi_k-k+1)$. {\rm (ii)} If $k$ is a right
intermediate, then take the extremal vertex $i$ as $n$ in $I_0$.
Then \eqref{eq: position simple} tells that
$\phi^{-1}(\alpha_k,0)=(n, \xi_k-n+k)$. It is obvious for $k$ which
is a source or a sink.
\end{proof}

\begin{remark} \label{rem: key}
For a left (resp. right) intermediate $k$, we can take also the
extremal vertex $i$ as $n$ (resp. $1$) in $I_0$. Then we have
\begin{equation}\label{eq: position simple 2}
\phi^{-1}(\alpha_k,0)=(i^*,\xi_{k^*}-2m_{k^*}+d(i,k))
\end{equation}
by the formula \cite[page 18]{KKK13b}.
\end{remark}

\begin{proposition} \label{Prop: same component}
Let $\rho$ be an $N$-sectional $($resp. $S$-sectional$)$ path
containing simple root. Then every positive roots contained in
$\rho$ has the same first $($resp. second$)$ component.
\end{proposition}

\begin{proof}
(1) Let $t$ be a source in $Q$; i.e.,
$\xymatrix@R=3ex{*{\bullet}<3pt> \ar@{<-}[r]_<{t-1}
&*{\bullet}<3pt> \ar@{->}[r]_<{t} &*{\bullet}<3pt> \ar@{-}[l]^<{\ \
t+1} }$. Since $t$ is a source,$$\text{ $\Dim \mathtt{I}(t-1)=
[a,t]$,  $\Dim \mathtt{I}(t+1)= [t,b]$ for some $a \le t-1 $ and $b
\ge t+1$.}$$ Note that $t-1$ or $t+1 \in I_0$. By Lemma \ref{Lem:
Key Lem} (c), the following subquiver is contained in $\Gamma_Q$:
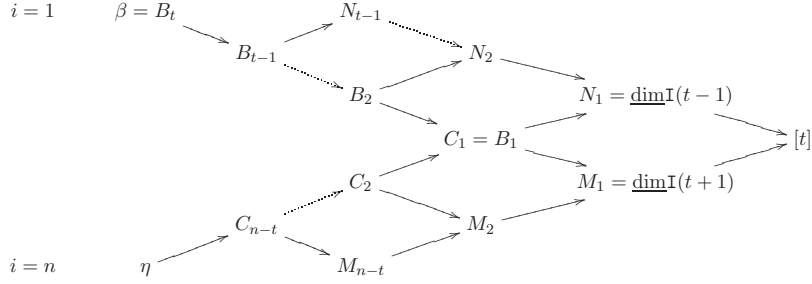
$$\scalebox{0.7}{\xymatrix@R=1ex{
i=1&\beta=B_t\ar@{->}[dr] && N_{t-1}\ar@{.>}[dr] \\
&& B_{t-1}\ar@{.>}[dr]\ar@{->}[ur] && N_2 \ar@{->}[dr] \\
&&& B_2\ar@{->}[dr]\ar@{->}[ur] &&  N_1=\Dim \mathtt{I}(t-1)\ar@{->}[dr] \\
&&&& C_1=B_1\ar@{->}[ur]\ar@{->}[dr]  && [t] \\
&&& C_2\ar@{->}[ur]\ar@{->}[dr] && M_1=\Dim \mathtt{I}(t+1)\ar@{->}[ur] \\
&& C_{n-t} \ar@{.>}[ur]\ar@{->}[dr] && M_2\ar@{->}[ur] \\
i=n& \eta \ar@{->}[ur] && M_{n-t}\ar@{->}[ur] }}$$ By the additive
property of the dimension vectors \eqref{eq: additive function}, we
have
\begin{itemize}
\item $[t+1,b]=M_1-[t]=B_1-N_1=B_2-N_2=B_{t-1}-N_{t-1}=\beta$ and
\item $[a,t-1]=N_1-[t]=C_1-M_1=C_2-M_2=C_{n-t}-M_{n-t}=\eta$.
\end{itemize}
Since $\beta+N_2=B_2 \in \Phi^+_n$, $N_2$ is $[b+1,d]$ or $[a_2,t]$.
If $N_2=[b+1,d]$, then $(N_1,N_2)=0$, which yields a contradiction
to Lemma \ref{Lem: Key Lem} (b). Thus $N_2=[c,t]$. In this way, one
can show that $N_u$ for $1 \le u \le t-1$ is of the form $[a_u,t]$.
By the similar way, we can prove that $M_v$ for $1 \le v \le n-t$ is of the form $[t,b_v]$. \\
(2) Let $t$ be a sink in $Q$; i.e., $\xymatrix@R=3ex{*{\bullet}<3pt>
\ar@{->}[r]_<{t-1}  &*{\bullet}<3pt> \ar@{<-}[r]_<{t}
&*{\bullet}<3pt> \ar@{-}[l]^<{\ \ t+1} }$.  Since $t$ is a
sink,$$\text{ $\Dim \mathtt{P}(t-1)= [a,t]$,  $\Dim \mathtt{P}(t+1)=
[t,b]$ for some $a \le t-1 $ and $b \ge t+1$.}$$
By using a similar argument given in (1), one can prove our assertion. \\
(3) Let $t$ be a left intermediate in $Q$; i.e.,
$\xymatrix@R=3ex{*{\bullet}<3pt> \ar@{->}[r]_<{t-1}
&*{\bullet}<3pt> \ar@{->}[r]_<{t} &*{\bullet}<3pt> \ar@{-}[l]^<{\ \
t+1} }$. Then we have
$$\Dim \mathtt{I}(t-1)=[k,t-1], \ \Dim \mathtt{I}(t)=[k,t], \ \Dim \mathtt{P}(t^*)=[t,l] \text{ and } \Dim \mathtt{P}(t^*-1)=[t+1,l]$$
for some $k \le t-1$ and $t+1 \le l$. Moreover, Lemma \ref{Lem: Key
Lem} (a) tells that $\phi^{-1}(\alpha_t,0)=(1,\xi_t-t+1)$. By Lemma
\ref{Lem: Key Lem} (c), the following subquivers are contained in
$\Gamma_Q$:
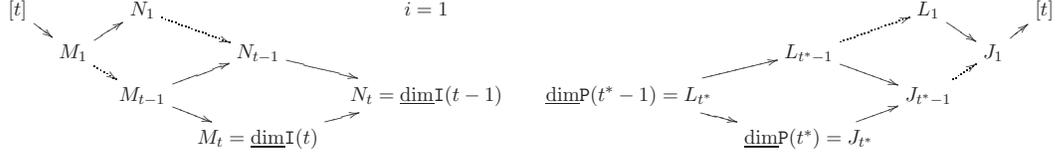
$$
\scalebox{0.7}{\xymatrix@C=2ex@R=1ex{
[t]\ar@{->}[dr] && N_{1}\ar@{.>}[dr] & & i=1 \\
& M_{1}\ar@{.>}[dr]\ar@{->}[ur] && N_{t-1} \ar@{->}[dr] \\
&& M_{t-1}\ar@{->}[dr]\ar@{->}[ur] &&  N_{t}=\Dim \mathtt{I}(t-1) \\
&&& M_{t}=\Dim \mathtt{I}(t)\ar@{->}[ur]   \\
}}\quad \scalebox{0.7}{\xymatrix@C=2ex@R=1ex{
&&L_1\ar@{->}[dr]&&[t] \\
&L_{t^*-1}\ar@{.>}[ur]\ar@{->}[dr]&&J_1\ar@{->}[ur] \\
\Dim \mathtt{P}(t^*-1)=L_{t^*}\ar@{->}[ur]\ar@{->}[dr]&&J_{t^*-1}\ar@{.>}[ur] \\
& \Dim \mathtt{P}(t^*)=J_{t^*}\ar@{->}[ur] }}$$ Since $[t]+N_k=M_k
\in \Phi^+_n$, $N_k$ is $[a_k,t-1]$ or $[t+1,b_k]$. However
$(N_{t-1},N_{t})=1$ implies that $N_{t-1}=[a_{t-1},t-1]$ and hence
$M_{t-1}=[a_{t-1},t]$. Then we can show $M_{k}=[a_{k},t]$ for $1 \le
k \le t$ by using the additive property of the dimension vectors
\eqref{eq: additive function}.

By applying the same argument, one can prove $J_k=[t,b_k]$ for $1 \le k \le t^*$. \\
(4) Let $t$ be a right intermediate in $Q$; i.e.,
$\xymatrix@R=3ex{*{\bullet}<3pt> \ar@{<-}[r]_<{t-1}
&*{\bullet}<3pt> \ar@{<-}[r]_<{t} &*{\bullet}<3pt> \ar@{-}[l]^<{\ \
t+1} }$. Then we have
$$ \Dim \mathtt{I}(t)=[t,l], \ \Dim \mathtt{I}(t+1)=[t+1,l], \ \Dim \mathtt{P}(t^*)=[k,t] \text{ and } \Dim \mathtt{P}(t^*+1)=[k,t-1] $$
for some $k \le t-1$ and $t+1 \le l$. By using a similar argument
given in (3), one can prove our assertion.
\end{proof}

\begin{proposition} \label{Prop: maximal path}
Let $\rho$ be a maximal $N$-sectional path or a maximal
$S$-sectional path. Then it contains a simple root.
\end{proposition}

\begin{proof}
By Proposition \ref{Prop: same component}, $\rho$ contains at most
one simple root. Assume that $\rho$ is a maximal $N$-sectional path
which does not contain a simple root, and we write $\rho= N_t \to
N_{t-1} \to \cdots \to N_1$.

Assume first that $\phi^{-1}_1(N_1) = a \neq 1$. Hence,
$N_1=\tau^k(\underline{\dim}\mathtt{I}(a))$, for some $k\ge 0$. If
the arrow between $a$ and $a-1$ in $\Delta_n$ was
$\xymatrix@R=3ex{*{\bullet}<3pt> \ar@{->}[r]_<{a-1}
&*{\bullet}<3pt> \ar@{-}[l]^<{ \ \ a} }$, the irreducible morphism
$\mathtt{I}(a)\to \mathtt{I}(a-1)$ would give the arrow
$(a,\xi_a)\to (a-1,\xi_a+1)$. Applying $\tau^k$, we would have a
contradiction to the maximality of $\rho$. Thus, $a$ is either right
intermediate or a source:
$$\text{$\xymatrix@R=3ex{*{\bullet}<3pt>
\ar@{<-}[r]_<{a-1}  &*{\bullet}<3pt> \ar@{<-}[r]_<{a}
&*{\bullet}<3pt> \ar@{-}[l]^<{\ \ a+1} }$ or
$\xymatrix@R=3ex{*{\bullet}<3pt> \ar@{<-}[r]_<{a-1}
&*{\bullet}<3pt> \ar@{->}[r]_<{a} &*{\bullet}<3pt> \ar@{-}[l]^<{\ \
a+1} }$ in $Q$.}$$

In both cases, $N_1$ is the vertex for an injective module:
otherwise, we have an arrow $N_1\to
\tau^{k-1}(\underline{\dim}\mathtt{I}(a-1))$ in $\Gamma_Q$ and it
contradicts the maximality of $\rho$ again. Thus we have
$\phi^{-1}(N_1,0)=(a,\xi_a)$. Hence Lemma \ref{Lem: Key Lem} (c)
tells that $\phi^{-1}(N_t,0)= (n,\xi_a-n+a)$. Thus Lemma 1.7(a)
implies that $N_t=[a]$ if $a$ is right intermediate and $N_1=[a]$ if
$a$ is a source, contradicting our assumption that $\rho$ does not
contain a simple root. Hence we have $\phi^{-1}_1(N_1)=1$.

If $\phi^{-1}_2(N_1)=\xi_1$ then $N_1=\Dim \mathtt{I}(1)$. Since $1$
is a source or sink, then $N_1=[1]$, or $N_{t=n}=[1]$ and
$\phi^{-1}(N_n)=(n,\xi_n-2m_n)$ contradicting our assumption on
$\rho$ again.

Now we assume that $\phi^{-1}_2(N_1) < \xi_1$. If $\phi^{-1}_1(N_t)
=b \neq n$, then
$$\text{$\xymatrix@R=3ex{*{\bullet}<3pt>
\ar@{->}[r]_<{b-1}  &*{\bullet}<3pt> \ar@{->}[r]_<{b}
&*{\bullet}<3pt> \ar@{-}[l]^<{\ \ b+1} }$ or
$\xymatrix@R=3ex{*{\bullet}<3pt> \ar@{->}[r]_<{b-1}
&*{\bullet}<3pt> \ar@{<-}[r]_<{b} &*{\bullet}<3pt> \ar@{-}[l]^<{\ \
b+1} }$ in $Q$,}$$ by the similar reason as above. In both cases,
$\phi^{-1}(N_t,0)=(b,\xi_b-2m_b)$. Hence Lemma \ref{Lem: Key Lem}
(c) tells that $\phi^{-1}_1(N_1)= (1,\xi_b-2m_b+(b-1))$. Thus Lemma
\ref{Lem: Key Lem} (a) and Remark \ref{rem: key} imply that
$N_1=[b^*]$ if $b$ is a left intermediate, and $N_t=[b^*]$ if $b$ is
a sink, respectively. Thus we have $t=n$ and $\phi^{-1}_1(N_n)=n$.
However, the below inequality
$$\phi^{-1}_2(N_n)=\phi^{-1}_2(N_1)-n+1 < \xi_n-2m_n=\xi_1-n+1$$
yields a contradiction to \eqref{eq: known characterization}.
Similarly, one can prove our assertion for a maximal $S$-sectional
path.
\end{proof}


\begin{theorem} \label{Thm: sectional}
Every positive root in an $N$-sectional path has the same first
component and every root in an $S$-sectional path has the same
second component.
\end{theorem}

\begin{proof}
By the previous proposition, every maximal $N$-sectional path (resp.
maximal $S$-sectional path) contains a simple root. Thus our
assertion follows from Proposition \ref{Prop: same component} and
the fact that every $N$-sectional path (resp. $S$-sectional path) is
contained in some maximal $N$-sectional path (resp. maximal
$S$-sectional path).
\end{proof}

Thus we say that an $N$-sectional path $\rho$ is the {\it maximal
$(N,i)$-sectional path} if all positive roots contained in $\rho$
have $i$ as a first component. Similarly, we can define a notion of
the {\it maximal $(S,i)$-sectional path}.

\begin{corollary} \label{Cor: sectional}
For $1 \le i \le n$, the Auslander-Reiten quiver $\Gamma_Q$ of
finite type $A$ contains a maximal $N$-sectional path of length
$n-i$ once and exactly once. At the same time, $\Gamma_Q$ contains a
maximal $S$-sectional path of length $i-1$ once and exactly once.
\end{corollary}

\begin{proof}
For $i \in I$, there are $(n-i+1)$-many positive roots whose first
component is $i$, and $i$-many positive roots whose second component
is $i$. Thus our assertion follows from Theorem \ref{Thm:
sectional}.
\end{proof}

\begin{corollary} \label{Cor: one path}
For vertices $\beta$, $\beta' \in \Phi_n^+$ with $\beta \prec_Q
\beta'$ and $\beta+\beta' \in \Phi_n^+$,
\begin{equation} \label{eq: paths}
\text{there exists only one path from $\beta'$ to $\beta+\beta'$.
So is from $\beta+\beta'$ to $\beta$.}
\end{equation}
\end{corollary}

\begin{proof}
The proof comes from \eqref{eq: paths Q} and Theorem \ref{Thm:
sectional}.
\end{proof}

To get a $\Gamma_Q$ from $Q$, we have used $\tau(\beta)$
\eqref{eq:phi} or the additive property of the dimension vectors
\eqref{eq: additive function}. On the other hand, Lemma \ref{Lem:
Key Lem} and Theorem \ref{Thm: sectional} provide a combinatorial
way for computing the bijection $\phi^{-1}$ without using the
Coxeter element and the additive property of the dimension vectors.

\medskip

The following remark is well-known to experts but we record here for
later use.

\begin{remark}  \label{Alg: easy} 
({\rm a}) By {\rm Lemma \ref{Lem: Key Lem}} (a), we know the images of $\Pi_n$ in $\Gamma_Q$. \\
({\rm b}) For each $a \in I_0$, we draw a hook as follows:
\begin{itemize}
\item[({\rm i})] If $a$ is a source, $$\overbrace{N_n \to \cdots N_{a-1} \to N_a}^{\text{maximal $(N,a)$-sectional path}}=[a]
=\underbrace{S_{a} \gets S_{a-1} \gets \cdots \gets
S_1}_{\text{maximal $(S,a)$-sectional path}}.$$
\item[({\rm ii})] If $a$ is a sink,
$$\overbrace{N_n \gets \cdots N_{n-a} \gets N_{n-a+1}}^{\text{maximal $(N,a)$-sectional path}}=[a]
=\underbrace{S_{n-a+1} \to S_{n-a+2} \to \cdots \to
S_1}_{\text{maximal $(S,a)$-sectional path}}.$$
\item[({\rm iii})] If $\xymatrix@R=3ex{*{\bullet}<3pt>
\ar@{->}[r]_<{a-1}  &*{\bullet}<3pt> \ar@{->}[r]_<{a}
&*{\bullet}<3pt> \ar@{-}[l]^<{\ \ a+1} }$ in $Q$,
$$\overbrace{N_{n-a+1} \to \cdots N_{2} \to N_{1}}^{\text{maximal
$(N,a)$-sectional path}} =[a] = \underbrace{S_{1} \to S_{2} \to
\cdots \to S_a}_{\text{maximal $(S,a)$-sectional path}}.$$
\item[({\rm iv})] If $\xymatrix@R=3ex{*{\bullet}<3pt>
\ar@{<-}[r]_<{a-1}  &*{\bullet}<3pt> \ar@{<-}[r]_<{a}
&*{\bullet}<3pt> \ar@{-}[l]^<{\ \ a+1} }$ in $Q$, $$\overbrace{N_a
\gets \cdots N_{n-1} \gets N_{n}}^{\text{maximal $(N,a)$-sectional
path}} =[a] = \underbrace{S_{n} \gets S_{n-1} \gets \cdots \gets
S_{n+1-a}}_{\text{maximal $(S,a)$-sectional path}}.$$
\end{itemize}
Here $\phi^{-1}_1(N_k)=k$ and $\phi^{-1}_1(S_l)=l$. \\
({\rm c}) If $(i,p)$ in $\Gamma_Q$ is located at an intersection of
an $(N,a)$-sectional path and an $(S,b)$-sectional path,
$$ \phi^{-1}([a,b],0)=(i,p).$$
\end{remark}

Let $\kappa$ and $\sigma$ be subsets of $\Phi^{+}_n$ defined as
follows:
\begin{equation} \label{eq: kappa sigma}
\kappa \seteq \{  \beta \in \Phi^{+}_n \ | \ \phi^{-1}_1(\beta)=1
\}, \quad \sigma \seteq \{  \beta \in \Phi^{+} \ | \
\phi^{-1}_1(\beta)=n  \}.
\end{equation}
We enumerate the positive roots in
$\kappa=\{\kappa_1,\ldots,\kappa_r\}$ and
$\sigma=\{\sigma_1,\ldots,\sigma_s\}$ in the following way:
$$  \phi^{-1}_{2}(\kappa_{i+1})+2= \phi^{-1}_{2}(\kappa_{i}), \   \phi^{-1}_{2}(\sigma_{j+1})-2= \phi^{-1}_{2}(\sigma_{j}) \   \text{for } 1 \le i < r
\text{ and } 1 \le j <s.$$

\begin{corollary}  \label{Cor: first and last} 
\begin{enumerate}
\item[({\rm a})] If $\kappa_i=[a,b]$, then $\kappa_{i+1}=[b+1,c]$.
\item[({\rm b})] If $\sigma_j=[a,b]$, then $\sigma_{j+1}=[b+1,c]$.
\item[({\rm c})] $\sum_{i=1}^{r}\kappa_i=[1,n]=\sum_{j=1}^{s}\sigma_j$
\item[({\rm d})] If $\phi^{-1}_1([1,n])=k$, then $r=k=m_1$ , $s=n+1-k=m_n$ and hence $$\phi^{-1}([1,n])=(m_1,\xi_1-m_1+1)=(n+1-m_n,\xi_n-m_n+1).$$
\end{enumerate}
\end{corollary}

\begin{proof}
By Remark \ref{Alg: easy}, we have a subquiver in $\Gamma_Q$ which
can be described as follows:
$$
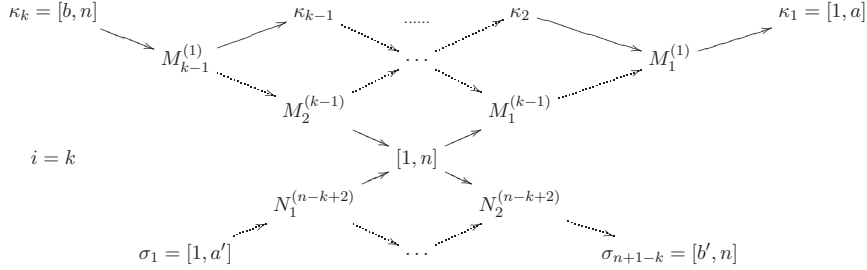
\scalebox{0.7}{\xymatrix@C=3ex@R=1ex{
\kappa_k=[b,n]\ar@{->}[dr] && \kappa_{k-1}\ar@{.>}[dr] \ar@{}[rr]_{\cdots\cdots} && \kappa_2\ar@{->}[dr] && \kappa_1=[1,a] \\
& M^{(1)}_{k-1}\ar@{.>}[dr]\ar@{->}[ur]  && \cdots\ar@{.>}[ur]\ar@{.>}[dr] && M^{(1)}_1\ar@{->}[ur] \\
&& M^{(k-1)}_{2}\ar@{.>}[ur]\ar@{->}[dr]&& M^{(k-1)}_1 \ar@{.>}[ur]\\
i=k&&& [1,n]\ar@{->}[ur]\ar@{->}[dr]  \\
&& N^{(n-k+2)}_{1}\ar@{->}[ur]\ar@{.>}[dr]&& N^{(n-k+2)}_2 \ar@{.>}[dr]\\
& \sigma_1=[1,a']\ar@{.>}[ur] && \cdots\ar@{.>}[ur] && \sigma_{n+1-k}=[b',n] \\
}}
$$
Then our assertions easily follow from observations on the given
subquiver.
\end{proof}

\section{The generalized quantum affine Schur-Weyl duality} \label{sec: SW-duality}
In this section, we briefly recall the backgrounds and theories on
the generalized quantum affine Schur-Weyl duality developed in
\cite{KKK13a,KKK13b}. Thus, for precise statements and definitions
referred in this section, see
\cite{AK,HL11,KKK13a,KKK13b,Kas02,Oh14}.

\subsection{Quantum affine algebras and the category $\mathscr{C}_Q$.}
Let $I = \{0,1, \ldots, n \}$ be a set of indices. An {\it affine
Cartan datum} is a quadruple $(\cm,\wl,\Pi,\Pi^\vee)$ consisting of
\begin{enumerate}
\item[({\rm a})] a matrix $\cm$ of corank $1$, called the {\it affine Cartan matrix} satisfying
$$ ({\rm i}) \ a_{ii}=2 \ (i \in I), \quad ({\rm ii}) \ a_{ij} \in \Z_{\le 0}, \quad  ({\rm iii}) \ a_{ij}=0 \text{ if } a_{ji}=0$$
with $\mathsf{D}= {\rm diag}(\mathsf{d}_i \in \Z_{>0}  \mid i \in
I)$ making $\mathsf{D}\cm$ symmetric,
\item[({\rm b})] a free abelian group $\wl$ of rank $n+2$, the {\it weight lattice},
\item[({\rm c})] an linearly independent set $\Pi = \{ \alpha_i \mid i\in I \} \subset \wl$, the set of {\it simple roots},
\item[({\rm d})] an linearly independent set $\Pi^{\vee} = \{ h_i \mid i\in I\} \subset \wl^{\vee} := \Hom( \wl, \Z )$, the set of {\it simple coroots},
\end{enumerate}
which satisfy
\begin{itemize}
\item[(1)] $\langle h_i, \alpha_j \rangle  = a_{ij}$ for all $i,j\in I$,
\item[(2)] for each $i \in I$, there exists $\Lambda_i \in \wl$ such that $\langle h_i, \Lambda_j \rangle =\delta_{ij}$ for all $j \in I$.
\end{itemize}
 We call
$\rl = \sum_{i\in I}  \Z \alpha_i$ the {\it root lattice}, and
$\rl^+ = \sum_{i\in I} \Z_{\ge 0} \alpha_i$ the {\it positive cone}
of the root lattice. For $\beta=\sum_{i \in I} k_i \alpha_i \in
\rl^+$, the {\it height} of $\beta$ is defined by $|\beta|=\sum_{i
\in I} k_i$. We denote by $\delta=\sum_{i \in I}\mathsf{a}_i h_i$
the {\it null root} and by $c=\sum_{i \in I} \mathsf{c}_ih_i$ the
{\it center} (\cite[Chapter 4]{Kac}).

Set $\h = \Q \tens_\Z \wl^\vee$. Then there exists a symmetric
bilinear form $( \ , \ )$ on $\h^*$ satisfying
$$ \langle h_i,\lambda \rangle = \dfrac{  2(\alpha_i,\lambda)}{(\alpha_i,\alpha_i)} \quad \text{ for any } i \in I \text{ and } \lambda \in \h^*.$$
We normalize the bilinear form by $ \langle c,\lambda \rangle =
(\delta,\lambda)$ for any $\lambda \in \h^*.$

We choose $0 \in I$ as the leftmost vertices in the tables in
\cite[pages 54, 55]{Kac} except $A^{(2)}_{2n}$-case in which we take
the longest simple root as $\alpha_0$.

We denote $\g$ by the affine Kac-Moody algebra associated with the
Cartan datum $(\cm,\wl,\Pi,\Pi^\vee)$ and $\g_0$ by the subalgebra
of $\g$ generated by $e_i$, $f_i$ and $K_i$ for $i \in I_0$. Note
that $\g_0$ is a finite-dimensional simple Lie algebra.

Let $\gamma$ be the smallest positive integer such that
$$ \gamma(\alpha_i,\alpha_i)/2 \in \Z \quad \text{ for any } i \in I.$$
Note that $(\alpha_i,\alpha_i)/2$ takes the rational number
$1,2,1/2,1/3$ when $\g=G^{(1)}_2$.

For  an indeterminate $q$, $m,n \in \Z_{\ge 0}$ and $i\in I$, we
define $q_i = q^{(\alpha_i,\alpha_i)/2}$ and
\begin{equation*}
 \begin{aligned}
 \ &[n]_i =\frac{ q^n_{i} - q^{-n}_{i} }{ q_{i} - q^{-1}_{i} },
 \ &[n]_i! = \prod^{n}_{k=1} [k]_i ,
 \ &\left[\begin{matrix}m \\ n\\ \end{matrix} \right]_i=  \frac{ [m]_i! }{[m-n]_i! [n]_i! }.
 \end{aligned}
\end{equation*}

Let $\ko$ be an algebraic closure of $\C(q)$ in $\cup_{m
>0}\C((q^{1/m}))$.

\begin{definition} \label{Def: GKM}
The {\em quantum affine algebra} $U_q(\g)$ associated with
$(\cm,\wl,\Pi,\Pi^{\vee})$ is the associative $\ko$-algebra
generated by $e_i,f_i$ $(i \in I)$ and $q^{h}$ $(h \in
\gamma^{-1}\wl^{\vee})$ satisfying following relations:
\begin{enumerate}
  \item  $q^0=1, q^{h} q^{h'}=q^{h+h'} $ for $ h,h' \in \gamma^{-1}\wl^{\vee},$
  \item  $q^{h}e_i q^{-h}= q^{\langle h, \alpha_i \rangle} e_i,
          \ q^{h}f_i q^{-h} = q^{-\langle h, \alpha_i \rangle }f_i$ for $h \in \gamma^{-1}\wl^{\vee}, i \in I$,
  \item  $e_if_j - f_je_i =  \delta_{ij} \dfrac{K_i -K^{-1}_i}{q_i- q^{-1}_i }, \ \ \text{ where } K_i=q_i^{ h_i},$
  \item  $\displaystyle \sum^{1-a_{ij}}_{k=0}
  (-1)^ke^{(1-a_{ij}-k)}_i e_j e^{(k)}_i =  \sum^{1-a_{ij}}_{k=0} (-1)^k
  f^{(1-a_{ij}-k)}_i f_jf^{(k)}_i=0 \quad \text{ for }  i \ne j, $
\end{enumerate}
where $e_i^{(k)}=e_i^k/[k]_i!$ and $f_i^{(k)}=f_i^k/[k]_i!$.
\end{definition}
We denote by $U_q'(\g)$ the subalgebra of $U_q(\g)$ generated by
$e_i,f_i,K^{\pm 1}_i$ $(i \in I)$ and call it also the {\it quantum
affine algebra}.

We call $\wl_{\cl} \seteq \wl/\Z\delta$ the {\it classical weight
lattice}. For $\cl: \wl \to \wl_\cl$ the canonical projection, we
have
$$ \wl_\cl = \soplus_{i \in I} \Z \cl(\Lambda_i) \quad \text{and}
\quad \wl^\vee_\cl\seteq \Hom(\wl_\cl,\Z)= \{ h \in \wl^\vee \mid
\langle h,\delta \rangle =0 \} = \soplus_{i \in I} \Z h_i.$$ Set
$\Pi_\cl=\cl(\Pi)$ and $\Pi_\cl^\vee=\{ h_0\,\cdots,h_n\}$. Then we
can regard $U_q'(\g)$ as the quantum affine algebra associated with
the {\it classical affine Cartan datum}
$(\cm,\wl_\cl,\Pi_\cl,\Pi^\vee_\cl)$.

Set
$$\varpi_i={\rm gcd}(\mathsf{c}_0,\mathsf{c}_i)^{-1}(\mathsf{c}_0\Lambda_i-\mathsf{c}_i\Lambda_0) \in \wl \quad \text{ for } i \in I_0.$$ Then
$\{ \cl(\varpi_i) \mid i \in I_0\}$ forms a basis of $\wl_\cl^0
\seteq \{ \lambda \in \wl_\cl \mid \langle c , \lambda \rangle =0
\}$. The Weyl group $W_0$ of $\g_0$ acts on $\wl^0_\cl$ (\cite[\S
1.2]{AK}).

For $i \in I_0$, there exists a unique simple $U_q'(\g)$-module
$V(\va_i)$ (up to an isomorphism), called the {\it $i$th fundamental
module}, satisfying the certain properties (\cite[\S 1.3]{AK}).
Moreover, there exist the left dual $V(\varpi_i)^*$ and the right
dual ${}^*V(\varpi_i)$ of $V(\varpi_i)$ with the following
$U_q'(\g)$-homomorphisms
\begin{equation} \label{eq: dual}
 V(\varpi_i)^* \otimes V(\varpi_i)  \overset{{\rm tr}}{\longrightarrow} \ko \quad \text{ and } \quad
V(\varpi_i) \otimes {}^*V(\varpi_i)  \overset{{\rm
tr}}{\longrightarrow} \ko
\end{equation}
where
\begin{equation} \label{eq: p star}
V(\varpi_i)^*\seteq  V(\varpi_{i^*})_{(p^*)^{-1}}, \
{}^*V(\varpi_i)\seteq  V(\varpi_{i^*})_{p^*} \ \  \text{ and } \ \
p^* \seteq (-1)^{\langle \rho^\vee ,\delta
\rangle}q^{(\rho,\delta)}.
\end{equation}
Here $\rho$ is defined by  $\langle h_i,\rho \rangle=1$ and
$\rho^\vee$ is defined by $\langle \rho^\vee,\alpha_i  \rangle=1$
for all $i \in I$ (see \cite[\S 1.3]{AK}).

We say that a $U_q'(\g)$-module $M$ is {\it good} if it has a {\it
bar involution}, a crystal basis with {\it simple crystal graph},
and a {\it global basis} (see \cite{Kas02} for precise definitions).
For instance, every fundamental module is a good module. Note that
every good module is a simple $U_q'(\g)$-module.


\begin{definition} \cite{HL11} (see also \cite[\S 3.3]{KKK13b}) \label{def: Q module category}
Let $Q$ be a Dynkin quiver of finite type $A_n$ (resp. $D_n$) and
$U_q'(\g)$ be the quantum affine algebra of type $A^{(1)}_{n}$
(resp. $D^{(1)}_{n}$). For any positive root $\beta$ contained in
$\Phi_n^+$ associated to $\g_0$, we set the $U_q'(\g)$-module
$V_Q(\beta)$ defined as follows:
\begin{align} \label{V_Q(beta)}
V_Q(\beta) \seteq V(\va_i)_{(-q)^{p}} \quad \text{ where } \quad
\phi^{-1}(\beta,0)=(i,p).
\end{align}
We define the smallest abelian full subcategory
$\mathscr{C}^{(1)}_Q$ consisting of finite dimensional integrable
$U_q'(\g)$-modules such that
\begin{itemize}
\item[({\rm a})] it is stable by taking 
subquotient, tensor product and extension,
\item[({\rm b})] it contains $V_Q(\beta)$ for all $\beta \in \Phi_n^+$.
\end{itemize}
\end{definition}

For the rest of this subsection, we briefly recall the morphisms
between $U_q'(A^{(t)}_{n})$-modules $(t=1,2)$
$$ V(\varpi_i)_a \tens V(\varpi_j)_b \quad \text{and}  \quad V(\varpi_k)_c.$$
These kinds of morphisms are known as {\it Dorey's type morphisms}
and well-studied for the classical untwisted quantum affine algebras
of types $A^{(1)}_{n}$, $B^{(1)}_{n}$, $C^{(1)}_{n}$ and
$D^{(1)}_{n}$ in \cite{CP96}. Recently, the author investigated such
morphisms for the classical twisted quantum affine algebras of types
$A^{(2)}_{2n-1}$, $A^{(2)}_{2n}$ and $D^{(2)}_{n+1}$ in \cite{Oh14}.

\begin{theorem} \cite[Theorem 6.1]{CP96} \label{Thm: Dorey A_n(1)}
Let $U_q'(\g)$ be the affine algebra of type $A^{(1)}_{n}$. For $1
\le i,j,k \le n$ and $a,b,c \in \Z$,
$$ \Hom_{U_q'(\g)}\big( V(\varpi_i)_{(-q)^a} \tens V(\varpi_j)_{(-q)^b},V(\varpi_k)_{(-q)^c} \big) \neq 0 $$
if and only if one of the following holds:
\begin{itemize}
\item[{\rm(i)}] $i+j \le n$, $i+j=k$, $a-c=-j$ and $b-c=i$,
\item[{\rm(ii)}] $i+j>n$, $k=i+j-n-1$, $c-a=-n-1+j$ and $b-a=n+1-i$.
\end{itemize}
\end{theorem}

In this paper, we denote the $D^{(2)}_3$ in \cite{Kac} by
$A^{(2)}_3$ with the following enumeration on Dynkin diagram:
$$\xymatrix@R=3ex{*{\circ}<3pt>  \ar@{<=}[r]_<{0 \  \  } & *{\circ}<3pt> \ar@{=>}[r]_<{2 \  \ } \ar@{}[r]_>{\,\,\  1}  & *{\circ}<3pt>}$$

\begin{theorem} \cite[Theorem 3.5, Theorem 3.9]{Oh14} 
\begin{enumerate}
\item[({\rm a})] For $i+j=k \le n$, there exists a surjective $U_q'(A^{(2)}_{2n-1})$-module $($resp. $U_q'(A^{(2)}_{2n})$-module$)$ homomorphism
\begin{align} \label{eq: p ij}
  p_{i,j} \colon V(\varpi_i)_{(-q)^{-j}} \otimes V(\varpi_j)_{(-q)^{i}} \twoheadrightarrow V(\varpi_{k}).
\end{align}
\item[({\rm b})] There exists a surjective $U_q'(A^{(2)}_{2n})$-module homomorphism
\begin{align} \label{eq: p 1nn}
  p_{1,n} \colon V(\varpi_n)_{(-q)^{-1}} \otimes V(\varpi_1)_{(-q)^{n}}
  \twoheadrightarrow V(\varpi_{n}).
  \end{align}
\end{enumerate}
\end{theorem}

\subsection{The denominators of normalized $R$-matrices.}
For good modules $M$ and $N$, we have an {\it intertwiner} between
$U_q'(\g)$-modules
\begin{equation} \label{Def: R(MN)} \Rnorm_{M,N}: M_\aff \otimes N_\aff \to \ko(z_M,z_N) \otimes_{\ko[z_M^{\pm 1},z_N^{\pm 1}]}
 N_\aff \otimes M_\aff \end{equation}
which satisfies
$$ \Rnorm_{M,N} \circ z_M = z_M \circ \Rnorm_{M,N}, \ \Rnorm_{M,N} \circ z_N = z_N \circ \Rnorm_{M,N} \
\text{ and } \ \Rnorm_{M,N}(\mathsf{v}_{M}\otimes
\mathsf{v}_{N})=\mathsf{v}_{N} \otimes \mathsf{v}_{M}.$$ Here
$M_\aff$ is the {\it affinization} of $M$ and $z_M$ is the
$U_q'(\g)$-automorphism of $M_\aff$ of weight $\delta$. We call
$\Rnorm_{M,N}$ the {\it normalized $R$-matrix}.

Note that $$\Rnorm_{M,N}(M_\aff \otimes N_\aff) \subset \ko(z_N/z_M)
\otimes_{\ko[(z_N/z_M)^{\pm 1}]}(N_\aff \otimes M_\aff) $$ and there
exists a unique monic polynomial $d_{M,N}(u) \in \ko[u]$ such that
\begin{equation}\label{Def: d(MN)}
d_{M,N}(z_N/z_M)\Rnorm_{M,N}(M_\aff \otimes N_\aff) \subset (N_\aff
\otimes M_\aff).
\end{equation}
We call $d_{M,N}(u)$ the {\it denominator} of $\Rnorm_{M,N}$.

\begin{theorem} \cite{AK,Kas02}  \label{Thm: basic properties} 
\begin{enumerate}
\item For good modules $M_1$ and $M_2$, the zeros of $d_{M_1,M_2}(z)$ belong to
$\C[[q^{1/m}]]\;q^{1/m}$ for some $m\in\Z_{>0}$.
\item $ V(\va_i)_{a_i} \otimes  V(\va_j)_{a_j}$ remains simple if and only if
$$ d_{i,j}(z)\seteq d_{V(\va_i),V(\va_j)}(z) $$ does not have
zeros at $z=a_i/a_j$ or $a_j/a_i$.
\item
Let $M$ be a finite dimensional simple integrable $U'_q(\g)$-module
$M$. Then, there exists a finite sequence $$\left( (i_1,a_1),\ldots,
(i_l,a_l)\right) \text{ in } (I_0\times \ko^\times)^l$$ such that
$d_{i_k,i_{k'}}(a_{k'}/a_k) \not=0$ for $1\le k<k'\le l$ and $M$ is
isomorphic to the head of
$\bigotimes_{i=1}^{l}V(\varpi_{i_k})_{a_k}$. Moreover, such a
sequence $\left((i_1,a_1),\ldots, (i_l,a_l)\right)$ is unique up to
permutation.
\end{enumerate}
\end{theorem}

Thus the information on zeros of denominators between fundamental
modules is crucial to investigate the finite dimensional integrable
$U_q'(\g)$-modules. The denominators between fundamental modules are
calculated in \cite{AK,DO94,KKK13b,Oh14} for all classical quantum
affine algebras.

\begin{theorem} \label{Thm; denominator} \cite{DO94,Oh14}
For $1\leq k,l \leq n$, we have
\begin{align}
& d_{k,l}(z)= \hspace{-3ex}  \displaystyle \prod_{s=1}^{ \min(k,l,n+1-k,n+1-l)} \hspace{-6ex} (z-(-q)^{2s+|k-l|}) \qquad\qquad\quad \ \ \  \text{ if $\g=A_n^{(1)}$}, \label{eq: dij A_n 1} \\
& d_{k,l}(z)= \hspace{-2ex} \displaystyle \prod_{s=1}^{\min(k,l)}
\hspace{-1ex} (z-(-q)^{2s+|k-l|})(z-(p^*)(-q)^{2s-k-l})  \text{ if
$\g=A_{2n-1}^{(2)}$ or $A_{2n}^{(2)}$}. \label{eq: dij A_n^2}
\end{align}
\end{theorem}

\subsection{Quiver Hecke algebras and their finite dimensional modules.}
We now recall the definition of the quiver Hecke algebra. Let $I$ be
an index set and $\ko$ be a field. For any symmetrizable Cartan
datum $(\cm,\wl,\Pi,\Pi^{\vee})$, $\mathcal{Q}_{i,j}(u,v)\in
\ko[u,v]$ $(i,j\in I)$ denote polynomials satisfying the following
conditions:
\begin{align} \label{eq: condition A}
\mathcal{Q}_{i,j}(u,v) = \left\{
                 \begin{array}{ll}
                   \sum_{(\alpha_i,\alpha_i)p+(\alpha_j,\alpha_j)q =-2(\alpha_i,\alpha_j)} t_{i,j;p,q} u^pv^q & \hbox{if } i \ne j,\\
                   0 & \hbox{if } i=j,
                 \end{array}
               \right.
\end{align}
where $t_{i,j;p,q}=t_{j,i;q,p} \in \ko$ and $t_{i,j;-a_{ij},0} \ne
0$. Thus we have $\mathcal{Q}_{i,j}(u,v) = \mathcal{Q}_{j,i}(v,u)$.
Note that the symmetric group on $m$-letters, $\mathfrak{S}_m=
\langle \mathfrak{s}_1,\mathfrak{s}_2,\ldots,\mathfrak{s}_{m-1}
\rangle$, acts on $I^m$ by place permutations.

\begin{definition} \cite{KL09,KL11,R08} \
The {\it quiver Hecke algebra} $R(m)$ associated with polynomials
$(\mathcal{Q}_{i,j}(u,v))_{i,j\in I}$ is the $\Z$-graded
$\ko$-algebra defined by three sets of generators
$$\{e(\nu) \mid \nu = (\nu_1,\ldots, \nu_m) \in I^m\}, \;\{x_k \mid 1 \le k \le m\}, \;\{\tau_l \mid 1 \le l \le m-1\} $$
satisfying the following relations:
\begin{align*}
& e(\nu) e(\nu') = \delta_{\nu,\nu'} e(\nu),\ \sum_{\nu \in I^{m}}
e(\nu)=1,\
x_k e(\nu) =  e(\nu) x_k, \  x_k x_l = x_l x_k, \allowdisplaybreaks \\
& \tau_l e(\nu) = e(\mathfrak{s}_l(\nu)) \tau_l,\  \tau_k \tau_l =
\tau_l \tau_k \text{ if } |k - l| > 1,
\  \tau_k^2 e(\nu) = \mathcal{Q}_{\nu_k, \nu_{k+1}}(x_k, x_{k+1}) e(\nu), \allowdisplaybreaks  \\[5pt]
&  (\tau_k x_l - x_{\mathfrak{s}_k(l)} \tau_k ) e(\nu) = \left\{
                                                           \begin{array}{ll}
                                                             -  e(\nu) & \hbox{if } l=k \text{ and } \nu_k = \nu_{k+1}, \\
                                                               e(\nu) & \hbox{if } l = k+1 \text{ and } \nu_k = \nu_{k+1},  \\
                                                             0 & \hbox{otherwise,}
                                                           \end{array}
                                                         \right. \\[5pt]
&( \tau_{k+1} \tau_{k} \tau_{k+1}\hspace{-0.3ex} - \hspace{-0.3ex}
\tau_{k} \tau_{k+1} \tau_{k} ) e(\nu) \hspace{-0.5ex}
=\hspace{-0.5ex}\delta_{\nu_{k},\nu_{k+2}}\hspace{-1ex}\dfrac{\mathcal{Q}_{\nu_k,\nu_{k+1}}(x_k,x_{k+1})
\hspace{-0.3ex}-\hspace{-0.3ex}
\mathcal{Q}_{\nu_k,\nu_{k+1}}(x_{k+2},x_{k+1})}{x_{k}-x_{k+2}}
e(\nu).
\end{align*}
\end{definition}

By assigning $\Z$-grading on generator as below, $R(m)$ becomes
$\Z$-graded:
\begin{align*}
\deg(e(\nu))=0, \quad \deg(x_k e(\nu))=
(\alpha_{\nu_k},\alpha_{\nu_k}), \quad  \deg(\tau_l
e(\nu))=-(\alpha_{\nu_l},\alpha_{\nu_{l+1}}).
\end{align*}

We denote the direct sum of the Grothendieck groups of the
categories $\Rep(R(m))$ of finite dimensional graded $R(m)$-modules
by
$$ [\Rep(R)] = \bigoplus_{ m \in \Z_{\ge 0}} [\Rep(R(m))]. $$
Note that $[\Rep(R)]$ has a free $\A \seteq \Z[q,q^{-1}]$-module
structure induced from the $\Z$-grading on $R(m)$, i.e. $(qM)_k =
M_{k-1}$ for a graded module $M = \bigoplus_{k \in \Z} M_k $.

For $m,n \in \Z_{\ge 0}$, let
$$
R(m) \otimes R(n) \to R(m+n)
$$
be the $\ko$-algebra homomorphism given by $e(\mu) \otimes e(\nu)
\mapsto e(\mu * \nu) \ (\mu \in I^m, \nu \in I^n)$, $x_k\otimes 1
\mapsto x_k \ (1 \le k \le m)$, $1\otimes x_k \mapsto x_{m+k} \ (1
\le k \le n)$, $\tau_k\otimes 1 \mapsto \tau_k \ (1 \le k <n)$,
$1\otimes \tau_k \mapsto \tau_{m+k} \ (1 \le k <n)$, where $\mu *
\nu $ denotes the concatenation of $\mu $ and $\nu$.

For an $R(m)$-module $M$ and an $R(n)$-module $N$, the induced
$R(m+n)$-module
$$
M \conv N \seteq R(m+n) \otimes_{R(m) \otimes R(n)} (M \otimes N)
$$
is called the \emph{convolution product of $M$ and $N$ }. For $M \in
\Rep(m)$, $M_k \in \Rep(m_k)$ ($k =1,\ldots, n$) and $r \in
\Z_{>0}$, let
\begin{align*}
M^{\conv 0} = \ko, \qquad M^{\conv r} = \underbrace{M\conv \cdots
\conv M}_r, \qquad \dct{k=1}{n} M_k = M_1 \conv \cdots \conv M_n.
\end{align*}

For $\beta \in \rl^+$ with $|\beta|=m$, We define $I^\beta$ and
$e(\beta)$ as follows:
$$ I^\beta = \{ \nu = (\nu_1,\ldots,\nu_m) \in I^m \ | \ \alpha_{\nu_1}+ \cdots + \alpha_{\nu_m}=\beta \} \quad \text{ and } \quad
e(\beta) = \sum_{\nu \in I^\beta} e(\nu).$$ Then $e(\beta)$ becomes
a central idempotent. We define the subalgebra $R(\beta)$ of $R(m)$
by
$$R(\beta) \seteq R(m)e(\beta)$$ and call it the {\it quiver Hecke algebra
at $\beta$}.

\begin{theorem} \cite{KL09,KL11,R08} For any symmetrizable
Cartan datum $(\cm,\wl,\Pi,\Pi^{\vee})$, take
$(\mathcal{Q}_{i,j}(u,v))_{i,j \in I}$ satisfying \eqref{eq:
condition A}. Then there exists an $\A$-algebra isomorphism
$$ U_\A^-(\g)^\vee \simeq  [\Rep(R)]$$
where the multiplication of $[\Rep(R)]$ is given by the {\it
convolution product}.
\end{theorem}

For a module $M$, We denote by ${\rm hd}(M)$ the {\it head} of $M$
and ${\rm soc}(M)$ the {\it socle} of $M$, respectively.

\begin{theorem} \cite[Theorem 4.7]{BKM12}\cite[Theorem 3.1]{Mc12} $($see also \cite{Kato12,KR09}$)$ \label{thm: BkMc}
For a finite simple Lie algebra $\g_0$, fix an equivalence class
$[\widetilde{w}_0]$ of reduced expression of $w_0$ and take any
reduced expression $\widetilde{w}_0$ in the class
$[\widetilde{w}_0]$. We fix a convex total order $\le_{\redez}$
induced from the reduced expression $\redez$ $($see \eqref{eq:
total}$)$.

Let $R$ be the quiver Hecke algebra corresponding to $\g_0$. For
each positive root $\beta \in \Phi^+_n$, there exists a simple
module $S_{\widetilde{w}_0}(\beta)$ such that
\begin{itemize}
\item[({\rm a})]  for all $m \in \Z_{\ge 0}$, $S_{\widetilde{w}_0}(\beta)^{\conv m}$ is simple,
\item[({\rm b})] for every $\underline{m}=(m_1,\cdots,m_{\mathsf{N}}) \in \Z_{\ge 0}^{\mathsf{N}}$, there exists a non-zero $R$-module homomorphism
\begin{align*}
&\rmat{\um} \colon \overset{\to}{S}_{\redez}(\um) \seteq
S_{\redez}(\beta_1)^{\conv m_1}\conv\cdots \conv
S_{\redez}(\beta_N)^{\conv m_N}\\
&\hs{20ex}{\Lto}\overset{\gets}{S}_{\redez}(\um) \seteq
S_{\redez}(\beta_N)^{\conv m_N}\conv\cdots \conv
S_{\redez}(\beta_1)^{\conv m_1}.
\end{align*}
and ${\rm Im}(\rmat{\um}) \simeq {\rm
hd}\left(\overset{\to}{S}_{\redez}(\um)\right) \simeq {\rm
soc}\left(\overset{\gets}{S}_{\redez}(\um)\right)$ is simple.
\item[({\rm c})] for every simple $R$-module $M$, there exists a unique sequence $\um \in \Z_{\ge 0}^{\mathsf{N}}$ such that
$M \simeq {\rm Im}(\rmat{\um}) \simeq {\rm hd}\big(
\overset{\to}{S}_{\redez}(\um) \big).$
\item[({\rm d})] for any minimal pair $(\beta_k,\beta_l)$ of $\beta_j=\beta_k+\beta_l$ with respect to the convex total $<_\redez$, there exists an exact sequence
$$ \qquad 0 \to S_{\redez}(\beta_j)
\to S_{\redez}(\beta_k) \conv S_{\redez}(\beta_l) \overset{{\mathbf
r}_{\um}}{\Lto} S_{\redez}(\beta_l) \conv S_{\redez}(\beta_k) \to
S_{\redez}(\beta_j) \to 0,$$ where $\um \in \Z_{\ge 0}^{\mathsf{N}}$
such that $m_k=m_l=1$ and $m_i=0$ for all $i \ne k,l$.
\end{itemize}
\end{theorem}

\begin{remark}
\begin{itemize}
\item[{\rm (i)}] In this paper, we use the module of $S_{\widetilde{w}_0}(\beta)$ in \cite{KKK13b}. Thus the statement {\rm (c)} of
the above theorem is different from \cite[Theorem 2.9]{BKM12} even
though we use the same convention $\le_{\redez}$ for the convex
total order induced from $\widetilde{w}_0$ (see. \cite[Remark
4.3.3]{KKK13b}).
\item[{\rm (ii)}] For a minimal pair $(\beta_k,\beta_l)$ of $\beta_j$ with respect to $<_\redez$, the statement {\rm (d)} of
the above theorem tells that the modules $S_{\redez}(\beta_k) \conv
S_{\redez}(\beta_l)$ and $S_{\redez}(\beta_l) \conv
S_{\redez}(\beta_k)$ have composition length $2$.
\item[{\rm (iii)}] For any pair of reduced expressions $\redez$ and
$\widetilde{w_0}'$ of $w_0$ which are adapted to $Q$, we have
$$S_{\redex}(\beta) \simeq S_{\widetilde{w_0}'}(\beta) \quad \text{ for all } \beta \in \Phi^+_{n},$$
by \eqref{eq: compati} and Theorem \ref{thm: BkMc} ({\rm d}). Thus
we denote by $S_Q(\beta)$ the simple $R(\beta)$-module
$S_{\redez}(\beta)$ for any reduced expression $\redez$ adapted to
$Q$.
\end{itemize}
\end{remark}

%

\begin{definition} We say the quiver Hecke algebra $R(\beta)$ at $\beta$ is {\it symmetric} if $\mathcal{Q}_{i,j}(u,v)$ is a polynomial
in $u-v$ for all $i,j \in {\rm supp}(\beta)$. Here $${\rm
supp}(\beta)\seteq \{ j \in I \ | \ a_j \ne 0 \text{ where } \beta =
\sum_{j \in I} a_j \alpha_j \}.$$
\end{definition}

\subsection{The generalized quantum affine Schur-Weyl duality functors.}
Now we shortly review the generalized quantum affine Schur-Weyl
duality functors which were studied in \cite{KKK13a,KKK13b}.
\medskip

Let $\mathcal{S}$ be an index set and $\{ V_s \}_{s \in
\mathcal{S}}$ be a family of good $U'_q(\g)$-modules indexed by
$\mathcal{S}$.

Assume that we have a triple $(J,X,s)$ consisting of an index set
$J$ and two maps $X:J \to \ko^\times, \ s: J \to \mathcal{S}$. For a
given $(J,X,s)$ and $\{ V_s \}_{s \in \mathcal{S}}$, we define a
quiver $Q^J=(Q^J_0,Q^J_1)$ by observing the order of zeros of
denominator as follows:
\begin{enumerate}
\item $Q^J_0= J$.
\item For $i,j \in J$, we put $\mathtt{d}_{ij}$ many arrows from $i$ to $j$, where $\mathtt{d}_{ij}$ is the order of the zero of
$d_{V_{s(i)},V_{s(j)}}(z_2/z_1)$ at $X(j)/X(i)$.
\end{enumerate}
For a quiver $Q^J$, we can associate
\begin{itemize}
\item a symmetric Cartan matrix $\cm^J=(a^J_{ij})_{i,j \in J}$ by
\begin{align} \label{eq: sym Cartan mtx}
a^J_{ij} =  2 \quad \text{ if } i =j, \quad a^J_{ij} =
-\mathtt{d}_{ij}-\mathtt{d}_{ji} \quad \text{ if } i \ne j,
\end{align}
\item the set of polynomials $(\mathcal{Q}^J_{i,j}(u,v))_{i,j \in I}$
$$
\mathcal{Q}^J_{i,j}(u,v)
=(u-v)^{\mathtt{d}_{ij}}(v-u)^{\mathtt{d}_{ji}} \quad \text{ if } i
\ne j.
$$
\end{itemize}
We denote by $R_J$ the symmetric quiver Hecke algebra associated
with $(\mathcal{Q}^J_{i,j}(u,v))_{i,j \in I}$.

\begin{theorem} \cite{KKK13a} \label{thm:gQASW duality} There exists a functor
$$\F : \Rep(R_J) \rightarrow \CC_\g $$
where $\CC_\g$ denotes the category of finite-dimensional integrable
$U_q'(\g)$-modules. Moreover, the functor enjoys the following
properties:
\begin{enumerate}
\item[({\rm a})] $\F$ is a tensor functor; that is, there exist $U_q'(\g)$-module isomorphisms
$$\F(R_J(0)) \simeq \ko \quad \text{ and } \quad \F(M_1 \circ M_2) \simeq \F(M_1) \tens \F(M_2)$$
for any $M_1, M_2 \in \Rep(R_J)$.
\item[({\rm b})] If the underlying graph of $Q^J$ is a Dynkin diagram of finite type $A$, $D$ or $E$, then
the functor $\F$ is exact.
\end{enumerate}
\end{theorem}
We call the functor $\F$ the {\it generalized quantum affine
Schur-Weyl duality functor}.

\begin{theorem} \cite{KKK13b} \label{thm:gQASW duality 2}
Let $\g$ be an affine Kac-Moody algebra of type $A^{(1)}_{n}$
$($resp. $D^{(1)}_{n})$ and $Q$ be a quiver whose underlying graph
is a Dynkin diagram is of finite type $A_n$ $($resp. $D_n)$. Set $$
J \seteq \left\{ (i,p) \in \Z Q \ | \ \phi(i,p) \in \Pi_n \times \{
0 \}  \right\},$$ where the map is given in \eqref{eq:phi}. We
define two maps $s: J \to \{ V(\varpi_i) \ | \ i \in I_0 \}$ and $X:
J \to  \ko^\times$ as
$$ s(i,p)=V(\varpi_i) \ \text{ and } \ X(i,p)= (-q)^p  \ \text{ for } (i,p) \in J.$$
\begin{itemize}
\item[({\rm a})] The underlying graph of the quiver $Q^J$ associated to $(J,X,s)$ is of finite type $A_{n}$ $($resp. $D_{n})$. Hence the functor
$$\F^{(1)}_Q: \Rep(R_J) \rightarrow \mathscr{C}^{(1)}_Q$$
in {\rm Theorem \ref{thm:gQASW duality}} is exact.
\item[({\rm b})] The functor $\F^{(1)}_Q$ sends simples to simples, bijectively. In particular, $\F^{(1)}_Q$ sends $S_Q(\beta)$ to $V_Q(\beta)$.
\end{itemize}
\end{theorem}

\begin{remark} \label{rem: alter}
In Theorem \ref{thm: comb proof}, we will provide an alternative
proof of ({\rm a}) in the above theorem by using {\it only} the
results in this paper. In \cite{KKK13b}, the authors of
\cite{KKK13b} used the result of \cite{HL11} to prove ({\rm a}) in
the above theorem.
\end{remark}

\section{Dorey's rules and minimal pair}
In this section, we prove that $\Gamma_Q$ encodes the information of
the Dorey's type morphisms for $U_q'(A^{(1)}_{n})$, and all pair
$(\alpha,\beta)$ of $\gamma \in \Phi^+_n$ are minimal with suitable
convex total orders, which is compatible with respect to the convex
partial order $\preceq_Q$.

\begin{definition} Let  $\beta$ be a vertex in $\Gamma_Q$. (Equivalently, $\beta \in \Phi^+_n$.)
\begin{enumerate}
\item[({\rm a})] The {\it upper ray of $\beta$} is the path consisting of one $S$-sectional path and one $N$-sectional path satisfying the following properties:
\begin{itemize}
\item  $\underbrace{S_1 \to \cdots S_{a-1} \to S_a}_{\text{$S$-sectional path}}=\beta
=\overbrace{N_{b} \to N_{b-1} \to  N_1}^{\text{$N$-sectional
path}}$,
\item there is no vertex $S_0 \in \Gamma_Q$ such that $S_0 \to S_1$ is an $S$-sectional path in $\Gamma_Q$,
\item there is no vertex $N_0 \in \Gamma_Q$ such that $N_1 \to N_0$ is an $N$-sectional path in $\Gamma_Q$.
\end{itemize}
\item[({\rm b})] The {\it lower ray of $\beta$} is the path consisting of one $S$-sectional path and one $N$-sectional path satisfying the following properties:
\begin{itemize}
\item  $\overbrace{N_1 \to \cdots N_{a-1} \to N_a}^{\text{$N$-sectional path}}=\beta
=\underbrace{S_{b} \to S_{b-1} \to  S_1}_{\text{$S$-sectional
path}}$,
\item there is no vertex $N_0 \in \Gamma_Q$ such that $N_0 \to N_1$ is an $N$-sectional path in $\Gamma_Q$,
\item there is no vertex $S_0 \in \Gamma_Q$ such that $S_1 \to S_0$ is an $S$-sectional path in $\Gamma_Q$.
\end{itemize}
\end{enumerate}
\end{definition}

In Example \ref{ex: example 1}, the upper ray and the lower ray of
$[1,4]$ can be described as follows:
$$ \scalebox{0.9}
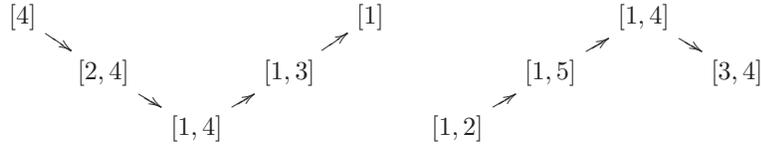
{\xymatrix@C=2ex@R=1ex{
 [4]\ar@{->}[dr] &&&& [1] \\
& [2,4]\ar@{->}[dr] && [1,3]\ar@{->}[ur] \\
&& [1,4]\ar@{->}[ur] \\
}} \quad \scalebox{0.9}{\xymatrix@C=2ex@R=1ex{
&&  [1,4]\ar@{->}[dr] \\
&  [1,5]\ar@{->}[ur] && [3,4] \\
 [1,2]\ar@{->}[ur]
}}
$$

\begin{theorem} \label{Thm: i_k+i_l=i_j}
For every pair $(\alpha,\beta)$ of $\alpha+\beta=\gamma\in
\Phi^+_n$, we write
$$\phi^{-1}_1(\alpha)=i_\alpha, \quad \phi^{-1}_1(\beta)=i_\beta \quad \text{ and } \quad \phi^{-1}_1(\gamma)=i_\gamma.$$
Then we have
$$ i_\alpha+i_\beta=i_\gamma  \quad \text{ or } \quad  (n+1-i_\alpha)+(n+1-i_\beta)=(n+1-i_\gamma).$$
\end{theorem}

\begin{proof}
We write $\gamma=[a,b]$. By Corollary \ref{Cor: one path}, the pair
$(\alpha,\beta)$ is contained in the upper ray of $[a,b]$ or the
lower ray of $[a,b]$. Assume that it is contained in the upper ray;
that is, there exists $a \le c \le b$ such that $\alpha=[a,c]$ and
$\beta=[c+1,b]$. Using \eqref{eq: known characterization}, Theorem
\ref{Thm: sectional} and Remark \ref{Alg: easy}, the situation in
$\Gamma_Q$  can be described as follows:
$$
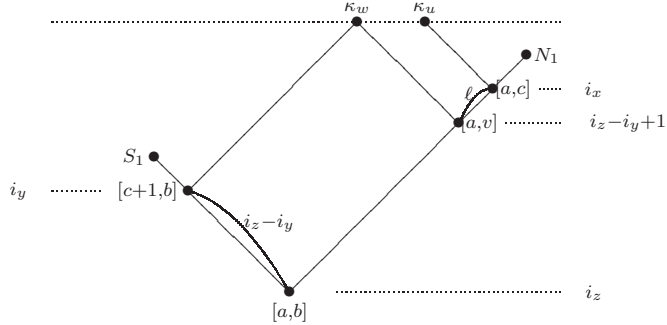
\scalebox{0.9}{{\xy (-15,0)*{}="T0";(20,0)*{}="T1"; (30,0)*{}="T2";
(45,0)*{}="T3"; (50,0)*{}="T6";(40,0)*{}="T5";(30,0)*{}="T4";
(60,0)*{}="T7"; (0,-20)*{}="B1";(5,-25)*{}="B2"; (20,-40)*{}="B3";
(55,-5)*{}="R1";(50,-10)*{}="R2";(45,-15)*{}="R3"; "T4"; "B2"
**\dir{-}; "B3"; "B1" **\dir{-};
"B1"*{\bullet};"B2"*{\bullet};"B3"*{\bullet}; "B3"; "R1" **\dir{-};
"R2"; "R3"; **\crv{(47.5,-9.5)}?(.6)+(-1.5,0)*{\scriptstyle \ell};
"B2"; "B3"; **\crv{(12.5,-27)}?(.5)+(4.5,0)*{\scriptstyle i_z-i_y};
"R1"*{\bullet};"R2"*{\bullet};"R3"*{\bullet};
"R1"+(3,0)*{\scriptstyle N_1}; "R2"+(3,0)*{\scriptstyle
[a,c]};"R2"+(7,0); (60,-10) **\dir{.};(65,-10)*{\scriptstyle i_x};
"R3"+(3,0)*{\scriptstyle [a,v]}; "R3"+(7,0); (60,-15) **\dir{.};
(70,-15)*{\scriptstyle i_z-i_y+1}; "B1"+(-3,0)*{\scriptstyle S_1};
"B2"+(-6,0)*{\scriptstyle [c+1,b]}; "B2"+(-13,0); (-15,-25)
**\dir{.}; (-20,-25)*{\scriptstyle i_y}; "B3"+(0,-3)*{\scriptstyle
[a,b]}; "B3"+(7,0); (60,-40) **\dir{.};(65,-40)*{\scriptstyle i_z};
"T4"+(0,2)*{\scriptstyle \kappa_w}; "T5"+(0,2)*{\scriptstyle
\kappa_u}; "T4"*{\bullet}; "T5"*{\bullet}; "T5"; "R2"
**\dir{-};"T4"; "R3" **\dir{-};"T0"; "T7" **\dir{.};
\endxy}}
$$
By Theorem \ref{Thm: sectional} and Remark \ref{Alg: easy},
$\kappa_u=[e,c]$ ($e \le c$) and $\kappa_w=[c+1,f]$ ($c+1 \le f$).
 Then Corollary \ref{Cor: first and last} (a)
tells that $w=u+1$ and hence $\ell=1$. Thus we can conclude
$$i_\alpha=i_\gamma-i_\beta+1-\ell=i_\gamma-i_\beta$$ which implies
our first assertion. If a pair $(\alpha,\beta)$ of $\gamma \in
\Phi^+_n$ is contained in the lower ray of $\beta_z$, then one can
prove our second assertion by applying a similar argument.
\end{proof}

\begin{remark} \label{rem: same edge}
 Using the above theorem, we can observe the following: For pairs $(\alpha,\beta)$, $(\alpha',\beta')$ such that
$\alpha+\beta=\alpha'+\beta'=\gamma \in \Phi_n^+$ and they are
contained in the upper ray (resp. lower ray) of $\gamma$ together,
the number of edges between $\alpha$ and $\beta$ in the upper ray
(resp. lower ray) coincides with the one between $\alpha'$ and
$\beta'$.
\end{remark}

Now we show that every pair $(\alpha,\beta)$ of
$\alpha+\beta=\gamma\in\Phi^+_n$ is {\it indeed a minimal pair with
respect to a suitable convex total order} compatible with the convex
partial order $\preceq_Q$.

\begin{theorem} \label{thm: minimal}
For every pair  $(\alpha,\beta)$ of $\alpha+\beta=\gamma$, there is
a convex total order $<$ such that
\begin{itemize}
\item[({\rm a})] it is compatible with the convex partial order $\preceq_Q$,
\item[({\rm b})] $(\alpha,\beta)$ is a minimal pair of $\gamma$ with respect to the convex total order.
\end{itemize}
\end{theorem}

\begin{proof}
As we observed, all pair $(\alpha,\beta)$ of $\gamma$ is located in
the upper ray of $\gamma$ or lower ray of $\gamma$. Assume that
$(\alpha,\beta)$ is in the upper ray. Take any pair
$(\alpha',\beta')$ of $\gamma$.

{\rm (i)} If $(\alpha',\beta')$ is also contained in the upper ray
of $\gamma$, the observation in Remark \ref{rem: same edge} tells
that we have one of the following two situations in $\Gamma_Q$:

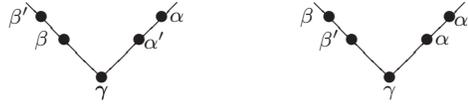
$$
{\xy (0,0)*{}="CB"; (-10,10)*{}="LT";(10,10)*{}="RT"; "CB"; "LT"
**\dir{-}; "CB"; "RT" **\dir{-}; "CB"*{\bullet};
"CB"+(0,-2)*{\scriptstyle \gamma}; "CB"+(0,-2)*{\scriptstyle
\gamma}; "CB"+(5,5)*{\bullet}; "CB"+(-5,5)*{\bullet};
"CB"+(8,8)*{\bullet}; "CB"+(-8,8)*{\bullet};
"CB"+(10,8)*{\scriptstyle \alpha}; "CB"+(7,5)*{\scriptstyle
\alpha'}; "CB"+(-11,8)*{\scriptstyle \beta'};
"CB"+(-8,5)*{\scriptstyle \beta};
\endxy}
\qquad\qquad {\xy (0,0)*{}="CB"; (-10,10)*{}="LT";(10,10)*{}="RT";
"CB"; "LT" **\dir{-}; "CB"; "RT" **\dir{-}; "CB"*{\bullet};
"CB"+(0,-2)*{\scriptstyle \gamma}; "CB"+(5,5)*{\bullet};
"CB"+(-5,5)*{\bullet}; "CB"+(8,8)*{\bullet}; "CB"+(-8,8)*{\bullet};
"CB"+(10,8)*{\scriptstyle \alpha'}; "CB"+(7,5)*{\scriptstyle
\alpha}; "CB"+(-11,8)*{\scriptstyle \beta};
"CB"+(-8,5)*{\scriptstyle \beta'};
\endxy}
$$
Thus, with respect to the total convex order $<^U_Q$ in Remark
\ref{Rmk: total orders}, we have
$$  \alpha <^U_Q \alpha' <^U_Q \gamma <^U_Q \beta <^U_Q \beta' \text{ or }  \alpha' <^U_Q \alpha <^U_Q \gamma <^U_Q \beta' <^U_Q \beta.$$

{\rm (ii)} Assume that $(\alpha',\beta')$ is contained in the lower
ray of $\gamma$. Then the situation can be expressed as follows:
$$
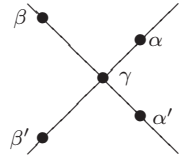
{\xy (0,0)*{}="CB"; (-10,10)*{}="LT";(10,10)*{}="RT";
(-10,-10)*{}="BLT";(10,-10)*{}="BRT"; "CB"; "LT" **\dir{-}; "CB";
"RT" **\dir{-}; "CB"; "BLT" **\dir{-}; "CB"; "BRT" **\dir{-};
"CB"*{\bullet}; "CB"+(3,0)*{\scriptstyle \gamma};
"CB"+(5,5)*{\bullet}; "CB"+(-8,8)*{\bullet};
"CB"+(7,5)*{\scriptstyle \alpha}; "CB"+(-11,8)*{\scriptstyle \beta};
"CB"+(-8,-8)*{\bullet};
"CB"+(5,-5)*{\bullet};"CB"+(8,-5)*{\scriptstyle \alpha'};
"CB"+(-11,-8)*{\scriptstyle \beta'};
\endxy}
$$
Hence, with respect to the total convex order $<^U_Q$ in Remark
\ref{Rmk: total orders}, we have
$$  \alpha' <^U_Q \alpha <^U_Q \gamma <^U_Q \beta' <^U_Q \beta.$$
By {\rm (i)} and {\rm (ii)}, we can conclude that every pair
$(\alpha,\beta)$ in the upper ray becomes minimal with respect to
the total convex order $<^U_Q$. In the similar way, one can see that
every pair $(\alpha,\beta)$ in the lower ray becomes minimal with
respect to the total convex order $<^L_Q$.
\end{proof}

For the rest of this paper, we say that {\it a pair $(\alpha,\beta)$
is minimal} when $\alpha+\beta \in \Phi_n^+$ and the pair is a
minimal pair with respect to a suitable convex total order which is
induced by some reduced expression $\redez$ adapted to $Q$.

\begin{corollary} \label{Cor: Dorey on Gamma}
\begin{enumerate}
\item[({\rm a})] For every pair $(\alpha,\beta)$ of $\alpha+\beta=\gamma\in\Phi^+_n$, we have a surjection
\begin{align}\label{eq: surj R}
S_Q(\beta) \conv S_Q(\alpha) \twoheadrightarrow S_Q(\gamma) \quad
\text{ in } \Rep(R_{A_n}).
\end{align}
Hence we have a surjection
\begin{align}\label{eq: surj C}
V_Q(\beta) \otimes V_Q(\alpha) \twoheadrightarrow V_Q(\gamma) \quad
\text{ in } \mathscr{C}^{(1)}_Q.
\end{align}
\item[({\rm b})] For every positive root $\beta \in \Phi^+_n$, there exists a sequence $(i_1,i_2,\cdots,i_{|\beta|}) \in I_0^{|\beta|}$ such that
$$\stens_{t=1}^{|\beta|} V_Q(\alpha_{i_t})  \twoheadrightarrow V_Q(\beta).$$
\end{enumerate}
\end{corollary}

\begin{proof}
Note that for a minimal pair $(\alpha,\beta)$ of
$\alpha+\beta=\gamma$ with respect to some convex total order
$\le_{\redez}$, Theorem \ref{thm: BkMc} ({\rm d}) guarantees the
existence of \eqref{eq: surj R} in $\Rep(R_{A_n})$.

On the other hand, Theorem \ref{thm: minimal} tells that, for every
pair $(\alpha,\beta)$ of $\gamma$, there exists a reduced expression
$\redez'$ adapted to $Q$ such that the pair is minimal with respect
to the total order $\le_{\redez'}$. Thus our first assertion for
$\Rep(R_{A_n})$ follows.

Recall that the functor $\mathcal{F}^{(1)}_Q$ sends simples to
simples. Thus $\mathcal{F}^{(1)}_Q\big({\rm Im} (f) \big)$ is
non-zero for any non-zero homomorphism $f$ in $\Rep(R_{A_n})$; i.e.,
$\mathcal{F}^{(1)}_Q$ is faithful. Then our first assertion for
$\mathscr{C}^{(1)}_Q$ follows since $\mathcal{F}^{(1)}_Q$ is a
tensor functor sending $S_Q(\beta)$ to $V_Q(\beta)$. The existence
of \eqref{eq: surj C} is also guaranteed by Theorem \ref{Thm: Dorey
A_n(1)} and Theorem \ref{Thm: i_k+i_l=i_j}.

The second assertion follows from the first assertion.
\end{proof}

From Corollary \ref{Cor: Dorey on Gamma} (b), the condition (b) in
Definition \ref{def: Q module category} can be modified as follows:

\begin{itemize}
\item[({\rm b}$'$)] It contains $V_Q(\alpha_k)$ for all $\alpha_k \in \Pi_n$.
\end{itemize}

\begin{corollary}For $\phi^{-1}(\alpha,0)=(i,a),\phi^{-1}(\beta,0)=(j,b)$, assume that $(k,c) \in \Gamma_Q$ such that
the triple $\{ (i,a),(j,b),(k,c)\}$ satisfies the one of the
condition {\rm (i)} and {\rm (ii)} in Theorem \ref{Thm: Dorey
A_n(1)}. Then we have
$$\gamma\seteq \phi(k,c)=\alpha+\beta.$$
\end{corollary}

\begin{proof}
By assumption, $\alpha,\beta$ are contained in the upper ray or
lower ray of $\gamma$, respectively. Applying the argument in
Theorem \ref{Thm: i_k+i_l=i_j}, one can prove the assertion.
\end{proof}

Recall Example \ref{ex: example 1}.
\begin{enumerate}
\item[({\rm i})] Using the reading in Remark \ref{Rmk: total orders} (A), the reduced expression $\redez$ of $w_0$ associated to the reading is
$$\redez=s_1s_3s_2s_1s_4s_3s_2s_1s_5s_4s_3s_2s_1s_5s_4$$ and the convex total order on $\Phi_5^+$ is given as follows:
\begin{align*}
&
[1]<_\redez[3]<_\redez[1,3]<_\redez[2,3]<_\redez[3,4]<_\redez[1,4]<_\redez[2,4]<_\redez[4]
\\ & \hspace{7ex}
<_\redez[3,5]<_\redez[1,5]<_\redez[2,5]<_\redez[4,5]<_\redez[5]<_\redez[1,2]<_\redez[2].
\end{align*}
\noindent With this convex total order, one can observe that every
pair in lower rays becomes minimal.
\item[({\rm ii})] Using the reading in Remark \ref{Rmk: total orders} (B), the reduced expression $\redez'$ of $w_0$ associated to the reading is
$$\redez'=s_3s_4s_5s_1s_2s_3s_4s_5s_1s_2s_3s_4s_1s_2s_1$$ and the convex total order on $\Phi_5^+$ is given as follows:
\begin{align*}
& [3]<_{\redez'}[3,4]<_{\redez'}[3,5]<_{\redez'}[1]<_{\redez'}[1,3]<_{\redez'}[1,4]<_{\redez'}[1,5]<_{\redez'}[1,2] \\
& \hspace{10ex}
<_{\redez'}[2,3]<_{\redez'}[2,4]<_{\redez'}[2,5]<_{\redez'}[2]<_{\redez'}[4]<_{\redez'}[4,5]<_{\redez'}[5].
\end{align*}
\noindent With this convex total order, one can observe that every
pair in upper rays becomes minimal.
\end{enumerate}

\begin{corollary} \label{cor: length2}
For every pair $(\alpha,\beta)$ of  $\alpha+\beta \in \Phi^+$,
$$V_Q(\beta)\otimes V_Q(\alpha)\text{ and } S_Q(\beta) \conv S_Q(\alpha) \text{ are of length $2$}.$$
\end{corollary}

\begin{proof}
Since every pair $(\alpha,\beta)$ is minimal, our assertions follow
from Theorem \ref{thm: BkMc} (d) and Theorem \ref{thm:gQASW duality
2} (b).
\end{proof}

\begin{remark} \label{rmk: non-adapted}
For a reduced expression of the longest element $w_0$ of $A_5$
$$\widetilde{w}_0=s_1s_2s_3s_5s_4s_3s_1s_2s_3s_5s_4s_3s_1s_2s_3,$$
one can easily check that it is {\it not} adapted to any Dynkin
quiver $Q$ of type $A_5$. The convex total order induced by
$\widetilde{w}_0$ is given as follows:
\begin{align*}
& [1]<_\redez[1,2]<_\redez[1,3]<_\redez[5]<_\redez[1,5]<_\redez[4,5]<_\redez[2]<_\redez[2,5] \\
& \hspace{8ex}
<_\redez[2,3]<_\redez[1,4]<_\redez[2,4]<_\redez[4]<_\redez[3,5]<_\redez[3,4]<_\redez[3].
\end{align*}
Thus the pair $([1,2],[3,5])$ is not a minimal pair of $[1,5]$ with
respect to $<_\redez$, since we have
$$[1,2]<_\redez [1,3] <_\redez [1,5] <_\redez [4,5] <_\redez [3,5].$$

One can compute that
$$ S_{\widetilde{w}_0}([1,2])=\mathsf{L}(21), \ S_{\widetilde{w}_0}([1,3])=\mathsf{L}(321),
\ S_{\widetilde{w}_0}([4,5])=\mathsf{L}(45), \
S_{\widetilde{w}_0}([3,5])=\mathsf{L}(345)$$ where $\mathsf{L}(21)$,
$\mathsf{L}(321)$, $\mathsf{L}(45)$ and $\mathsf{L}(345)$ are
$1$-dimensional modules which are defined in the similar way of
\eqref{def: L(a,b)} below. Moreover, we have
\begin{itemize}
\item $\mathsf{L}(345) \conv \mathsf{L}(21)$ and
$\mathsf{L}(45) \conv \mathsf{L}(321)$ have composition length $2$,
\item ${\rm hd}\big( \mathsf{L}(345) \conv \mathsf{L}(21) \big)
\simeq {\rm soc}\big( \mathsf{L}(45) \conv \mathsf{L}(321) \big)$
are simple.
\end{itemize}
Hence we can conclude that $${\rm hd}\big( \mathsf{L}(345) \conv
\mathsf{L}(21) \big) \not\simeq {\rm hd}\big( \mathsf{L}(45) \conv
\mathsf{L}(321) \big) \simeq S_{\widetilde{w}_0}([1,5]),$$ by
\cite[Corollary 3.9]{KKKO14}.
\end{remark}

\section{Revisit of construction of exact sequences and simple modules.} \label{sec: revisit}
In this section, we revisit the theories on construction of exact
sequences and simple modules of $R_{A_n}$ developed in
\cite{BKOP,BKM12,HMM09,KKK13b,KKKO14,KP11,KR09,Mc12} by using the
explicit description of $\Gamma_Q$ in Section \ref{sec:
Combinatorics}.

\medskip

For a $\beta=[a,b] \in \Phi_n^+$, we define a $1$-dimensional simple
$R_{A_n}(\beta)$-module $L(a,b)=\ko u(a,b)$ as follows:
\begin{equation}
\begin{aligned} \label{def: L(a,b)}
& x_m u{(a,b)} =0 \ (1 \le m \le b-a+1) , \\ & \tau_k u{(a,b)} =0 \ (1 \le k < b-a+1), \\
& e(\nu) u{(a,b)}= \begin{cases}
u{(a,b)} & \text{if $\nu=(a,a+1, \ldots, b)$,} \\
0 & \text{otherwise.}
\end{cases}
\end{aligned}
\end{equation}

\begin{lemma} \label{lem: Q rev}
Fix a Dynkin quiver $\overset{\gets}{Q}$ of type $A_n$, called the
\emph{linear quiver}, as follows:
\begin{align} \label{def: linear quiver}
\xymatrix@R=3ex{ *{ \bullet }<3pt> \ar@{<-}[r]_<{1}
&*{\bullet}<3pt> \ar@{<-}[r]_<{2}  &*{ \bullet }<3pt>
\ar@{<.}[rr]_<{3} &&*{ \bullet }<3pt> \ar@{<-}[r]_<{n-1}  & *{
\bullet }<3pt> \ar@{-}[l]^<{\ \ n} }
\end{align}

For $\beta=(a,b) \in \Phi_n^+$, we have
$$ S_{\overset{\gets}{Q}}(\beta) \simeq L(a,b).$$
\end{lemma}
\begin{proof}
By Lemma \ref{Lem: Key Lem} (a), we have
$$ \phi^{-1}(\alpha_k,0)=(n,\xi_n-2(n-k)) \quad \text{ for all } k \in I_0.$$
As is well-known, the shape of $\Gamma_{\overset{\gets}{Q}}$ is
given as follows:
$$ \scalebox{0.9}
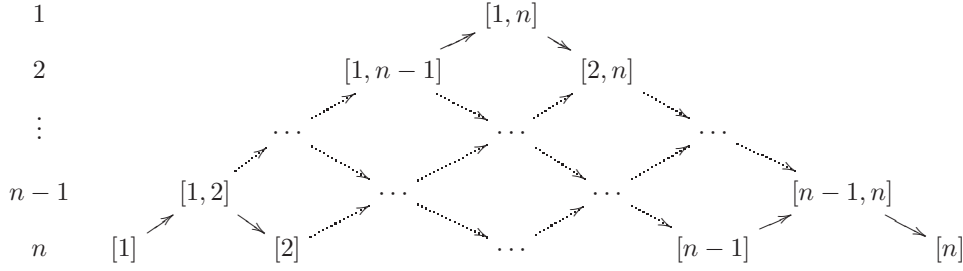
{\xymatrix@C=2ex@R=1ex{
1&&&&& [1,n]\ar@{->}[dr]\\
2&&&& [1,n-1]\ar@{.>}[dr]\ar@{->}[ur] && [2,n]\ar@{.>}[dr]\\
\vdots&&& \cdots\ar@{.>}[dr]\ar@{.>}[ur] && \cdots\ar@{.>}[dr]\ar@{.>}[ur] && \cdots\ar@{.>}[dr]  \\
n-1&& [1,2] \ar@{->}[dr]\ar@{.>}[ur], && \cdots\ar@{.>}[dr]\ar@{.>}[ur] && \cdots\ar@{.>}[ur]\ar@{.>}[dr] && [n-1,n]\ar@{->}[dr] \\
n&[1]\ar@{->}[ur]&& [2]\ar@{.>}[ur] && \cdots \ar@{.>}[ur] && [n-1]
\ar@{->}[ur] && [n] }}
$$

For a simple root $\beta$, it is trivial. By an induction hypothesis
on a height $|\beta|$ and Corollary \ref{Cor: Dorey on Gamma}, we
have $(c+1,b) < (a,c)$ and a surjection
$$ L(a,c) \conv L(c+1,b) \twoheadrightarrow S_{\overset{\gets}{Q}}(\beta) \quad \text{ for $\beta =(a,b) \in \Phi_n^+$ with $|\beta| \ge 2$}. $$
Since there is a canonical surjection
$$ L(a,c) \conv L(c+1,b) \twoheadrightarrow L(a,b),$$
and $L(a,c) \conv L(c+1,b)$ has a unique simple head, our assertion
follows.
\end{proof}

\begin{lemma} \cite{KP11} \label{Lem: irr cond}
For $[a_1,b_1]$, $[a_2,b_2] \in \Phi_n^+$ with $a_1=a_2$ or
$b_1=b_2$,
$$ \text{$L(a_1,b_1) \conv L(a_2,b_2) \simeq L(a_2,b_2) \conv L(a_1,b_1)$ is simple}.$$
\end{lemma}

Now the following proposition can be regarded as a generalization of
the above lemma.

\begin{proposition} \label{prop: switching}
For any Dynkin quiver $Q$ of type $A_n$ and $\alpha=[a_1,b_1]$,
$\beta=[a_2,b_2] \in \Phi_n^+$ with $a_1=a_2$ or $b_1=b_2$,
$$ \text{ $S_Q(\alpha) \conv S_Q(\beta) \simeq S_Q(\beta) \conv S_Q(\alpha)$ is simple.} $$
\end{proposition}

\begin{proof}
By Remark \ref{Alg: easy}, we have
$$ |\phi^{-1}_2(\alpha) -\phi^{-1}_2(\beta)| = |\phi^{-1}_1(\alpha) -\phi^{-1}_1(\beta)|.$$
By \eqref{eq: dij A_n 1} and Theorem \ref{Thm: basic properties}, we
have
$$ \text{ $V_Q(\alpha) \otimes V_Q(\beta) \simeq V_Q(\beta) \otimes V_Q(\alpha)$ is simple.} $$
Thus our assertion follows from Theorem \ref{thm:gQASW duality 2}.
\end{proof}

\begin{corollary} \label{cor: switching}
For any Dynkin quiver $Q$ of type $A_n$ and  $\beta_k=[a_k,b_k] \in
\Phi_n^+$ $(k \in \Z_{>0})$, assume that $a_k = a_l$ $($resp.
$b_k=b_l)$ for all $ k \ne l $. Then we have
$$S_Q(\beta_1) \conv S_Q(\beta_2) \conv \cdots \conv S_Q(\beta_k) \text{ is simple.}$$
\end{corollary}

\begin{proposition} \label{prop: redu 1}
Let $Q$ be any Dynkin quiver of finite type $A_n$.
\begin{enumerate}
\item[{\rm (a)}] For a source or a sink $i$ in $Q$, and $a \le i \le b$,
$$ S_Q([a,i]) \conv S_Q([i,b]) \simeq  S_Q([i,b])  \conv S_Q([a,i]) \text{ is simple.} $$
\item[{\rm (b)}] For a right intermediate or a left intermediate $i$ in $Q$, and $a < i < b$
$$ S_Q([a,i]) \conv S_Q([i,b]) \text{ and }  S_Q([i,b])  \conv S_Q([a,i]) \text{ are reducible.} $$
\end{enumerate}
\end{proposition}

\begin{proof}
(a) Write $\phi^{-1}([i,b])=(k,p)$ and $\phi^{-1}([a,i])=(l,q)$. By
Remark \ref{Alg: easy} {\rm (b-i)} and {\rm (b-ii)}, we have
$$|k-l| \ge |p-q|.$$
Hence, by \eqref{eq: dij A_n 1}, $d_{k,l}(z)$ does not have a zero
at $(-q)^{|p-q|}$. Thus
$$ V_Q([a,i]) \otimes V_Q([i,b]) \simeq  V_Q([i,b])  \otimes V_Q([a,i])\text{ is simple} ,$$
which yields our first assertion.

\noindent (b) Assume that $i$ is a left intermediate. By Remark
\ref{Alg: easy}, all positive root of the form $[i,b]$ or $[a,i]$
are contained in the following hook in $\Gamma_Q$:
$$\overbrace{N_{n-i+1} \to \cdots N_{2} \to N_{1}}^{\text{maximal
$(N,i)$-sectional path}} =[i] = \underbrace{S_{1} \to S_{2} \to
\cdots \to S_i}_{\text{maximal $(S,i)$-sectional path}}.$$ Thus we
can assume that $[i,b]=N_{k}$ $(2 \le k \le n-i+1)$ and
$[a,i]=S_{l}$ $(2 \le l \le i)$. Then we have
$$\phi^{-1}_1(N_k)=k, \quad  \phi^{-1}_1(S_l)=l \quad \text{and} \quad \phi^{-1}_2(S_l)-\phi^{-1}_2(N_k)=k+l-2.$$
In this case, one can check that
\begin{itemize}
\item $\min \{ k,l,n+1-k,n+1-k \} = \min \{ k ,l \}$,
\item $|k-l|+2s=k+l-2$ when $1 \le s = \min \{k, l \} -1$.
\end{itemize}
 Thus $$ V_Q([a,i]) \otimes V_Q([i,b]) \text{ and }  V_Q([i,b])  \otimes V_Q([a,i])\text{ are reducible},$$
which yields our second assertion. The case for a right intermediate
$i$, we can prove by applying the same argument.
\end{proof}

\begin{corollary}
Let $Q$ be Dynkin quiver of finite type $A_n$ and $i$ be a source or
a sink in $Q$. For $\beta_k=[a_k,b_k] \in \Phi_n^+$ $(k \in
\Z_{>0})$, assume that $a_k = i$ or $b_k = i$ for all $k$. Then we
have
$$S_Q(\beta_1) \conv S_Q(\beta_2) \conv \cdots \conv S_Q(\beta_k) \text{ is simple.}$$
\end{corollary}

\begin{lemma} \cite{KP11} \label{Lem: far}
For $[a_1,b_1]$, $[a_2,b_2] \in \Phi_n^+$ with $b_1 < a_2-1$,
$$ \text{$L(a_1,b_1) \conv L(a_2,b_2) \simeq L(a_2,b_2) \conv L(a_1,b_1)$ is simple}.$$
\end{lemma}

Now the following proposition can be regarded as a generalization of
the above lemma.

\begin{proposition} \label{prop: far}
For any Dynkin quiver $Q$ of finite type $A_n$ and $[a,b], [a',b']
\in \Phi_n^+$,  assume that $ a < b < a'-1 < b'$. Then we have
$$ S_Q([a,b]) \conv S_Q([a',b']) \simeq  S_Q([a',b'])  \conv S_Q([a,b]) \text{ is simple.} $$
\end{proposition}

\begin{proof}
Note that $[a,b'] \in \Phi_n^+$ and there is no $[a',b] \not \in
\Phi_n^+$. Thus
\begin{itemize}
\item $[a,b']$ is located at the intersection of the maximal $(N,a)$-sectional path and the maximal $(S,b')$-sectional path,
\item there is no intersection between the maximal $(S,a)$-sectional path and the maximal $(N,b')$-sectional path.
\end{itemize}
Then the situations can be described as follows:
$$
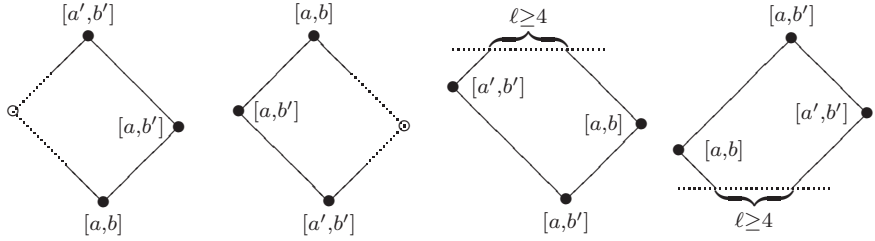
{\xy (-35,0)*++{\Lray}; (-5,0)*++{\Rray}; (25,0)*++{\Uray}; (55,0)*++{\Dray}; \endxy}$$

Write $\phi^{-1}([a,b])=(i,p)$, $\phi^{-1}([a',b'])=(j,q)$ and
$\phi^{-1}([a,b'])=(k,r)$. For the first and second cases, one can
easily notice that $|i-j|>|p-q|$. Thus the assertion follows from
\eqref{eq: dij A_n 1}.

Assume the third case. Then it is enough to show that, for any $1
\le s \le \min \{ i,j,n+1-i,n+1-j \}$,
\begin{equation} \label{eq: zero cond}
|i-j|+2s \ne 2k-i-j.
\end{equation}
Set $v=\max \{ i, j \}$ and $w =  \min \{ i,j,n+1-i,n+1-j \}$. Then
we have
$$|i-j|+2s = 2k-i-j \quad \text{ if } s=k-v.$$
However, the fact $k>i+j$ implies $k-v> w$. Thus the assertion
holds. For the fourth case, one can prove by applying the similar
argument of the third case.
\end{proof}

Now, we have an alternative proof of Theorem \ref{thm:gQASW duality
2} (a) as we emphasized in Remark \ref{rem: alter}:

\begin{theorem} \label{thm: comb proof}
With the same choice of $(J,X,s)$ in {\rm Theorem \ref{thm:gQASW
duality 2}}, the underlying graph of $Q^J$ is of finite type $A_n$.
\end{theorem}
\begin{proof}
By Corollary \ref{Cor: Dorey on Gamma}, there exists a surjective
morphism
$$ V_Q(\alpha_{k}) \otimes V_Q(\alpha_{k+1}) \twoheadrightarrow V_Q(\alpha_{k}+\alpha_{k+1}) \ \text{ or } \
V_Q(\alpha_{k+1}) \otimes V_Q(\alpha_{k}) \twoheadrightarrow
V_Q(\alpha_{k}+\alpha_{k+1}) $$ and hence $ V_Q(\alpha_{k}) \otimes
V_Q(\alpha_{k+1})$ is reducible.

On the other hand, Proposition \ref{prop: far} tells that
$$V_Q(\alpha_{k}) \otimes V_Q(\alpha_{l}) \simeq V_Q(\alpha_{l}) \otimes V_Q(\alpha_{k}) \text{ is simple for $|k-l|>1$.}$$
By Theorem \ref{Thm: basic properties} and the fact that
$d_{k,l}(z)$ in \eqref{eq: dij A_n 1} has only zeros of multiplicity
$1$, our assertion follows.
\end{proof}

\begin{definition}
A simple $R(\beta_k)$-module $M$ is called {\it real} if $M \conv M$
is again simple.
\end{definition}

For example, every $S_Q(\beta)$ is real by Theorem \ref{thm: BkMc}
{\rm (a)}. Now, we recall one of the main results in \cite{KKKO14}:

\begin{theorem} \cite{KKKO14} $($see also \cite{KKKO14b}$)$ \label{Thm: length 2 KKKO}
Let $M_k$ be a simple $R(\beta_k)$-module $(k=1,2)$ and assume one
of $M_1$ and $M_2$ is real, and
$\Rnorm_{\mathcal{F}^{(1)}_Q(M_1),\mathcal{F}^{(1)}_Q(M_2)}(z)$ has
a simple pole at $z=1$. Then $M_1 \conv  M_2$ has composition length
$2$. In particular, there exists an exact sequence
$$    0 \to N  \to M_1 \conv  M_2 \overset{ \rmat{M_1,M_2}}{\longrightarrow} M_2 \conv  M_1 \to N  \to 0, $$
where $\rmat{M_1,M_2}$ is non-zero,
 ${\rm Im}(\rmat{M_1,M_2})$ is the unique simple socle of $M_2 \conv  M_1$ and
$N$ is the unique simple head of $M_2 \conv  M_1$ $($up to
isomorphisms$)$.
\end{theorem}

Note that the statements in the rest of this section also hold when
we replace the module $S_Q(\beta)$ (resp. $M \in \Rep(R_{A_n})$) in
the statements with $V_Q(\beta)$ (resp. $M \in
\mathscr{C}^{(1)}_Q$), by applying the exact functor
$\mathcal{F}_Q^{(1)}$.

\begin{proposition}\label{prop: big mesh} Let $Q$ be any Dynkin quiver of finite type $A_n$.
Assume we have the following rectangle subquiver in $\Gamma_Q
\colon$
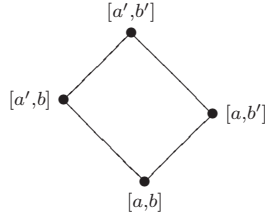
$$\scalebox{0.9}{{\xy (0,10)*++{\rectangle}; \endxy}}$$
for $a\ne a'$ and $b \ne b'$. Then there exists a short exact
sequence
\begin{equation}\label{eq: rec quiver}
\begin{aligned}
   & 0 \to S_Q([a,b]) \conv S_Q([a',b'])  \to  S_Q([a,b']) \conv S_Q([a',b]) \\
   &  \hspace{10ex}\overset{\rmat{}}{\longrightarrow}
   S_Q([a',b]) \conv S_Q([a,b']) \to S_Q([a,b]) \conv S_Q([a',b'])  \to 0,
\end{aligned}
\end{equation}
such that $\rmat{}:=\rmat{S_Q([a,b']),S_Q([a',b])}$ and
\begin{itemize}
\item $S_Q([a,b]) \conv S_Q([a',b']) \simeq S_Q([a',b']) \conv S_Q([a,b])$ is simple,
\item $S_Q([a',b]) \conv S_Q([a,b'])$ has composition length $2$ with the unique simple socle ${\rm Im}(\rmat{})$ and
the simple head $S_Q([a',b']) \conv S_Q([a,b])$ $($up to
isomorphisms$)$.
\end{itemize}
\end{proposition}

\begin{proof}
Write $\phi^{-1}([a',b])=(j,q)$, $\phi^{-1}([a,b'])=(i,p)$,
$\phi^{-1}([a,b])=(k,r)$ and $\phi^{-1}([a',b'])=(l,s)$. In the
rectangle subquiver, we can read that $$\text{$|k-l|>|r-s|$ and
hence $S_Q([a,b]) \conv S_Q([a',b']) \simeq S_Q([a',b']) \conv
S_Q([a,b])$ is simple.}$$

Now we claim that $\Rnorm_{V_Q([a',b]),V_Q([a,b'])}(z)$ has a pole
at $z=1$; that is, there exists  $1 \le t \le \min \{
i,j,n+1-i,n+1-j \}$ satisfying
\begin{equation} \label{eq: red condition}
|i-j|+2t = 2k-i-j.
\end{equation}

Set $v=\max \{ i, j \}$ and $u =  \min \{ i,j \}$. Then we have $w=
\min \{ u,n+1-v \}$ and
$$|i-j|+2t = 2k-i-j \quad \text{ if } t=k-v$$
Since $k \le i+j-1$, $k-v < w$ always. Thus our claim holds.

By Theorem \ref{Thm: length 2 KKKO}, it is enough to show that there
exists a non-zero $R$-homomorphism
$$\psi: S_Q([a',b]) \conv S_Q([a,b']) \twoheadrightarrow S_Q([a,b]) \conv S_Q([a',b']) \simeq S_Q([a',b']) \conv S_Q([a,b]).$$
Note that the following four cases can happen:
$$ a< a' \le  b'<b , \quad a'< a \le b <b', \quad  a< a' \le b< b' \quad \text{or} \quad  a'< a \le b'<b.$$

Let us consider the first case. From the rectangle quiver, we can
read that $[a',b]$ is located at the right part of the upper ray of
$[a',b]$ and hence $[b'+1,b]$ is located at the left part of the
upper ray. Thus we have an injective $R$-homomorphism
\begin{equation} \label{eq: 6 trem ses 1 step 1}
S_Q([a',b]) \hookrightarrow S_Q([a',b'])\otimes S_Q([b'+1,b]).
\end{equation}
By taking $ - \conv S_Q([a,b'])$ to \eqref{eq: 6 trem ses 1 step 1},
we have a composition
\begin{align*}
  & S_Q([a',b]) \conv S_Q([a,b'])  \longrightarrow S_Q([a',b']) \conv S_Q([b'+1,b])\conv S_Q([a,b']) \\
  & \hspace{16ex} \longrightarrow S_Q([a',b']) \conv S_Q([a,b]),
\end{align*}
since $([b'+1,b],[a,b'])$ is the pair contained in the upper ray of
$[a,b]$. By \cite[Corollary 3.11]{KKKO14}, the composition is
non-zero. From the fact that $S_Q([a',b']) \conv S_Q([a,b])$ is
simple, the composition is surjective and hence $S_Q([a',b']) \conv
S_Q([a,b])$ is a head of $S_Q([a',b]) \conv S_Q([a,b'])$. By
applying the similar argument, one can see that there exists such a
surjective homomorphism for the second case.

Now, we assume the third case; that is, $a< a' \le  b< b'$. Since
$([a,b],[b+1,b'])$ is the pair contained in the lower ray of
$[a,b']$, we have an injective $R$-homomorphism
\begin{equation} \label{eq: 6 trem ses 2 step 1}
S_Q([a,b']) \hookrightarrow S_Q([a,b])\conv S_Q([b+1,b']).
\end{equation}
By taking $S_Q([a',b])\conv -$ to \eqref{eq: 6 trem ses 2 step 1},
we have
\begin{align*}
& S_Q([a',b]) \conv S_Q([a,b'])  \longrightarrow S_Q([a',b]) \conv S_Q([a,b])\conv S_Q([b+1,b'])  \\
& \hspace{7ex}  \simeq   S_Q([a,b]) \conv S_Q([a',b])\conv
S_Q([b+1,b']) \longrightarrow S_Q([a,b]) \conv S_Q([a',b']),
\end{align*}
by Proposition \ref{prop: switching}. In the similar way, one can do
for the fourth case, also.
\end{proof}

\begin{remark}
When $Q$ is equal to $\overset{\gets}{Q}$, we always have the
following inequality in \eqref{eq: rec quiver}:
$$a' <a \le b <b'.$$
Thus the above proposition can be understood as generalizations of
the following statements \cite[Proposition 4.2.3]{KKK13a}:
\begin{enumerate}
\item[{\rm (a)}] For $a' \le a \le b \le b'$, $L(a,b) \conv L(a',b') \simeq L(a',b') \conv L(a,b)$ is simple.
\item[{\rm (b)}] For $a' <a \le b <b'$, there exists an exact sequence
\begin{equation*}
\begin{aligned}
& 0 \to L(a,b) \conv L(a',b')  \to L(a,b') \conv L(a',b) \\
&  \hspace{15ex}  \overset{\rmat{}}{\longrightarrow}  L(a',b) \conv
L(a,b') \to L(a,b) \conv L(a',b')  \to 0,
\end{aligned}
\end{equation*}
where the image of $\rmat{}$ coincides with the head of $L(a,b')
\conv L(a',b)$ and the socle of $L(a',b) \conv L(a,b')$.
\end{enumerate}
\end{remark}

Theorem \ref{Thm: sectional} tells that all positive roots
$\beta_k$'s in a maximal $S$-sectional (resp. $N$-sectional) path
$\rho$ share second (resp. first) component and have the same value
\begin{align*}
& \phi^{-1}_1(\beta_k)-\phi^{-1}_2(\beta_k)=\phi^{-1}_1(\beta_{k'})-\phi^{-1}_2(\beta_{k'}) \\
\text{ (resp. }
&\phi^{-1}_1(\beta_k)+\phi^{-1}_2(\beta_k)=\phi^{-1}_1(\beta_{k'})+\phi^{-1}_2(\beta_{k'})
\text{)}.
\end{align*}

For each  maximal $S$-sectional (resp. $N$-sectional) path $\rho$,
we assign a value
$$\chi_S(\rho)=\phi^{-1}_1(\beta)-\phi^{-1}_2(\beta) \text{ (resp. } \chi_N(\rho)=\phi^{-1}_1(\beta)+\phi^{-1}_2(\beta) \text{)}$$
where $\beta$ is any positive root in $\rho$.

Note that $\chi_S(\rho) \ne \chi_S(\rho')$ (resp. $\chi_N(\rho) \ne
\chi_N(\rho')$) for any distinct pair of maximal $S$-sectional
(resp. $N$-sectional) paths $(\rho,\rho')$.

We set $\rho_1$ the maximal $S$-sectional path such that
$\chi_S(\rho_1)$ is maximum and $\mathbf{i}_1$ the second component
of $\rho_1$. In this way, we define $\rho_k$ the maximal
$S$-sectional path such that $\chi_S(\rho_k)$ is $k$th maximum and
$\mathbf{i}_k$ the second component of $\rho_k$ for $1 \le k \le n$.
Similarly, we set $\rho'_k$ the maximal $N$-sectional path such that
$\chi_N(\rho'_k)$ is $k$th maximum and $\mathbf{j}_k$ the first
component of $\rho'_k$ for $1 \le k \le n$.

In Example \ref{ex: example 1},
$$(\mathbf{i}_1,\mathbf{i}_2,\mathbf{i}_3,\mathbf{i}_4,\mathbf{i}_5)=(1,3,4,5,2) \quad \text{and} \quad
(\mathbf{j}_1,\mathbf{j}_2,\mathbf{j}_3,\mathbf{j}_4,\mathbf{j}_5)=(3,1,2,4,5).$$

\begin{theorem} \label{thm: re-indexing} Let $Q$ be any Dynkin quiver of finite type $A_n$.
Assume that we have a simple $R_{A_n}$-module $M$.
\begin{enumerate}
\item[({\rm a})] There exists a sequence
$$(a_{1;\mathbf{i}_k},a_{2;\mathbf{i}_k},\ldots,a_{\mathbf{i}_k-1;\mathbf{i}_k},a_{\mathbf{i}_k;\mathbf{i}_k})$$
for $1 \le s \le \mathbf{i}_k$, $1 \le k \le n$ and
$a_{s,\mathbf{i}_k} \in \Z_{\ge 0}$ such that
$$ \dct{k=1}{n} \left(\dct{s=1}{\mathbf{i}_k} S_Q\big([s,\mathbf{i}_k]\big)^{\conv a_{s;\mathbf{i}_k}}\right) \twoheadrightarrow M,$$
where $\dct{s=1}{\mathbf{i}_k} S_Q\big([s,\mathbf{i}_k]\big)^{\conv
a_{s;\mathbf{i}_k}}$ is a simple module for every $1\le k \le n.$
\item[({\rm b})] There exists a sequence
$$(a_{\mathbf{j}_k;\mathbf{j}_k},a_{\mathbf{j}_k;\mathbf{j}_k+1},\ldots,a_{\mathbf{j}_k;n-1},a_{\mathbf{j}_k;n})$$
for $\mathbf{j}_k \le s \le n$, $1 \le k \le n$ and
$a_{s,\mathbf{i}_k} \in \Z_{\ge 0}$ such that
$$ \dct{k=1}{n} \left(\dct{s=\mathbf{j}_k}{n} S_Q\big([\mathbf{j}_k,s]\big)^{\conv a_{\mathbf{j}_k;s}}\right) \twoheadrightarrow M,$$
where $\dct{s=\mathbf{j}_k}{n} S_Q\big([\mathbf{j}_k,s]\big)^{\conv
a_{\mathbf{j}_k;s}}$ is a simple module for every $1\le k \le n.$
\end{enumerate}
\end{theorem}

\begin{proof}
With Theorem \ref{Thm: sectional} and Theorem \ref{thm: BkMc}, our
assertions come from the convex total orders $<^U_Q$, $<^L_Q$ and
Proposition \ref{prop: switching}.
\end{proof}

\begin{remark} Fix the Dynkin quiver as $\overset{\gets}{Q}$.
Then we have
$$(\mathbf{j}_1,\mathbf{j}_2,\ldots,\mathbf{j}_{n-1},\mathbf{j}_n)=(n,n-1,\ldots,2,1).$$

\begin{enumerate}
\item[({\rm a})] For a Kashiwara-Nakashima tableau $T$ of shape $\lambda$ of finite type $A_n$ (\cite{KN94}), we can associate $\ell(\lambda)$-tuple
of partitions
$(\mu^{(1)},\ldots,\mu^{(\ell(\lambda)-1)},\mu^{(\ell(\lambda))})$
as in \cite[Section 1]{KP11}. Here $\ell(\eta)$ denotes the length
of a partition $\eta$. Set $L(a,b)=\ko$ for $a>b$. Then
\cite[Theorem 4.8]{KP11} tells that the simple head $\triangle(T)$
of

$ \nabla(T) \seteq \dct{k=1}{\ell(\lambda)} \left(
\dct{s=1}{\ell(\mu^{(\ell(\lambda)+1-k)})} \hspace{-3.5ex}
S_{\overset{\gets}{Q}}\big( [\ell(\lambda)+1-k,\ell(\lambda)+1-k+\mu^{(\ell(\lambda)+1-k)}_s \hspace{-1ex} -1] \big)\right)$ 

\noindent is indeed a simple module over $R^\lambda_{A_n}$, the
cyclotomic quotient of $R_{A_n}$ with the dominant integral weight
$\lambda$. Moreover, the crystal structure on $\{ \triangle(T) \}$
(\cite{LV09}) is isomorphic to the crystal structure of $B(\lambda)$
of $V(\lambda)$ over $U_q(A_n)$ (\cite{Kash91,Kash93a}).
\item[({\rm b})] If we give a total order on $I_0$ as $\mathbf{j}_k > \mathbf{j}_{k+1}$ for $1 \le k \le n-1$, we have
\begin{align}\label{eq: natural order}1<2< \cdots <n-1<n.\end{align}
Then Lemma \ref{lem: Q rev} tells that
$S_{\overset{\gets}{Q}}(\beta)$ coincides with the cuspidal
representation $L_\beta$ in \cite{HMM09,KR09} with respect to the
total order \eqref{eq: natural order} on $I_0$.
\end{enumerate}
\end{remark}

\section{The categories $\mathscr{C}^{(2)}_Q$ over $U_q'(A^{(2)}_{2n-1})$ and $U_q'(A^{(2)}_{2n})$.}
In this subsection, we study the category $\mathscr{C}^{(2)}_Q$ over
$U_q'(A^{(2)}_{2n-1})$ and $U_q'(A^{(2)}_{2n})$ by using the
combinatorial descriptions in Section \ref{sec: Combinatorics} and
the Dorey's type morphisms for $U_q'(A^{(2)}_{2n-1})$ and
$U_q'(A^{(2)}_{2n})$ studied in \cite{Oh14}.

\medskip

Now, we fix $m \in \Z_{\ge 2}$ as $2n$ or $2n-1$, and a Dynkin
quiver $Q$ of finite type $A_m$. For $1 \le i \le m$ and $(-q)^p \in
\ko^\times$, we define
\begin{equation} \label{eq: star}
\begin{aligned}
& \{1,\ldots,n \} \ni i^\star \seteq \begin{cases} i \\
m+1-i  \end{cases} \text{ and } \\
& ((-q)^p)^\star \seteq  \begin{cases} (-q)^p & \text{ if } 1 \le i \le n, \\
(-1)^m(-q)^p& \text{ if } n  < i \le m. \end{cases}
\end{aligned}
\end{equation}

Using \eqref{eq: star},  we can obtain an injective map $^\star$,
from the set $\big\{ V(\va_i)_{(-q)^p} \mid  1 \le i \le m, \ p \in
\Z \big\}$ consisting of fundamental $U'_q(A^{(1)}_{m})$-modules to
the set $\big\{ V(\va_i)_{\pm(-q)^p} \mid  1 \le i \le n, \ p \in \Z
\big\}$ consisting of fundamental $U'_q(A^{(2)}_{m})$-modules, which
is given by (see \cite{Her10})
$$(V(\va_i)_{(-q)^{p}})^\star \seteq V(\va_{i^\star})_{((-q)^{p})^\star}.$$

For each $(i,p)\in \Gamma_Q$ with $\phi(i,p)=(\beta,0) \in \Phi^+_m
\times \{ 0 \}$, we set
\begin{equation} \label{eq: mathsf V}
 \mathsf{V}_Q(\beta) \seteq (V_Q(\beta))^\star = \begin{cases} V(\varpi_i)_{(-q)^p} & \text{ if } 1 \le i \le n, \\
                                          V(\varpi_{m+1-i})_{(-1)^m(-q)^p} & \text{ if } n  < i \le m, \end{cases}
\end{equation}
which is a fundamental module over $U'_q(A^{(2)}_{m})$.

By mimicking Definition \ref{def: Q module category} and using
\eqref{eq: mathsf V}, we can define the category
$\mathscr{C}^{(2)}_Q$ as follows (see also \cite{Her10}):

\begin{definition} \label{def: Q module category 2}
Let $Q$ be a Dynkin quiver of finite type $A_m$ and $U_q'(\g)$ be
the quantum affine algebra of type $A^{(2)}_{m}$ for $m =2n$ or
$2n-1 \in \Z_{\ge 2}$. We define the smallest abelian full
subcategory $\mathscr{C}^{(2)}_Q$ consisting of finite dimensional
integrable $U_q'(A^{(2)}_{m})$-modules such that
\begin{itemize}
\item[({\rm a})] it is stable by taking subquotient, tensor product and extension,
\item[({\rm b})] it contains $\mathsf{V}_Q(\beta)$ for all $\beta \in \Phi_m^+$.
\end{itemize}
Here, $\Phi_m^+$ is the set of positive root of $A_{m}$.
\end{definition}

The goal of this subsection is to show the following theorem:

\begin{theorem} \label{Thm: m 2}
For every positive root $\gamma \in \Phi^+_{m}$ with $|\gamma| \ge
2$, there exists a minimal pair $(\alpha,\beta)$ such that  there
exists a surjective $U'_q(A_{m}^{(2)})$-module homomorphism
$$\mathsf{V}_Q(\beta) \otimes \mathsf{V}_Q(\alpha) \twoheadrightarrow \mathsf{V}_Q(\gamma).$$
\end{theorem}

\medskip

It is known that (\cite[(1.7)]{AK}), for all $p \in \Z$,
\begin{equation} \label{eq: + iso -}
V(\varpi_n)_{(-q)^p} \simeq V(\varpi_n)_{-(-q)^p}
\end{equation}
where $V(\varpi_n)$ is a $U'_q(A_{2n-1}^{(2)})$-module.
\medskip

{\it In this subsection, we prove {\rm Theorem \ref{Thm: m 2}} only
for $m=2n-1$. In the way of proving {\rm Theorem \ref{Thm: m 2}} for
$m=2n-1$, we use \eqref{eq: + iso -} several times. To prove the
case when $m=2n$, one has to use the same arguments, the same pairs
of positive roots in this subsection, and use the morphism
\eqref{eq: p 1nn} instead of \eqref{eq: + iso -}.}
\medskip

Note that the $p^*$ in \eqref{eq: p star} is $(-1)^m(-q)^{m+1}$ for
$\g=A^{(2)}_{m}$ (\cite[Appendix A]{Oh14}).

\begin{lemma} \label{Lem: height 2}
For $\alpha_k+\alpha_{k+1}$ with $1 \le k < 2n-1$, we have a
surjective $U'_q(A_{2n-1}^{(2)})$-module homomorphism
$$ \mathsf{V}_Q(\alpha_k) \otimes \mathsf{V}_Q(\alpha_{k+1}) \twoheadrightarrow \mathsf{V}_Q(\alpha_k \hspace{-0.3ex}+\hspace{-0.3ex}\alpha_{k+1}) \text{ or }
\mathsf{V}_Q(\alpha_{k+1}) \otimes \mathsf{V}_Q(\alpha_{k})
\twoheadrightarrow
\mathsf{V}_Q(\alpha_k\hspace{-0.3ex}+\hspace{-0.3ex}\alpha_{k+1}).$$
\end{lemma}

\begin{proof}
\textbf{(Case 1: \xymatrix@R=3ex{ *{ \bullet }<3pt> \ar@{<-}[r]
&*{\bullet}<3pt> \ar@{->}[r]_<{k}  &*{\bullet}<3pt>
\ar@{<-}[r]_<{k+1} & *{\bullet}<3pt> } ) } In this case, Lemma
\ref{Lem: Key Lem} (a) tells that
$$ \phi^{-1}(\alpha_k)=(k,\xi_k) \quad \text{ and } \quad \phi^{-1}(\alpha_{k+1})=(2n-1-k,\xi_k-(2n-1)).$$
Thus, by Remark \ref{Alg: easy}, we have $\alpha_{k} \prec_Q
\alpha_{k+1}$,
$$ \phi^{-1}(\alpha_k+\alpha_{k+1})= (2n-1,\xi_k-(2n-1-k))$$
and hence
$\mathsf{V}_Q(\alpha_{k}+\alpha_{k+1})=V(\varpi_1)_{-(-q)^{\xi_k-(2n-1-k)}}$.
For each $k$, $ \mathsf{V}_Q(\alpha_k)$ and
$\mathsf{V}_Q(\alpha_{k+1})$ are summarized by the following table:
\begin{center}
\begin{tabular}{ | c | c | c  | } \hline
$k$ & $\mathsf{V}_Q(\alpha_{k+1})$ &  $ \mathsf{V}_Q(\alpha_k)$ \\
\hline $k \le n-1$ & $V(\varpi_{k+1})_{-(-q)^{\xi_k-(2n-1)}}$ &
$V(\varpi_{k})_{(-q)^{\xi_k}}$ \\ \hline $k\ge n$ &
$V(\varpi_{2n-1-k})_{(-q)^{\xi_k-(2n-1)}}$ &
$V(\varpi_{2n-k})_{-(-q)^{\xi_k}}$ \\ \hline
\end{tabular}
\end{center}
For $k \le n-1$, we have an injection $V(\varpi_{k+1})
\rightarrowtail V(\varpi_{1})_{(-q)^{k}} \otimes
V(\varpi_{k})_{(-q)^{-1}}$ by taking dual to \eqref{eq: p ij}.
Tensoring the right dual of $V(\varpi_{k})_{(-q)^{-1}}$, we have
$$ V(\varpi_{k+1}) \otimes V(\varpi_{k})_{-(-q)^{2n-1}} \twoheadrightarrow V(\varpi_{1})_{(-q)^{k}},$$
which yields our assertion. Similarly, we can prove for $k \ge n$.

\textbf{(Case 2: \xymatrix@R=3ex{ *{ \bullet }<3pt> \ar@{<-}[r]
&*{\bullet}<3pt> \ar@{->}[r]_<{k}  &*{\bullet}<3pt>
\ar@{->}[r]_<{k+1} & *{ \bullet }<3pt> } ) } In this case, we have
$$ \phi^{-1}(\alpha_k)=(k,\xi_k), \ \phi^{-1}(\alpha_{k+1})=(1,\xi_k-1-k),\
\phi^{-1}(\alpha_k+\alpha_{k+1})= (k+1,\xi_k-1)$$ and hence
$\mathsf{V}_Q(\alpha_{k+1})=V(\varpi_1)_{(-q)^{\xi_k-1-k}}$. For
each $k$, $ \mathsf{V}_Q(\alpha_k)$ and $\mathsf{V}_Q(\alpha_{k+1})$
are summarized by the following table:
\begin{center}
\begin{tabular}{ | c | c | c  | } \hline
$k$ & $\mathsf{V}_Q(\alpha_{k})$ &  $
\mathsf{V}_Q(\alpha_k+\alpha_{k+1})$ \\ \hline $k \le n-1$ &
$V(\varpi_{k})_{(-q)^{\xi_k}}$ & $V(\varpi_{k+1})_{(-q)^{\xi_k-1}}$
\\ \hline $k \ge n$ & $V(\varpi_{2n-k})_{-(-q)^{\xi_k}}$ &
$V(\varpi_{2n-k-1})_{-(-q)^{\xi_k-1}}$ \\ \hline
\end{tabular}
\end{center}
Using the same technique in \textbf{(Case 1)}, we have
$\mathsf{V}_Q(\alpha_{k+1}) \otimes \mathsf{V}_Q(\alpha_{k}) \twoheadrightarrow \mathsf{V}_Q(\alpha_k+\alpha_{k+1})$.\\
\textbf{(Case 3: \xymatrix@R=3ex{ *{\bullet}<3pt> \ar@{->}[r]
&*{\bullet}<3pt> \ar@{<-}[r]_<{k}  &*{\bullet}<3pt>
\ar@{->}[r]_<{k+1} & *{\bullet}<3pt> } ) } We have
\begin{align*}
& \phi^{-1}(\alpha_k)=(2n-k,\xi_k-2n+2), \ \phi^{-1}(\alpha_{k+1})=(k+1,\xi_k+1), \\
& \phi^{-1}(\alpha_k+\alpha_{k+1})= (1,\xi_k+1-k)
\end{align*}
and hence
$\mathsf{V}_Q(\alpha_k+\alpha_{k+1})=V(\varpi_1)_{(-q)^{\xi_k+1-k}}$.
 For each $k$, $ \mathsf{V}_Q(\alpha_k)$ and $\mathsf{V}_Q(\alpha_{k+1})$ are summarized by the following table:
\begin{center}
\begin{tabular}{ | c | c | c  | } \hline
$k$ & $\mathsf{V}_Q(\alpha_{k})$ &  $ \mathsf{V}_Q(\alpha_{k+1})$ \\
\hline $k \le n-1$ & $V(\varpi_{k})_{-(-q)^{\xi_k-2n+2}}$ &
$V(\varpi_{k+1})_{(-q)^{\xi_k+1}}$ \\ \hline $k \ge n$ &
$V(\varpi_{2n-k})_{(-q)^{\xi_k-2n+2}}$ &
$V(\varpi_{2n-k-1})_{-(-q)^{\xi_k+1}}$ \\ \hline
\end{tabular}
\end{center}
By the same way of the \textbf{(Case 1,2)}, one can prove that
$\mathsf{V}_Q(\alpha_{k}) \otimes \mathsf{V}_Q(\alpha_{k+1})
\twoheadrightarrow \mathsf{V}_Q(\alpha_k+\alpha_{k+1}).$ For the
remaining cases, we have
\begin{center}
\scalebox{0.73}{\begin{tabular}{ | c | c | c | c  | } \hline $Q$ &
$\phi^{-1}(\alpha_{k},0)$ &  $\phi^{-1}(\alpha_{k+1},0)$ &  $
\phi^{-1}(\alpha_{k}+\alpha_{k+1},0)$ \\ \hline $\xymatrix@R=3ex{
*{\bullet}<3pt> \ar@{->}[r]  &*{\bullet}<3pt> \ar@{<-}[r]_<{k}
&*{\bullet}<3pt> \ar@{<-}[r]_<{k+1} & *{ \bullet }<3pt> }$ &
$(2n-k,\xi_k-2n+2)$ & $(2n-1,\xi_{k}-2n+k+3)$ &
$(2n-k+1,\xi_k-2n+3)$
\\  \hline
$\xymatrix@R=3ex{ *{\bullet}<3pt> \ar@{->}[r]  &*{\bullet}<3pt>
\ar@{->}[r]_<{k}  &*{\bullet}<3pt> \ar@{<-}[r]_<{k+1} & *{ \bullet
}<3pt> }$ & $(1,\xi_k-k+1)$ & $(2n-k-1,\xi_{k}-2n+1)$ &
$(2n-k,\xi_{k}-2n+2)$
\\  \hline
$\xymatrix@R=3ex{ *{\bullet}<3pt> \ar@{->}[r]  &*{\bullet}<3pt>
\ar@{->}[r]_<{k}  &*{\bullet}<3pt> \ar@{->}[r]_<{k+1} & *{ \bullet
}<3pt> }$ & $(1,\xi_k-k+1)$ & $(1,\xi_{k}-k-1)$ & $(2,\xi_k-k)$
\\  \hline
$\xymatrix@R=3ex{ *{\bullet}<3pt> \ar@{<-}[r]  &*{\bullet}<3pt>
\ar@{<-}[r]_<{k}  &*{\bullet}<3pt> \ar@{->}[r]_<{k+1} & *{ \bullet
}<3pt> }$ & $(2n-1,\xi_k-2n+1+k)$& $(k+1,\xi_{k}+1)$ & $(k,\xi_{k})$
\\  \hline
$\xymatrix@R=3ex{ *{\bullet}<3pt> \ar@{<-}[r]  &*{\bullet}<3pt>
\ar@{<-}[r]_<{k}  &*{\bullet}<3pt> \ar@{<-}[r]_<{k+1} & *{ \bullet
}<3pt> }$ & $(2n-1,\xi_k-2n+1+k)$ & $(2n-1,\xi_{k}-2n+3+k)$ &
$(2n-2,\xi_{k}-2n+2+k)$
\\  \hline
\end{tabular}}
\end{center}
By using the same technique, one can prove our assertion for the
remaining cases.
\end{proof}

\begin{lemma} \label{Lem: minimal pair 1}
Let $\gamma=[k,\ell] \in \Phi^+_n$ with $|\gamma| \ge 3$. We write
the upper ray and the lower ray of $[k,\ell]$ as follows:
\begin{align*}
 \text{Upper ray}: & \underbrace{S^u_1 \to \cdots S^u_{a_1-1} \to S^u_{a_1}}_{\text{$S$-sectional path}}=[k,\ell]
=\overbrace{N^u_{b_1} \to N^u_{b_1-1} \to  N^u_1}^{\text{$N$-sectional path}} \\
 \text{Lower ray}: &\overbrace{N^l_1 \to \cdots N^l_{b_2-1} \to N^l_{b_2}}^{\text{$N$-sectional path}}=[k,\ell]
=\underbrace{S^l_{a_2} \to S^l_{a_2-1} \to
S^l_1}_{\text{$S$-sectional path}}.
\end{align*}
Then the followings hold:
\begin{enumerate}
\item[({\rm a})] If $a_1=1$ and $b_1 >1$, then all $N^l_j$ $(1 \le j < b_2 )$ are of the form $[k,t]$ with $t < \ell$ and
all $N^u_s$ $(1 \le s < b_1 )$ are of the form $[k,v]$ with $v > l$.
Moreover, $N^l_{b_2-1}=[k,\ell-1]$ and the pair
$([\ell],[k,\ell-1])$ is in the lower ray of $\gamma$.
\item[({\rm b})] If $b_1=1$ and $a_1 >1$, then all $S^l_j$ $(1 \le j < a_2 )$ are of the form $[b,\ell]$ with $k < b$
and all $S^u_s$ $(1 \le s < a_1 )$ are of the form $[a,\ell]$ with
$a < k$. Moreover, $S^l_{b_2-1}=[k+1,\ell]$ and the pair
$([k+1,\ell],[k])$ is  in the lower ray of $\gamma$.
\item[({\rm c})]If $a_2=1$ and $b_2 >1$, then all $N^u_j$ $(1 \le j < b_1 )$ are of the form $[k,t]$ with $t < \ell$ and
all $N^l_s$ $(1 \le s < b_2 )$ are of the form $[k,v]$ with $v > l$.
Moreover, $N^u_{b_1-1} = [k,\ell-1]$ and $([\ell],[k,\ell-1])$ is in
the upper ray of $\gamma$.
\item[({\rm d})] If $b_2=1$ and $a_2 >1$, then all $S^u_j$ $(1 \le j < a_1 )$ are of the form $[b,\ell]$ with $k < b$
and all $S^l_s$ $(1 \le s < a_2 )$ are of the form $[a,\ell]$ with
$a < k$. Moreover, $S^u_{a_1-1}=[k+1,\ell]$ and the pair
$([k+1,\ell],[k])$ is in the upper ray of $\gamma$.
\end{enumerate}
\end{lemma}

\begin{proof} We only give a proof of (a). The remaining cases can be proved in the similar way.
The second assertion of (a) follows from Theorem \ref{Thm:
sectional} and \eqref{eq: paths}. In particular, $N^u_1 \ne [k]$.
Note that $$ \text{$N^l_1 \to \cdots N^l_{b_2-1} \to N^l_{b_2} =
N^u_{b_1} \to N^u_{b_1-1} \to  N^u_1$ is a maximal $N$-sectional
path.}$$ Thus by Proposition \ref{Prop: maximal path}, $N^l_1=[k]$
and hence $k$ is a sink or a right intermediate. On the other hand,
$b_1>1$ implies that $\xymatrix@R=3ex{*{\bullet}<3pt>
\ar@{<-}[r]_<{\ell-1}  &*{\bullet}<3pt> \ar@{<-}[r]_<{\ell}
&*{\bullet}<3pt> \ar@{-}[l]^<{\ \ \ell+1} }$ in $Q$. Hence
$S^l_1=[\ell]$. Thus, by Remark \ref{Alg: easy}, the situation in
$\Gamma_Q$ can be drawn as follows:
$$
\scalebox{0.9}{{\xy (20,20)*{}="T0"; (40,0)*{}="B3";(34,0)*{}="B23";
(5,5)*{}="L3";(10,0)*{}="B1";(20,0)*{}="B2";(25,25)*{}="TT0";
(17,17)*{}="L1"; (10,10)*{}="L2";(0,0)*{}="B0"; "T0"; "B3"
**\dir{-};"T0"; "L3" **\dir{-};"T0"; "TT0" **\dir{-}; "L3"; "B1"
**\dir{-};"L2"; "B2" **\dir{-};"L1"; "B23" **\dir{-}; "B3";
"B3"+(5,5) **\dir{-}; "B0"; "B3"+(10,0) **\dir{.};
"B3"+(15,0)*{\scriptstyle n};
"B1"*{\bullet};"B2"*{\bullet};"B3"*{\bullet};"B23"*{\bullet};
"T0"*{\bullet};"L1"*{\bullet};"L3"*{\bullet};"L2"*{\bullet};
"L1"+(-5,0)*{\scriptstyle N^l_{b_2-1}}; "T0"+(5,0)*{\scriptstyle
[k,\ell]}; "L2"+(-3,0)*{\scriptstyle N^l_j};
"L3"+(-7,0)*{\scriptstyle N^l_1=[k]}; "B3"+(3,-2.5)*{\scriptstyle
[\ell]=\sigma_w}; "B1"+(0,-2.5)*{\scriptstyle \sigma_u};
"B2"+(0,-2.5)*{\scriptstyle \sigma_v};
"B23"+(-1.5,-2.5)*{\scriptstyle \sigma_{w-1}};
\endxy}}
$$
Theorem \ref{Cor: sectional} and Corollary \ref{Cor: first and last}
tell that for any $u< v <w$,
$$ \sigma_u=[k,c], \quad \sigma_u=[e,f] \quad \text{and} \quad k \le c < e \le f < \ell.$$
By applying Theorem \ref{Cor: sectional} once again, we can conclude
that $N^l_j=[k,f]$ for $f<\ell$. Furthermore, we can observe that
$$\text{$\sigma_{w-1}=[g,\ell-1] \ $ for some $ \ g \le \ell-1$ and
hence $ \ N^l_{b_2-1}=[k,\ell-1]$.}$$
\end{proof}

\begin{lemma} \label{Lem: no pair }
Let $\gamma=[k,\ell] \in \Phi^+_n$ with $|\gamma| \ge 3$.
\begin{enumerate}
\item[({\rm a})] If there is no pair of $\gamma$ in the upper ray, there exists a pair $(\alpha,\beta)$ in
the lower ray such that
$$ |\phi^{-1}_1(\alpha)-\phi^{-1}_1(\gamma)| =1 \quad \text{ or } \quad  |\phi^{-1}_1(\beta)-\phi^{-1}_1(\gamma)| =1.$$
\item[({\rm b})] If there is no pair of $\gamma$ in the lower ray, there exists a pair $(\alpha,\beta)$ in
the upper ray such that
$$ |\phi^{-1}_1(\alpha)-\phi^{-1}_1(\gamma)| =1 \quad \text{ or } \quad  |\phi^{-1}_1(\beta)-\phi^{-1}_1(\gamma)| =1.$$
\end{enumerate}
\end{lemma}
\begin{proof} (a) Using the notations in Lemma \ref{Lem: minimal pair 1}, Corollary \ref{Cor: sectional} implies that
$$N_1^l=[k] \quad \text{ and } \quad S_1^l=[\ell] \quad \text{ under the assumption of (a)}.$$
Thus $k$ can not be a source and a right intermediate in $Q$, and
$\ell$ can not be a sink and a right intermediate in $Q$.

\textbf{(Case 1: $k$ is a sink and $\ell$ is a source)} If $k=1$ or
$\ell=n$, then the situation in $\Gamma_Q$ can be drawn as follows:
$$
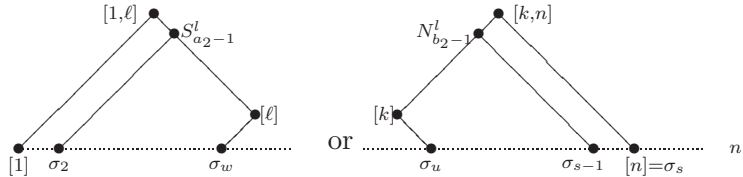
\scalebox{0.9}{{\xy (20,20)*{}="T0"; (35,5)*{}="R2";(0,0)*{}="B1";
(23,17)*{}="R1"; (0,0)*{}="B0";(6,0)*{}="B1"; (30,0)*{}="B2"; "T0";
"R2" **\dir{-};"T0"; "B0" **\dir{-};"R2"; "B2" **\dir{-};"R1"; "B1"
**\dir{-}; "B0"; "B2"+(10,0) **\dir{.};
"B1"*{\bullet};"R2"*{\bullet};"B2"*{\bullet};"B0"*{\bullet};
"T0"*{\bullet};"R1"*{\bullet}; "R1"+(5,0)*{\scriptstyle
S^l_{a_2-1}}; "T0"+(-5,0)*{\scriptstyle [1,\ell]};
"R2"+(2,0)*{\scriptstyle [\ell]}; "B0"+(0,-2.5)*{\scriptstyle [1]};
"B1"+(0,-2.5)*{\scriptstyle \sigma_2}; "B2"+(0,-2.5)*{\scriptstyle
\sigma_w};
\endxy}} \quad
\text{ or } \scalebox{0.9}{{\xy (20,20)*{}="T0";
(40,0)*{}="B3";(34,0)*{}="B23"; (5,5)*{}="L3";(10,0)*{}="B1";
(17,17)*{}="L1"; (0,0)*{}="B0"; "T0"; "B3" **\dir{-};"T0"; "L3"
**\dir{-}; "L3"; "B1" **\dir{-};"L1"; "B23" **\dir{-}; "B0";
"B3"+(10,0) **\dir{.}; "B3"+(15,0)*{\scriptstyle n};
"B1"*{\bullet};"B3"*{\bullet};"B23"*{\bullet};
"T0"*{\bullet};"L1"*{\bullet};"L3"*{\bullet};
"L1"+(-5,0)*{\scriptstyle N^l_{b_2-1}}; "T0"+(5,0)*{\scriptstyle
[k,n]}; "L3"+(-2,0)*{\scriptstyle [k]}; "B3"+(3,-2.5)*{\scriptstyle
[n]=\sigma_s}; "B1"+(0,-2.5)*{\scriptstyle \sigma_u};
"B23"+(-1.5,-2.5)*{\scriptstyle \sigma_{s-1}};
\endxy}}
$$

For each case, $\sigma_2=[2,a]$ and $\sigma_{s-1}=[b,n-1]$,
respectively. Thus $(S^l_{a_2-1}=[2,\ell],[1])$ and
$([n],N^l_{b_2-1}=[k,n-1])$ are desired ones, respectively.

If $k \ne 1$ and $\ell \ne n$, then Remark \ref{Alg: easy} and
Corollary \ref{Cor: first and last} imply that the situation in
$\Gamma_Q$ can be depicted as follows:
$$\scalebox{0.9}
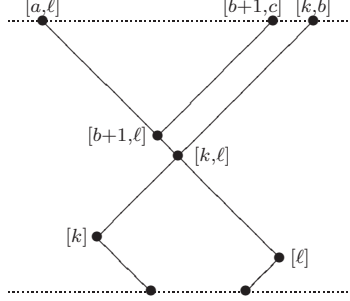
{{\xy
(0,40)*{}="T0";(34,40)*{}="T11";(40,40)*{}="T1"; (20,20)*{}="M0";
(35,5)*{}="MR1";(8,8)*{}="ML1"; (30,0)*{}="B1";(16,0)*{}="B0";
"M0"+(-3,3); "T11" **\dir{-}; "M0"; "MR1" **\dir{-};"M0"; "T1"
**\dir{-}; "T0"+(-5,0); "T1"+(5,0) **\dir{.}; (-5,0); (45,0)
**\dir{.}; "M0"; "ML1" **\dir{-};"M0"; "T0" **\dir{-}; "ML1"; "B0"
**\dir{-};"MR1"; "B1" **\dir{-};
"T0"*{\bullet};"T1"*{\bullet};"M0"*{\bullet};"T11"*{\bullet};
"MR1"*{\bullet};"ML1"*{\bullet}; "B0"*{\bullet};"B1"*{\bullet};
"M0"+(-3,3)*{\bullet}; "M0"+(5,0)*{\scriptstyle [k,\ell]};
"MR1"+(3,0)*{\scriptstyle [\ell]};"ML1"+(-3,0)*{\scriptstyle [k]};
"T0"+(0,2)*{\scriptstyle [a,\ell]};"T1"+(0,2)*{\scriptstyle [k,b]};
"T11"+(-3,2)*{\scriptstyle [b+1,c]}; "M0"+(-9,3)*{\scriptstyle
[b+1,\ell]};
\endxy}}
$$
for some $k < b < c < a < \ell$. Thus $([k,b],[b+1,\ell])$ is a pair in the upper ray of $\gamma$ which yields a contradiction to the assumption. \\
\textbf{(Case 2: $k$ is a not sink or $\ell$ is not a source)}
Equivalently,
$$\xymatrix@R=3ex{*{\bullet}<3pt>
\ar@{<-}[r]_<{k-1}  &*{\bullet}<3pt> \ar@{<-}[r]_<{k}
&*{\bullet}<3pt> \ar@{-}[l]^<{\ \ k+1} } \text{ or }
\xymatrix@R=3ex{*{\bullet}<3pt> \ar@{<-}[r]_<{\ell-1}
&*{\bullet}<3pt> \ar@{<-}[r]_<{\ell} &*{\bullet}<3pt> \ar@{-}[l]^<{\
\ \ell+1} } \text{ in $Q$}.$$ Thus, using the notations in Lemma
\ref{Lem: minimal pair 1}, the situation in $\Gamma_Q$ can be drawn
as follows:
$$
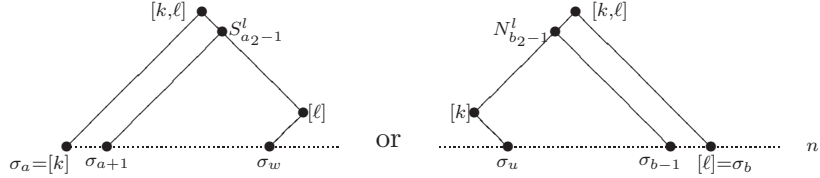
\scalebox{0.9}{{\xy (20,20)*{}="T0"; (35,5)*{}="R2";(0,0)*{}="B1";
(23,17)*{}="R1"; (0,0)*{}="B0";(6,0)*{}="B1"; (30,0)*{}="B2"; "T0";
"R2" **\dir{-};"T0"; "B0" **\dir{-};"R2"; "B2" **\dir{-};"R1"; "B1"
**\dir{-}; "B0"; "B2"+(10,0) **\dir{.};
"B1"*{\bullet};"R2"*{\bullet};"B2"*{\bullet};"B0"*{\bullet};
"T0"*{\bullet};"R1"*{\bullet}; "R1"+(5,0)*{\scriptstyle
S^l_{a_2-1}}; "T0"+(-5,0)*{\scriptstyle [k,\ell]};
"R2"+(2,0)*{\scriptstyle [\ell]}; "B0"+(-4,-2.5)*{\scriptstyle
\sigma_a=[k]}; "B1"+(0,-2.5)*{\scriptstyle \sigma_{a+1}};
"B2"+(0,-2.5)*{\scriptstyle \sigma_w};
\endxy}} \quad
\text{ or } \quad \scalebox{0.9}{{\xy (20,20)*{}="T0";
(40,0)*{}="B3";(34,0)*{}="B23"; (5,5)*{}="L3";(10,0)*{}="B1";
(17,17)*{}="L1"; (0,0)*{}="B0"; "T0"; "B3" **\dir{-};"T0"; "L3"
**\dir{-}; "L3"; "B1" **\dir{-};"L1"; "B23" **\dir{-}; "B0";
"B3"+(10,0) **\dir{.}; "B3"+(15,0)*{\scriptstyle n};
"B1"*{\bullet};"B3"*{\bullet};"B23"*{\bullet};
"T0"*{\bullet};"L1"*{\bullet};"L3"*{\bullet};
"L1"+(-5,0)*{\scriptstyle N^l_{b_2-1}}; "T0"+(5,0)*{\scriptstyle
[k,\ell]}; "L3"+(-2,0)*{\scriptstyle [k]};
"B3"+(2,-2.5)*{\scriptstyle [\ell]=\sigma_{b}};
"B1"+(0,-2.5)*{\scriptstyle \sigma_{u}};
"B23"+(-1.5,-2.5)*{\scriptstyle \sigma_{b-1}};
\endxy}}
$$
Then $(S^l_{a_2-1},\sigma_{a+1})=([k+1,l],[k+1,c])$ and
$(N^l_{b_2-1},\sigma_{b-1})=([k,\ell-1],[d,\ell-1])$, respectively.
Hence $([k+1,\ell],[k])$ and $([\ell],[k,\ell-1])$ are desired one,
respectively. The proof for $(b)$ can be obtained by applying the
similar argument.
\end{proof}

\bigskip
\noindent \textbf{\emph{Proof of Theorem \ref{Thm: m 2}.}} For
$\gamma \in \Phi^+_{2n-1}$ with $|\gamma|=2$, it is already proved
in Lemma \ref{Lem: height 2}. Thus it suffices to consider when
$|\gamma|=|[k,\ell]| \ge 3$. To prove this, we need to observe the
situation in $\Gamma_Q$ for each case.

\textbf{(Case 1: $1 <\phi^{-1}_1(\gamma)=i_\gamma\le n$ )} (a) If
there is a pair $(\alpha,\beta)$ in the upper ray, then Theorem
\ref{Thm: i_k+i_l=i_j} tells that
$$ (\phi^{-1}_1(\alpha))^\star=\phi^{-1}_1(\alpha), \quad (\phi^{-1}_1(\beta))^\star = \phi^{-1}_1(\beta)
\quad \text{and} \quad
\phi^{-1}_1(\alpha)+\phi^{-1}_1(\beta)=\phi^{-1}_1(\gamma).$$ Then
we have a surjective homomorphism
$$ \mathsf{V}_Q(\beta) \otimes  \mathsf{V}_Q(\alpha) \twoheadrightarrow \mathsf{V}_Q(\gamma) $$
by \eqref{eq: p ij}.

(b) Now we deal with the case when there is no pair in the upper
ray. By Lemma \ref{Lem: minimal pair 1} and Lemma \ref{Lem: no pair
}, there exists a minimal pair $(\alpha,\beta)$ such that
$$ (\phi^{-1}_1(\alpha),\phi^{-1}_1(\beta)) =(2n-1,i_\gamma+1) \text{ or } (i_\gamma+1,2n-1).$$
For the case when $(\phi^{-1}_1(\alpha),\phi^{-1}_1(\beta))
=(2n-1,i_\gamma+1) $, it suffices to show that there exists a
surjective homomorphism
\begin{equation} \label{eq: goal 1}
V(\varpi_{(i_\gamma+1)^\star})_{((-q)^{p_\gamma-1})^\star} \otimes
V(\varpi_{1})_{-(-q)^{p_\gamma+2n-1-i_\gamma}} \twoheadrightarrow
V(\varpi_{i_\gamma})_{(-q)^{p_\gamma}},
\end{equation}
where $p_\gamma=\phi^{-1}_2(\gamma)$. \\
\noindent (b-i) If $i_\gamma<n$, then
$(i_\gamma+1)^\star=i_\gamma+1$ and
$((-q)^{p_\gamma-1})^\star=(-q)^{p_\gamma-1}$. By taking dual to
\eqref{eq: p ij}, we have
\begin{equation} \label{eq: goal 1-1}
 V(\varpi_{i_\gamma+1}) \rightarrowtail  V(\varpi_{i_\gamma})_{(-q)^1} \otimes V(\varpi_{1})_{(-q)^{-i_\gamma}}.
\end{equation}
By taking the left dual of $V(\varpi_{1})_{(-q)^{-i_\gamma}}$ to
\eqref{eq: goal 1-1}, we have
$$ V(\varpi_{i_\gamma+1}) \otimes V(\varpi_{1})_{-(-q)^{2n-i_\gamma}} \twoheadrightarrow  V(\varpi_{i_\gamma})_{(-q)^1}.$$
Thus we have \eqref{eq: goal 1} for $i_\gamma<n$.

\noindent (b-ii) If $i_\gamma=n$, then $(i_\gamma+1)^\star=n-1$ and
$((-q)^{p_\gamma-1})^\star=-(-q)^{p_\gamma-1}$. Then \eqref{eq: goal
1} becomes
$$V(\varpi_{n-1})_{-(-q)^{p_\gamma-1}} \otimes V(\varpi_{1})_{-(-q)^{p_\gamma+n-1}}
\twoheadrightarrow V(\varpi_{n})_{(-q)^{p_\gamma}} \simeq
V(\varpi_{n})_{-(-q)^{p_\gamma}},$$ by \eqref{eq: + iso -}. The
above equation follows from \eqref{eq: p ij} directly.

The case when $(\phi^{-1}_1(\alpha),\phi^{-1}_1(\beta))
=(i_\gamma+1,2n-1) $ can be proved in the same way.

\textbf{(Case 2: $\phi^{-1}_1(\beta_\gamma)=1$ )} Note that there is
no pair in the upper ray. Thus using the notations in Lemma
\ref{Lem: minimal pair 1}, we have
\begin{itemize}
\item $N^l_1=[k]$ and $S^l_1=[\ell]$,
\item $[k]=N^l_1 \to \cdots \to N^l_{b_2-1} \to N^l_{b_2}$ is a maximal $N$-sectional path,
\item $S^l_{a_2} \to S^l_{a_2-1} \to \cdots \to  S^l_1=[\ell]$ is a maximal $S$-sectional path.
\end{itemize}
Thus the situation in $\Gamma_Q$ can be drawn as follows:
$$\scalebox{0.9}
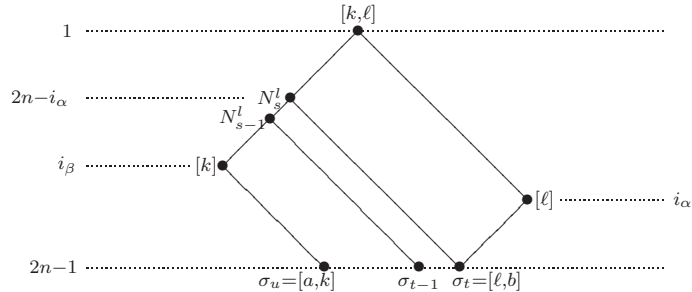
{{\xy
(20,40)*{}="T0"; (0,20)*{}="L2";(7,27)*{}="L1";(10,30)*{}="L01";
(45,15)*{}="R1"; (15,5)*{}="B1"; (35,5)*{}="B2";(29,5)*{}="B21";
"T0"; "L2" **\dir{-};"T0"; "R1" **\dir{-}; "L2"; "B1"
**\dir{-};"R1"; "B2" **\dir{-}; "L01"; "B2" **\dir{-}; (-20,40);
(65,40) **\dir{.}; (-20,20); "L2"+(-5,0) **\dir{.}; "R1"+(5,0);
"R1"+(20,0) **\dir{.}; "R1"+(23,0)*{\scriptstyle i_\alpha}; (-20,5);
(65,5) **\dir{.}; (-20,30); "L01"+(-7,0) **\dir{.};
(-27,30)*{\scriptstyle 2n-i_\alpha}; (-23,20)*{\scriptstyle
i_\beta}; (-23,40)*{\scriptstyle 1}; (-25,5)*{\scriptstyle 2n-1};
"T0"*{\bullet};"L2"*{\bullet};"L1"*{\bullet};"R1"*{\bullet};"B21"*{\bullet};"B2"*{\bullet};"B1"*{\bullet};"L01"*{\bullet};
"L1"; "B21" **\dir{-}; "L2"+(-2.5,0)*{\scriptstyle [k]};
"L1"+(-4,0)*{\scriptstyle N^l_{s-1}}; "L01"+(-3,0)*{\scriptstyle
N^l_s}; "R1"+(2.5,0)*{\scriptstyle [\ell]};
"T0"+(0,2.5)*{\scriptstyle [k,\ell]}; "B21"+(0,-2.5)*{\scriptstyle
\sigma_{t-1}}; "B2"+(4,-2)*{\scriptstyle \sigma_t=[\ell,b]};
"B1"+(-4,-2)*{\scriptstyle \sigma_u=[a,k]};
\endxy}}
$$
for some $a < k$ and $\ell<b$. By Corollary \ref{Cor: first and
last}, $\sigma_{t-1}=[c,\ell-1]$ $(c \le \ell-1)$ and hence
$N^l_{s-1}=[k,\ell-1]$. We fix a pair $(\alpha,\beta)$ as
$([\ell],[k,\ell-1])$. Then we can see that
$$\phi^{-1}_1(\beta)=2n+1-i_\alpha \text{ and } ( 2n-\phi^{-1}_1(\beta) )+( 2n-\phi^{-1}_1(\alpha) )=( 2n-\phi^{-1}_1(\gamma) )=2n-1.$$
(a) Assume that $i_\alpha=\phi^{-1}_1(\alpha) \ge n+1$, then
$\phi^{-1}_1(\beta) \le n$. Thus it suffices to show that there
exists a surjective homomorphism
\begin{equation} \label{eq: goal 2}
V(\varpi_{2n-i_\alpha+1})_{(-q)^{p_\gamma-2n+i_\alpha}} \otimes
V(\varpi_{2n-i_\alpha})_{-(-q)^{p_\gamma+i_\alpha-1}}
\twoheadrightarrow V(\varpi_{1})_{(-q)^{p_\gamma}}.
\end{equation}
By taking dual to \eqref{eq: p ij}, we have
\begin{equation} \label{eq: goal 2-1}
V(\varpi_{2n-i_\alpha+1}) \rightarrowtail
V(\varpi_{1})_{(-q)^{2n-i_\alpha}} \otimes
V(\varpi_{2n-i_\alpha})_{(-q)^{-1}}.
\end{equation}
By taking left dual of $V(\varpi_{2n-i_\alpha})_{(-q)^{-1}}$ to
\eqref{eq: goal 2-1}, we have
$$ V(\varpi_{2n-i_\alpha+1}) \otimes V(\varpi_{2n-i_\alpha})_{-(-q)^{2n-1}} \twoheadrightarrow V(\varpi_{1})_{(-q)^{2n-i_\alpha}}.$$
Thus we have \eqref{eq: goal 2} for $i_\alpha \ge n+1$.

\noindent (b) Assume that $i_\alpha \le n$, then $\phi^{-1}_1(\beta)
\ge n+1$. Thus it suffices to show that there exists a surjective
homomorphism
\begin{equation} \label{eq: goal 3}
V(\varpi_{i_\alpha-1})_{-(-q)^{p_\gamma-2n+i_\alpha}} \otimes
V(\varpi_{i_\alpha})_{(-q)^{p_\gamma+i_\alpha-1}} \twoheadrightarrow
V(\varpi_{1})_{(-q)^{p_\gamma}}.
\end{equation}
By taking dual to \eqref{eq: p ij}, we have
$$
V(\varpi_{i_\alpha}) \rightarrowtail
V(\varpi_{i_\alpha-1})_{(-q)^{1}} \otimes
V(\varpi_{1})_{(-q)^{-i_\alpha+1}}.$$ By taking right dual of
$V(\varpi_{i_\alpha-1})_{(-q)^{1}} $, we have
$$ V(\varpi_{i_\alpha-1})_{-(-q)^{-2n+1}} \otimes V(\varpi_{i_\alpha}) \twoheadrightarrow V(\varpi_{1})_{(-q)^{-i_\alpha+1}}.$$
Thus we have \eqref{eq: goal 3} for $i_\alpha \le n$.

The remaining case \textbf{(Case 3: $n
<\phi^{-1}_1(\beta_\gamma)=i_\gamma\le 2n-1$)} can be proved by
applying the similar arguments in \textbf{(Case 1)} and
\textbf{(Case 2)}.

\begin{corollary}\label{Cor: Dorey on Gamma 2}
 For every positive root $\beta \in \Phi^+_{m}$, there exists a sequence $(i_1,i_2,\cdots,i_{|\beta|}) \in I_0^{|\beta|}$ such that
$$\stens_{t=1}^{|\beta|} \mathsf{V}_Q(\alpha_{i_t})  \twoheadrightarrow \mathsf{V}_Q(\beta).$$
\end{corollary}

From Corollary \ref{Cor: Dorey on Gamma 2}, the condition (b) in
Definition \ref{def: Q module category 2} can be also restated as
follows:
\begin{itemize}
\item[({\rm b}$'$)] It contains $\mathsf{V}_Q(\alpha_k)$ for all $\alpha_k \in \Pi_m$.
\end{itemize}

\bibliographystyle{amsplain}

\begin{thebibliography}{99}

\bibitem{AK} T. Akasaka and M. Kashiwara,
{\it Finite-dimensional representations of quantum affine algebras},
Publ. RIMS. Kyoto Univ., {\bf33} (1997), 839-867.

\bibitem{Ar96} S. Ariki, {\em On the decomposition numbers of the
Hecke algebra of $G(M,1,n)$}, J. Math. Kyoto Univ. {\bf36} (1996),
789-808.

\bibitem{ARS} M. Auslander, I. Reiten and S. Smalo, {\it Representation theory of Artin algebras}, Cambridge studies in advanced
mathematics {\bf 36}, Cambridge 1995.

\bibitem{ASS} I. Assem, D. Simson and A. Skowro\'{n}ski, {\it Elements of the representation theory of associative algebras. Vol.1}, London
Math. Soc. Student Texts {\bf 65}, Cambridge 2006.

\bibitem{B99}
R. Bedard, {\it On commutation classes of reduced words in Weyl
groups}, European J. Combin. {\bf 20} (1999), 483--505.

\bibitem{BKOP} G. Benkart, S.-J. Kang, S.-j. Oh, and E. Park,
\emph{Construction of Irreducible representations over
Khovanov-Lauda-Rouquier algebras of finite classical type}, Int.
Math. Res. Not. {\bf 5} (2014) 1312--1366.

\bibitem{BFZ96}
A. Berenstein, S. Fomin and A. Zelevinsky, {\it Parameterizations of
canonical bases and totally positive matrices}, Adv. Math. {\bf 122}
(1996), 49--149.

\bibitem{Bour}
N.~Bourbaki.
\newblock {\em \'{E}l\'ements de math\'ematique. {F}asc. {XXXIV}. {G}roupes et
  alg\`ebres de {L}ie. {C}hapitres {IV}--{VI}}.
\newblock Actualit\'es Scientifiques et Industrielles, No. 1337. Hermann,
  Paris, 1968.

\bibitem{Br}
S. Brenner, \emph{A combinatorial characterisation of finite
Auslander-Reiten quivers. In: Representation Theory I. Finite
Dimensional Algebras} (ed. V. Dlab, P. Gabriel, G. Michler).
Springer LNM 1177 (1986), 13-49.

\bibitem{BK09A}J. Brundan and A. Kleshchev, {\em Blocks of cyclotomic Hecke algebras
and Khovanov-Lauda algebras}, Invent. Math.,  {\bf178} (2009)
451--484.

\bibitem{BK09B}
\bysame, \emph{Graded decomposition numbers for cyclotomic hecke
algebras},
  Adv. Math. \textbf{222} (2009), no.~6, 1883--1942.

\bibitem{BKM12}
J. Brundan, A. Kleshchev and P. J. McNamara, {\it Homological
properties of finite Khovanov-Lauda-Rouquier algebras}, Duke Math.
J., {\bf 163} (2014), 1353--1404.

\bibitem{BMMRT}
A. Buan, R. Marsh, M. Reineke, I. Reiten and G. Todorov, {\it
Tilting theory and cluster combinatorics}, Adv. Math. \textbf{204}
(2006), no,~2, 572--618.

\bibitem{CP96} V. Chari and A. Pressley, {\it Yangians, integrable quantum systems and Dorey's rule}, Comm. Math. Phys. {\bf 181} (1996),
no. 2, 265-302.

\bibitem{Che}  I. V. Cherednik,
{\em A new interpretation of Gelfand-Tzetlin bases}, Duke Math. J.,
{\bf54} (1987), 563--577.


\bibitem{DO94}
E. Date and M. Okado, {\it Calculation of excitation spectra of the
spin model related with the vector representation of the quantized
affine algebra of type $A^{(1)}_n$}, Internat. J. Modern Phys. A
{\bf 9} (3) (1994), 399--417.

\bibitem{FM01} E. Frenkel  and  E. Mukhin, {\it Combinatorics of $q$-characters of finite-dimensional
representations of quantum affine algebras}, Commun. Math. phys.
{\bf216} (2001), 23--57.

\bibitem{Gab80} P. Gabriel, {\it Auslander-Reiten sequences and Representation-finite algebras}, Lecture notes in
Math., vol. 831, Springer-Verlag, Berlin and New York, 1980,
pp.1.71.

\bibitem{GRV94} V. Ginzburg , N. Reshetikhin and E. Vasserot,
{\em Quantum groups and flag varieties}, A.M.S. Contemp. Math. {\bf
175} (1994), 101-130.

\bibitem{Hap}
D. Happel, {\it Triangulated categories in the representation theory
of finite-dimensional algebras}, London Math. Soc. Lecture Notes
119, Cambridge 1988.

\bibitem{Her10}
D. Hernandez, {\it Kirillov-Reshetikhin conjecture: the general
case}, Int. Math. Res. Not. 2010 (1) (2010), 149--193.

\bibitem{HL11} D. Hernandez and B. Leclerc, {\it Quantum Grothendieck rings and derived Hall algebras},
arXiv:1109.0862v2 [math.QA], to appear in J. Reine Angew. Math.

\bibitem{HMM09}
D.~Hill, G.~Melvin and D.~Mondragon, \emph{Representations of quiver
{H}ecke
  algebras via {L}yndon bases}, J. Pure Appl. Algebra {\bf 216} (2012), no. 5, 1052--1079.

\bibitem{HM10}
J.~Hu and A.~Mathas, \emph{Graded cellular bases for the cyclotomic
  {K}hovanov-{L}auda-{R}ouquier algebras of type ${A}$}, Adv. Math.
  \textbf{225} (2010), no.~2, 598--642.

\bibitem{Kac} V. Kac,
\newblock{\it Infinite dimensional Lie algebras},
\newblock{3rd ed., Cambridge University Press, Cambridge, 1990}.

\bibitem{KK12} S.-J. Kang and M. Kashiwara,
{\em Categorification of highest weight modules via
Khovanov-Lauda-Rouquier algebras}, Invent. Math. {\bf190} (2012),
699--742.

\bibitem{KKK13a}
S.-J. Kang, M. Kashiwara and M. Kim, {\it Symmetric quiver Hecke
algebras and R-matrices of quantum affine algebras}, arXiv:1304.0323
[math.RT].

\bibitem{KKK13b}
\bysame, {\it Symmetric quiver Hecke algebras and R-matrices of
quantum affine algebras II}, arXiv:1308.0651 [math.RT],  to appear
in Duke Math. J.

\bibitem{KKKO14} S.-J. Kang, M. Kashiwara,  M. Kim and S.-j. Oh,
\newblock{\it Simplicity of heads and socles of tensor products}, Compos. Math. {\bf 151} (2015), no. 2, 377--396.


\bibitem{KKKO14b} \bysame,
\newblock{\it Symmetric quiver Hecke algebras and R-matrices of
quantum affine algebras IV}, arXiv:1502.07415 [math.RT].

\bibitem{KKO13}
S.-J. Kang, M. Kashiwara and S.-j. Oh, {\it Supercategorification of
quantum Kac-Moody Algebras}, Adv. Math. {\bf 242} (2013) 116--162.

\bibitem{KP11} S.-J. Kang and E. Park,
{\em Irreducible modules over Khovanov-Lauda-Rouquier algebras of
type $A_n$ and semistandard tableaux}, J. Algebra {\bf339} (2011),
223--251.

\bibitem{Kash91}
M.~Kashiwara, \emph{On crystal bases of the $q$-analogue of
universal
  enveloping algebras}, Duke Math. J. \textbf{63} (1991), no.~2, 465--516.

\bibitem{Kash93a}
\bysame, \emph{The crystal base and {L}ittelmann's refined
{D}emazure character
  formula}, Duke Math. J. \textbf{71} (1993), no.~3, 839--858.

\bibitem{Kas02}
\bysame, {\it On level zero representations of quantum affine
algebras}, Duke. Math. J. {\bf112} (2002), 117--175.

\bibitem{KN94}
M.~Kashiwara and T.~Nakashima, \emph{Crystal graphs for
representations of the $q$-analogue of classical {L}ie algebras}, J.
Algebra \textbf{165} (1994),  no.~2, 295--345.

\bibitem{Kato12} S. Kato,
{\it Poincar\'e-Birkhoff-Witt bases and Khovanov-Lauda-Rouquier
algebras}, Duke Math. J. \textbf{163}, 3 (2014), 619--663.

\bibitem{KL09}
M.~Khovanov and A. D.~Lauda, \emph{A diagrammatic approach to
categorification of quantum groups
  {I}}, Represent. Theory \textbf{13} (2009), 309--347.

\bibitem{KL11}
\bysame, \emph{A diagrammatic approach to categorification of
quantum groups
  {II}}, Trans. Amer. Math. Soc. \textbf{363} (2011), no.~5, 2685--2700.

\bibitem{KR09}
A.~Kleshchev and A.~Ram, \emph{Representations of
{K}hovanov-{L}auda-{R}ouquier
  algebras and combinatorics of {L}yndon words}, Math. Ann. \textbf{349} (2011), no.~4, 943--975.


\bibitem{LLT} A. Lascoux, B. Leclerc and J.-Y. Thibon, {\em Hecke
algebras at roots of unity and crystal bases of quantum affine
algebras}, Comm. Math. Phys. {\bf181} (1996), 205-263.

\bibitem{LV09}
A.~Lauda and M.~Vazirani, \emph{Crystals from categorified quantum
groups}, Adv. Math. {\bf 228} (2011), no. 2, 803--861.

\bibitem{Mc12}
P. McNamara, {\it Finite dimensional representations of
Khovanov-Lauda-Rouquier algebras I: finite type}, to appear in J.
Reine Angew. Math.; arXiv:1207.5860.

\bibitem{Oh14}
S-j. Oh, {\it The Denominators of normalized R-matrices of types
$A_{2n-1}^{(2)}$, $A_{2n}^{(2)}$, $B_{n}^{(1)}$ and
$D_{n+1}^{(2)}$}, arXiv:1404.6715 [math.QA].

\bibitem{Oh14D}
\bysame, {\it Auslander-Reiten quiver of type D and generalized
quantum affine Schur-Weyl duality}, arXiv:1406.4555 [math.RT].

\bibitem{Papi94}
P. Papi, {\it A characterization of a special ordering in a root
system}, Proc. Amer. Math. {\bf120}
 (1994), 661--665.

\bibitem{R96}
C. Ringel, {\it PBW-bases of quantum groups}, J. Reine Angew. Math.
{\bf 470} (1996), pp. 51--88.

\bibitem{R08}
R.~Rouquier, \emph{2 {K}ac-{M}oody algebras}, arXiv:0812.5023
(2008).

\bibitem{VV09}
M. Varagnolo and E. Vasserot,
 \emph{Canonical bases and KLR algebras}, J. reine angew. Math. \textbf{659} (2011), 67--100.


\end{thebibliography}


\end{document}